\documentclass[10pt,a4paper]{article}
\title{On $p$-refined Friedberg--Jacquet integrals and the classical symplectic locus in the $\GL_{2n}$ eigenvariety}
\author{Daniel Barrera Salazar, Andrew Graham and Chris Williams}
\date{}

\usepackage[margin=1.23in]{geometry}

\usepackage{wrapfig}

\newcommand{\s}{\setlength{\itemsep}{0pt}}

\usepackage{multicol}
\usepackage{amssymb}

\usepackage{fancyhdr,xcolor}
\newcommand{\UPS}{\theta} 
\newcommand{\Pmin}{P_{\tilde\pi}}
\newcommand{\cPmin}{\cP_{\tilde\pi}}

\newcommand{\alphaG}{\alpha^{\cG,\cP}}
\newcommand{\alphaGprime}{\alpha^{\cG,\cP'}}

\newcommand{\tbyt}[4]{\left( \begin{array}{cc} #1 & #2 \\ #3 & #4 \end{array} \right)}

\newcommand{\Addresses}{{
  \bigskip
  \footnotesize

(Barrera Salazar) \textsc{Universidad de Santiago de Chile, Avenida Libertador Bernardo O'Higgins no. 3363, Estaci\'{o}n Central, Santiago, Chile}\par\nopagebreak
  \textit{E-mail address}: \texttt{daniel.barrera.s@usach.cl}

  \medskip 

  (Graham) \textsc{University of Oxford, Andrew Wiles Building, Woodstock Rd, Oxford OX2 6GG}\par\nopagebreak
  \textit{E-mail address}: \texttt{andrew.graham@maths.ox.ac.uk}

  \medskip

  (Williams) \textsc{University of Nottingham, University Park, Nottingham, NG7 2RD, United Kingdom}\par\nopagebreak
  \textit{E-mail address}: \texttt{chris.williams1@nottingham.ac.uk}

}}

\makeatletter 
\def\input@path{{../}} 
\makeatother
\usepackage{preamble-BL}
\usepackage{comment}

\newcommand{\sar}[2]{\ar@{}[#1]|-*[@]{#2}}

\usepackage{tikz}
\usetikzlibrary{positioning}
\usepackage{mathtools}
\usepackage{environ}
\NewEnviron{myquote}{\vspace{1ex}\par
	\hfill\llap{($\dagger$)}\hfill\parbox{\textwidth-2cm}%
	{\emph{\BODY}}%
	\vspace{1ex}\par}

\NewEnviron{myquote2}{\vspace{0ex}\par
	\hspace{5pt}\parbox{\textwidth-2cm}%
	{\centering\BODY} \hfill\llap{($*$)}\hfill%
	\vspace{1ex}\par}

\usepackage{tocloft}

\cftpagenumbersoff{part}

\newcommand{\vol}{\operatorname{vol}}

\newcommand{\GG}{\mathbf{G}}

\titleformat{\part}
{\large\scshape\centering}
{Part \thepart.}
{0.5em}
{}
[]

\begin{document}

	\maketitle

	\renewcommand{\thefootnote}{\fnsymbol{footnote}} 
	\footnotetext{\today. \emph{2020 MSC:} Primary 11F33,  11F67; Secondary 11R23, 11G22
	}     
	\renewcommand{\thefootnote}{\arabic{footnote}}

	\begin{abstract}

  Friedberg--Jacquet proved that if $\pi$ is a cuspidal automorphic representation of $\mathrm{GL}_{2n}(\A)$, then $\pi$ is a functorial transfer from $\mathrm{GSpin}_{2n+1}$ if and only if a global zeta integral $Z_H$ over $H = \mathrm{GL}_n \times \mathrm{GL}_n$ is non-vanishing on $\pi$. We conjecture a $p$-refined analogue: that any $P$-parahoric $p$-refinement $\tilde\pi^P$ is a functorial transfer from $\mathrm{GSpin}_{2n+1}$ if and only if a $P$-twisted version of $Z_H$ is non-vanishing on the $\tilde\pi^P$-eigenspace in $\pi$. This twisted $Z_H$ appears in all constructions of $p$-adic $L$-functions via Shalika models. We connect our conjecture to the study of classical symplectic families in the $\mathrm{GL}_{2n}$ eigenvariety, and -- by proving upper bounds on the dimensions of such families -- obtain various results towards the conjecture.
\end{abstract}

	\setcounter{tocdepth}{1}
	\footnotesize
	\tableofcontents
	\normalsize

	\section{Introduction}

Let $\GG$ be a reductive group. In this paper, we prove new connections between two areas in the study of automorphic forms for $\GG$:
\begin{itemize}
		\item[(1)] Non-vanishing of global period integrals for $G$ over a subgroup $H$ such that $G/H$ is a spherical variety, described in terms of Langlands functoriality and $L$-values; and
	
	\item[(2)] Congruences between systems of Hecke eigenvalues appearing in automorphic representations for $G$, captured through (classical) families in the eigenvariety for $G$.
\end{itemize}

The period integrals in (1) appear in the Gan--Gross--Prasad conjectures and are closely connected to the relative Langlands program. The families in (2) have been centrally important in number theory and arithmetic geometry for decades, essential to breakthroughs in the Langlands program (through modularity theorems, constructions of Galois representations, recent instances of Langlands functoriality, and proofs of local-global compatibility) and Iwasawa theory (in work on the Birch--Swinnerton-Dyer, Bloch--Kato and Iwasawa main conjectures).

In this paper, we consider these questions when $G = \GL_{2n}$ and $H = \GL_n \times \GL_n$. We first study (2), and our work towards this is explained in Theorem \ref{thm:intro 1} and Conjecture \ref{conj:intro} later in this introduction. We then use this study to consider (1), leading to Conjecture \ref{conj:intro FJ} and Theorem \ref{thm:intro 2} below.

\subsection{Classical families}

A system $\alpha$ of Hecke eigenvalues for $G$ is \emph{classical (cuspidal)} if it appears in a (cuspidal) automorphic representation $\pi$ of $G(\A)$. A \emph{classical (cuspidal) family} is any subspace of the eigenvariety in which the classical (cuspidal) points are Zariski-dense. A fundamental question is:
\begin{question-intro}\label{q:intro}
	In how many dimensions does $\alpha$ vary in a classical cuspidal family?
\end{question-intro}
In other words: let $\lambda$ be the weight of $\alpha$. Can we always find, for any $m \geq 0$, a classical cuspidal eigensystem $\alpha_m$ of some weight $\lambda_m$ such that $\alpha_m \equiv \alpha \newmod{p^m}$? In how many weight directions can we deform $\lambda$ to $\lambda_m$ and find such an $\alpha_m$?

A folklore expectation, described below, says every non-trivial classical family for $\GL_N$ arises  from some form of self-duality. Recall a cuspidal automorphic representation $\pi$ of $\GL_N(\A)$ is \emph{essentially self-dual} if there exists a Hecke character $\eta$ such that $\pi^\vee \cong \pi \otimes \eta^{-1}$; and then 
	\begin{equation}\label{eq:sym factorise}
		L(\pi \otimes \pi^\vee,s) = L(\mathrm{Sym}^2\pi \otimes \eta^{-1},s)\cdot L(\wedge^2\pi \otimes \eta^{-1},s)
		\end{equation}
		 factorises. Then:
	\begin{itemize}
		\item The left-hand side of \eqref{eq:sym factorise} has a simple pole at $s=1$, so either the symmetric square or exterior square $L$-function must have a pole at $s=1$. We say $\pi$ is \emph{orthogonal} in the first case, or \emph{symplectic} in the second. 
		\item A classical cuspidal eigensystem $\alpha$ (corresponding to a point in the eigenvariety, and appearing in an automorphic representation $\pi$) is orthogonal (resp.\ symplectic) if $\pi$ is orthogonal (resp.\ symplectic). 
		\item A classical cuspidal family for $\GL_N$ is orthogonal (resp.\ symplectic) if it contains a Zariski-dense set of orthogonal (resp.\ symplectic) points.
	\end{itemize}

In this paper, we consider Question \ref{q:intro} for symplectic families of $\GL_N(\A)$. Symplectic representations exist only for even $N$ (see \cite{AS14}), so let $G = \GL_{2n}$, and let $\alpha$ be attached to a regular algebraic cuspidal automorphic representation (RACAR) $\pi$ of $\GL_{2n}(\A)$ that admits a Shalika model (which is equivalent to $\pi$ being symplectic). We let $\pi_p$ denote the local component at $p$ (and use a similar notation scheme throughout the paper, for example for local vectors or local components of Hecke characters). We assume that $\pi_p$ is unramified, and the Satake parameter of $\pi_p$ is regular semisimple, in which case there are $(2n)!$ possible $p$-refinements $\tilde\pi = (\pi,\alpha)$ of $\pi$. Here a \emph{$p$-refinement} is a Hecke eigensystem $\alpha$ appearing in the Iwahori-invariants of $\pi_p$. 

In this paper, we define a stratification on the $(2n)!$ $p$-refinements $\alpha$ in terms of parabolic subgroups in $\GL_{2n}$, and we predict (in Conjecture \ref{conj:main conjecture}) the dimension of any symplectic family through a given $\alpha$ depends on its position in the stratification. We prove:
\begin{itemize}\s
	\item the upper bound on the dimension unconditionally;
	\item and the lower bound when $\alpha$ has non-critical slope.
\end{itemize}
We also give theoretical justification for the lower bound in general.

We predict that (modulo trivial variation, coming from twists by the norm) there can exist such symplectic families of exact dimension $d$ for any $d = 0,1,...,n$. This seems striking given that every component of the eigenvariety through any such $\alpha$ conjecturally has dimension $n$; so there should be classical families sitting inside `generically non-classical' components of the eigenvariety.

\begin{example}
	For $\GL_4$, there are 24 $p$-refinements $\tilde\pi$.  By \cite[Thm.\ 1.1.5]{Han17}, every irreducible cuspidal component of the $\GL_4$-eigenvariety is 2-dimensional (modulo trivial variation). Then: 
	\begin{itemize}\s
		\item 8 of the $\tilde\pi$ are essentially self-dual, and should vary in 2-dimensional symplectic families, each of which is then an irreducible component of the eigenvariety. 
		\item 8 are `symplectic rigid' -- we prove they do not vary in \emph{any} symplectic family. In any component through these points in the eigenvariety, the classical points should be discrete.
		\item 8 of them should vary in a 1-dimensional symplectic family, sitting in a 2-dimensional component of the eigenvariety, which should be generically non-classical.
	\end{itemize}
	In \S\ref{sec:examples} we give explicit examples of $(\pi,\alpha)$ in each of these cases, showing that `generically non-symplectic but with a positive-dimensional symplectic locus' cases do indeed occur.
\end{example}

\subsection{Previous work on classical families}
To put our results into context, we return to a general setting. Let $\fG$ be a reductive group. The previous work on Question \ref{q:intro} broadly falls into two cases:
\begin{itemize}\s
	\item[(I)] $\fG(\R)$ admits discrete series (true, for example, if $\fG$ forms part of a Shimura datum),
	\item[(II)] $\fG(\R)$ does \emph{not} admit discrete series.
\end{itemize}
In case (I), Question \ref{q:intro} is fairly well-understood: Urban \cite{Urb11} has shown that a (non-critical) cohomological cuspidal $\alpha$ \emph{always} varies `maximally', in all possible weight directions. This generalises the theory of Hida/Coleman families for modular forms ($\fG = \GL_2$). 

However, many fundamental cases -- e.g.\ $\GL_n$ for $n\geq 3$, and $\GL_2$ over non-totally-real fields -- are case (II), where our understanding of Question \ref{q:intro} is extremely poor. Ash--Pollack--Stevens \cite{APS08} and Calegari--Mazur \cite{CM09} considered the cases of $\GL_3$ and $\mathrm{Res}_{F/\Q}\GL_2$ respectively, for $F$ an imaginary quadratic field, and conjectured that:
\begin{myquote}
	For $\fG = \GL_3$ or $\mathrm{Res}_{F/\Q}\GL_2$, $\alpha$ varies in a positive-dimensional classical family if and only if $\alpha$ is essentially self-dual.
\end{myquote}

In \cite{Xia18}, Xiang has studied one direction of $(\dagger)$ more generally, proving that if $\alpha$ is essentially self-dual on $\GL_n$ (that is, both $\pi$ \emph{and} $\alpha$ are essentially self-dual) then $\alpha$ varies in a classical family in all `self-dual/pure' directions in weight space. Since every RACAR, hence every $\alpha$, has pure weight, this variation is `maximal' in the strongest possible sense.

One goal of this paper is to find analogues of $(\dagger)$ in higher-dimensional settings, where the picture is more subtle. Even when $\pi$ itself is essentially self-dual, it admits non-essentially-self-dual refinements $\alpha$, and we show that some of these can be varied in positive-dimensional classical families of smaller dimension. 

\subsection{Philosophy on classical families} \label{sec:intro philosophy}

Case (I) groups $\fG$ yield many classical families. A folklore expectation predicts this accounts for \emph{all} classical families, in the sense that every classical family is a $p$-adic Langlands transfer of a case (I) family. For example, conjecturally:

\begin{itemize}\s
\item[--] For $\GL_3$, all classical families are twists of symmetric square families for $\GL_2$;
\item[--] For $\mathrm{Res}_{F/\Q}\GL_2$, all classical families are twists of base-change families for $\GL_2$, or CM transfers of families for $\mathrm{Res}_{F'/F}\GL_1$, for $F'/F$ quadratic.
\end{itemize}

Before we describe our results precisely, let us explain why they fit strongly into this philosophy. We hesitantly suggest they provide further evidence towards it.

Any RACAR $\pi$ of $\GL_{2n}(\A)$ that admits a Shalika model is essentially self-dual, and a Langlands transfer of some RACAR $\Pi$ for $\mathrm{GSpin}_{2n+1}(\A)$. Note that $\cG \defeq \mathrm{GSpin}_{2n+1}$ is a case (I) group.  There are $2^n n!$ Iwahori $p$-refinements  of $\Pi$. 
By Urban's case (I) theorem, each of these varies in a maximal family over weight space (of dimension $n$, modulo trivial variation). 
In the style of Chenevier, each of these families should admit a transfer to $\GL_{2n}$ interpolating Langlands functoriality on classical points. These $n$-dimensional classical $\GL_{2n}$-families were constructed and studied in \cite{BDGJW}, and fall in the case studied by Xiang, corresponding exactly to the essentially self-dual eigensystems in $\pi$.

This only accounts, however, for $2^nn!$ of the $(2n)!$ possible $p$-refinements of $\pi$; even for $\GL_6$ this is only 48 out of 720. To look for classical families through the other refinements, we consider \emph{parabolic} families for $\cG$, as constructed and studied, for example, in \cite{HL11, BW20}. For any standard parabolic $\cP \subset \cG$, one can study $\cP$-parahoric refinements of $\Pi$. 
We show that for every refinement $\alpha$ of $\pi$, there exists a unique smallest parabolic $\cP \subset \cG$ such that $\alpha$ `is a functorial transfer of a $\cP$-refinement $\alphaG$ of $\Pi$'. Under a natural correspondence, $\cP$ corresponds to a unique `spin' parabolic $P\subset G$, and we call $\alpha$ an \emph{optimally $P$-spin refinement}. If $B$ is the corresponding Borel, the optimally $B$-spin refinements are exactly the $2^nn!$ essentially self-dual ones studied in \cite{BDGJW,Xia18}. All of this is defined in \S\ref{sec:P-spin}, where we give Weyl group, Hecke algebra, and combinatorial definitions of being $P$-spin, proving they are all equivalent.

Let $\alpha$ be an optimally $P$-spin refinement with associated spin eigensystem $\alphaG$. Under a non-criticality assumption, \cite{BW20} shows $\alphaG$ varies in a family in the $\cP$-parabolic $\cG$-eigenvariety over a smaller-dimensional weight space. Again, conceptually, this family should admit a transfer to the (Iwahoric) $G$-eigenvariety interpolating Langlands functoriality on classical points. This would produce a classical symplectic family in the $G$-eigenvariety through $\alpha$, of some smaller dimension depending on $\cP$ (hence $P$).

It is not clear how one should construct these transfer maps in general. There is a natural map of (abstract) Hecke algebras $\jmath^\vee : \cH^G \to \cH^{\cG}$ at Iwahoric level (see \eqref {eq:jmath hecke}), which should induce a map
\[
 \text{[Iwahoric-$\cG$-eigenvariety]}\longrightarrow \text{[Iwahoric-$G$-eigenvariety]}.
\]
However, one needs detailed automorphic information about classical points in the $\cG$-eigenvariety to control this, and in any case this recovers families already known to exist by \cite{BDGJW,Xia18}. At parahoric level the situation is worse: a transfer map\footnote{On the Galois representation side, this should correspond to the problem of consistently choosing triangulations of the $(\varphi,\Gamma)$-modules attached to every classical point in a paraboline family.}
\[
\text{[$\cP$-parahoric-$\cG$-eigenvariety]} \longrightarrow \text{[Iwahoric-$G$-eigenvariety]}
\]
should be induced from a map $\jmath^\vee_{\cP} : \cH^G \to \cH^{\cG,\cP}$ on abstract Hecke algebras, but now there is no natural map: the map  $\jmath^\vee$ above is surjective, so does not take values in  $\cH^{\cG,\cP} \subsetneq \cH^G$. To construct even a candidate $\jmath^\vee_{\cP}$, it seems necessary to \emph{presuppose} the existence of the family for $G$ one wants to construct. As such, we do not pursue this approach to families in this paper. 

\subsection{Our results on symplectic families}

To a spin parabolic $P$, in Definition \ref{def:X_P} we associate a subset $X_P \subset \{1,...,n\}$. Here $X_B = \{1,...,n\}$ and $X_G = \varnothing$. Let $\pi$ be a symplectic RACAR, and let $\alpha$ be an optimally $P$-spin refinement. In the main text, we denote this data by $\tilde\pi = (\pi,\alpha)$. We prove:

\begin{theorem-intro}\label{thm:intro 1}
	\begin{enumerate}[(i)]\s
		\item Any symplectic family $\sC$ through $\tilde\pi$ has dimension at most $\#X_P+1$.
		\item When $\tilde\pi$ has non-critical slope and regular weight, there exists a unique symplectic family through $\tilde\pi$, of dimension exactly $\#X_P+1$.
	\end{enumerate}
\end{theorem-intro}

 (Here we include, as in the main text, the 1-dimensional trivial variation).

In particular, if $\tilde\pi$ is optimally $G$-spin, then $\tilde\pi$ is `symplectic-rigid', varying in \emph{no} non-trivial symplectic family. There are, for example, 8 such refinements in the $\GL_4$ case.

Part (i) is Theorem \ref{thm:shalika obstruction}, which actually says more: that the weight support of such a family must lie in a $P$-parahoric weight space, which has dimension $\#X_P+1$.  
 To prove this, we show first that every classical point in $\sC$ is also optimally $P$-spin, and then obtain obstructions to the existence of optimally $P$-spin families varying outside the $P$-parabolic weight space. 

Part (ii) is Theorem \ref{thm:lower bound}. We show further that this unique component is \'etale over its image in weight space. To construct these families, we use a `refinement-switching' argument to move between points on the $\GL(2n)$-eigenvariety attached to a single $\pi$. The proof highlights interdependencies between the symplectic families through the different $p$-refinements, with implications for a hypothetical `infinite fern' construction for $\GL_{2n}$ (see Remark \ref{rem:infinite fern}). 

\medskip

We remark how Theorem \ref{thm:intro 1} fits into the philosophy above. Writing $\sE$ for the $\GL_{2n}$ eigenvariety of some fixed level, we expect there are an infinite number of closed embeddings $\{\iota_i : \sC_i \hookrightarrow \sE : i \in I\}$, where the $\sC_i$ are classical families in parabolic $\mathrm{GSpin}_{2n+1}$ eigenvarieties. Each $\sC_i$ is flat over the relevant parabolic weight space, and cannot be varied in higher dimension at the level of $\mathrm{GSpin}_{2n+1}$ eigensystems. However, $\sE$ varies over a higher-dimensional weight space, and in general $\iota_i(\sC_i)$ will sit properly inside some larger irreducible component of $\sE$. Theorem \ref{thm:intro 1} says that this irreducible component cannot have any further symplectic variation; that is, the subspaces $\iota_i(\sC_i)$ of $\sE$ cannot be assembled together into any classical family of higher dimension. In other words, all classical symplectic variation, and systematic congruences, should be accounted for by families in (parabolic) $\mathrm{GSpin}_{2n+1}$ eigenvarieties. This is predicted by our guiding philosophy on classical families in the eigenvariety, suggesting our results provide some further evidence for it. Indeed, motivated by the above theorem and the guiding philosophy, we conjecture:

\begin{conjecture-intro}\label{conj:intro}
	Every symplectic family through $\tilde\pi$ is the transfer of a classical parabolic family for $\mathrm{GSpin}_{2n+1}$ and has dimension $\#X_P+1$.
\end{conjecture-intro}

In \S\ref{sec:examples}, we give explicit examples for $\GL_4$ illustrating Theorem \ref{thm:intro 1} and Conjecture \ref{conj:intro}. 

\subsection{Non-vanishing of twisted period integrals}

We give an application to the study of non-vanishing of period integrals. Let $\pi$ be a RACAR of $G(\A)$, and let $H = \GL_n\times\GL_n \subset G$. If $\chi$ is an algebraic Hecke character and $\varphi \in \pi$, then in \eqref{eq:period integral} we define an attached global period integral for $H\subset G$, denoted $Z_H(\varphi,\chi,s)$. The same kind of period integral appears in the GGP conjectures, and is related to the relative Langlands program.

 A result of Friedberg--Jacquet \cite{FJ93} says that for any $s \in \C$, the following are equivalent:
	\begin{enumerate}[(1)]\s
		\item There exists $\varphi \in \pi$ such that $Z_H(\varphi,\chi,s+1/2) \neq 0$;
		\item $\pi$ is a functorial transfer of some $\Pi$ on $\mathrm{GSpin}_{2n+1}(\A)$, and $L(\pi\times \chi, s+1/2) \neq 0$.
	\end{enumerate}

This is related to the relative Langlands program \cite{SV17}; $G/H$ is a spherical variety, and $Z_H$ is an $H$-period integral (that appears, for example, in the GGP conjectures in related settings). This phenomenon is also explained in great generality in \cite[p.174]{JLR01}.

We propose a $p$-refined analogue of this. Let $P \subsetneq G$ be a proper spin parabolic, let $\beta \geq 1$, and let $ut_P^\beta \in G(\Qp)$ be the element defined in Notation \ref{not:u}. Here $u$ is a representative for the open orbit of the action of $B$ on $G/H$ and $t_P$ defines the Hecke operator at $P$.  Let $\tilde\pi$ be a $P$-parahoric $p$-refinement of $\pi$.

\begin{conjecture-intro}\label{conj:intro FJ}
	Suppose $\chi$ is finite order and has conductor $p^\beta > 1$. For any $s \in \C$, the following are equivalent:
\begin{enumerate}[(1)]\s
	\item There exists an eigenvector $\varphi \in \tilde\pi^P$ such that $Z_H(ut_P^\beta \cdot \varphi,\chi,s+1/2) \neq 0$.
	\item All of the following hold:
	\begin{itemize}\s 
		\item[--] $P$ is contained in the $(n,n)$-parabolic (in the sense of Notation \ref{not:(ni)-spin}).
		\item[--] $\tilde\pi^P$ is a functorial transfer of some $\cP$-refined $\tilde\Pi^{\cP}$ on $\mathrm{GSpin}_{2n+1}(\A)$, 
		\item[--]$L(\pi\times\chi, s+1/2) \neq 0$,
	\end{itemize}
\end{enumerate}
\end{conjecture-intro}

We actually state a stronger, and purely local, version of this conjecture in Conjecture \ref{conj:local FJ}. We give this weaker global form in the introduction as it is closer to the original result of Friedberg--Jacquet. The close connection between our local conjecture and this global one is explained in detail in Proposition \ref{prop:local vs global}.

The quantity $Z_H(ut_P^\beta \cdot \varphi,\chi,s+1/2)$, or closely related expressions, appear in constructions of $p$-adic $L$-functions via Shalika models \cite{AG94, Geh18, DJR18,BDW20,BDGJW, Wil-2n}. Conjecture \ref{conj:intro FJ} highlights a close relationship between the $P$-spin conditions defined in this paper, and settings where we can expect to construct non-zero $p$-adic $L$-functions via Shalika models. In this light, the requirement in (2) that $P$ is contained in the $(n,n)$-parabolic $Q$ is natural; the Panchishkin condition \cite{Pan94} predicts that to be able to attach a $p$-adic $L$-function to $\tilde\pi^P$, one requires $P \subset Q$. 

As evidence towards this conjecture, we use Theorem \ref{thm:intro 1} to prove:

\begin{theorem-intro}\label{thm:intro 2}
	\begin{itemize}\s
		\item[(i)] (2) $\Rightarrow$ (1) holds in Conjecture \ref{conj:intro FJ}.
	\item[(ii)] Suppose $\pi$ has regular weight and there is a non-critical slope further refinement $\tilde\pi$ of $\tilde\pi^P$ to Iwahori level. Then (1) $\Rightarrow$ (2) holds in Conjecture \ref{conj:intro FJ}.
	\end{itemize}
\end{theorem-intro}

In particular, the conjecture holds in full for a large class of $\tilde\pi^P$. We actually show (ii) (and deduce the full conjecture) under weaker assumptions on $\tilde\pi^P$, which we cautiously imagine could hold for \emph{all} $\tilde\pi^P$; see Theorem \ref{thm:p-refined FJ} and Remarks \ref{rem:p-refined FJ}.

Our proof of Theorem \ref{thm:intro 2}(i) is purely local, indeed proving the stronger implication in the local version (Conjecture \ref{conj:local FJ}): given (2), we directly exhibit an eigenvector satisfying (1) using methods developed in \cite{BDGJW}. To prove (ii), we deploy global methods, using ideas from \cite{BDW20,BDGJW} to show that if (1) holds, then we can construct a symplectic family through $\tilde\pi$ over the $P$-parahoric weight space. By (the stronger form of) Theorem \ref{thm:intro 1}(i), this forces $\tilde\pi^P$ to be $P$-spin, hence $\tilde\pi^P$ is a functorial transfer.

We expect that this relationship between non-vanishing of twisted period integrals attached to a $p$-refinement, and the refinement being a functorial transfer, should be true much more generally. In future work with Lee, we hope to treat the case of twisted Flicker--Rallis integrals for $\GL_n$ over a CM field, showing non-vanishing implies transfer from a unitary group.

\subsection*{Acknowledgements} We thank our co-authors on \cite{BDGJW}, Mladen Dimitrov and Andrei Jorza, for many stimulating discussions whilst preparing our earlier joint paper. We also thank Valentin Hernandez for interesting discussions on Zariski-density of crystalline points and the infinite fern, and the anonymous referees for their valuable comments and corrections. AG was (partly) supported by ERC-2018-COG-818856-HiCoShiVa and UK Research and Innovation grant MR/V021931/1. CW was supported by EPSRC Postdoctoral Fellowship EP/T001615/1. DBS was supported by FONDECYT 11201025. For the purpose of Open Access, the authors have applied a CC BY public copyright licence to any Author Accepted Manuscript (AAM) version arising from this submission.

\subsection*{Data Availability Statement}

Data sharing is not applicable to this article.

\subsection*{Conflict of Interest Statement}

The authors declare that there are no conflicts of interest.

\section*{Set-up and notation}

Let $n\geq 1$ and let $G \defeq \mathrm{GL}_{2n}$. We write $B = B_{2n}$ for the Borel subgroup of upper triangular matrices, $\overline{B} = \overline{B}_{2n}$ for the opposite Borel of lower triangular matrices and $T = T_{2n}$ for the maximal split torus of diagonal matrices. 

Let  $\cG = \mathrm{GSpin}_{2n+1}$. Fix a Borel subgroup $\cB \subset \cG$ and a maximal split torus $\cT \subset \cB$.

If $\pi$ is a regular algebraic cuspidal automorphic representation (RACAR) of $G(\A)$, then recall from the introduction that $\pi$ is \emph{symplectic} if there exists a Hecke character $\eta$ such that $L(\wedge^2 \pi \times \eta^{-1}, s)$ has a pole at $s=1$.  We call such a $\pi$ a RASCAR (RA-symplectic-CAR). In this case $\pi$ is essentially self-dual, in that $\pi^\vee \cong \pi \otimes \eta^{-1}$. Moreover, by \cite{AS06,FJ93} the following three conditions are equivalent:
\begin{itemize}\setlength{\itemsep}{0pt}
	 \item $\pi$ is symplectic;
	 \item $\pi$ is a functorial transfer of a RACAR $\Pi$ on $\cG(\A)$;
	 \item $\pi$ admits a Shalika model, in the sense of \cite[\S2.3]{FJ93}.
\end{itemize}
Here the functorial transfer is induced from the natural inclusion $\GSp_{2n} \subset \GL_{2n}$, noting that the $L$-group of $\cG$ is $\GSp_{2n}(\C)\rtimes G_{\Q}$, whilst the $L$-group of $G$ is $\GL_{2n}(\C)\rtimes G_{\Q}$.
			
 The equivalence above can be refined: if $\eta$ is the character such that $\pi^\vee \cong \pi \otimes \eta^{-1}$, then $\Pi$ has central character $\eta$, and $\pi$ has an $(\eta,\psi)$-Shalika model.

\part{$P$-spin refinements}

\section{Structure theory and parahoric $p$-refinements}

\subsection{Root systems and spin parabolics}\label{sec:structure gspin}

Our study of `spin' refinements is rooted in the structure theory of $\GL_{2n}$ and $\mathrm{GSpin}_{2n+1}$. We recall the following from \cite[\S6]{BDGJW}.

The spaces of algebraic characters/cocharacters of the torus $T \subset G = \GL_{2n}$ are
\[
X = \Z e_1 \oplus \Z e_2 \oplus \cdots \Z e_{2n}, \hspace{12pt} X^\vee = \Z e_1^* \oplus \Z e_2^* \oplus \cdots \Z e_{2n}^*.
\]
The root system for $G$ is $A_{2n-1}$, with roots $R = \{\pm(e_i - e_j) : 1 \leq i < j \leq 2n\}$, positive roots $\{e_i - e_j : i < j\}$, 
 and simple roots $\Delta_G = \{a_i \defeq e_i - e_{i+1} : i = 1,...,2n-1\}$. The Weyl group $\cW_G = \mathrm{S}_{2n}$ acts by permuting the $e_i$. We set this up so that $\sigma \in \cW_G$ sends $e_i$ to $e_{\sigma^{-1}(i)}$, hence $\sigma$ acts on a character $\mu = (\mu_1,...,\mu_{2n}) \in X$ as $\mu^\sigma = (\mu_{\sigma(1)},...,\mu_{\sigma(2n)})$.

Let $X_0 \subset X$ be the space of \emph{pure characters} $X_0 = \{\lambda \in X: \exists \sw(\lambda) \in \Z \text{ such that } \lambda_i + \lambda_{2n-i+1} = \sw(\lambda) \ \forall 1 \leq i \leq n\}$, and let 
\begin{equation}\label{eq:W_G^0}
	\cW_G^0 \defeq \{ \sigma \in \cW_G :  \sigma(X_0) \subset X_0\} \subset \cW_G.
\end{equation}
There is a splitting $\cW_G^0 = \{\pm 1\}^n \rtimes \mathrm{S}_n,$ where: 
\begin{itemize}\s
	\item for $1 \leq i \leq n$,  $\sigma \in \mathrm{S}_n$ sends $e_i$ to $e_{\sigma^{-1}(i)}$, and $e_{2n+1-i}$ to $e_{2n+1-\sigma^{-1}(i)}$; 
	\item and the $i$th copy of $\{\pm1\}$ acts by swapping $e_i \leftrightarrow e_{2n+1-i}$. 
\end{itemize}
Identifying $i \leftrightarrow e_i$, we view $\cW_G^0$ as a subgroup of $\mathrm{S}_{2n}$, and have the following easy fact:
\begin{lemma}\label{lem:relative}
	If $\sigma \in \cW_G^0$, then $\sigma(i) + \sigma(2n+1-i) = 2n+1$ for all $1 \leq i \leq n$.
\end{lemma}

Recall we fixed a Borel subgroup $\cB$ and maximal split torus $\cT$ in $\cG = \mathrm{GSpin}_{2n+1}$. This has rank $n+1$ \cite[Thm.\ 2.7]{Asg02}. We use calligraphic letters to denote objects for GSpin, whilst keeping other notational conventions as before. 

\begin{proposition}\label{prop:cG root system}
	The root system for $\cG$ is $(\cX, \cR, \cX^\vee, \cR^\vee)$, where
	\[
	\cX = \Z f_0 \oplus \Z f_1 \oplus \cdots \oplus \Z f_n, \hspace{12pt} \cX^\vee = \Z f_0^* \oplus \Z f_1^* \oplus \cdots \oplus \Z f_n^*,
	\] with roots $\cR = \{\pm f_i \pm f_j : 1 \leq i<j\leq n\} \cup \{f_i : 1 \leq i \leq n\}$, simple roots 
	\[
		\Delta_{\cG} = \{b_i \defeq f_i - f_{i+1}  : i = 1,...,n-1 \} \cup \{b_n \defeq f_n\},
	\]
	 and  positive roots $\{f_i : 1 \leq i \leq n\} \cup \{f_i \pm f_j : 1 \leq i < j \leq n\}$. 	The Weyl group $\cW_{\cG}$ is isomorphic to $\{\pm1\}^n \rtimes \mathrm{S}_n$, generated by permutations $\sigma \in \mathrm{S}_n$ and sign changes $\mathrm{sgn}_i$, which act on roots and coroots respectively as (for $j \neq i$)
	\begin{equation}\label{eq:weyl action cG}
		\sigma f_0 = f_0, \ \ \sigma f_i = f_{\sigma^{-1}(i)},\ \ \ \mathrm{sgn}_i f_0 = f_0 + f_i,\ \ \mathrm{sgn}_i(f_i) = -f_i, \ \ \mathrm{sgn}_j(f_i) = f_i,
	\end{equation}
	\[
	\sigma f_0^* = f_0^*, \ \ \sigma f_i^* = f_{\sigma^{-1}(i)}^*,\ \ \ \mathrm{sgn}_i f_0^* = f_0^*,\ \ \mathrm{sgn}_i(f_i^*) = f_0^* -f_i^*,  \ \ \mathrm{sgn}_j(f_i^*) = f_i^*.
	\]
\end{proposition}
\begin{proof}
	The first part is \cite[Prop.\ 2.4]{Asg02}, and the second \cite[Lem.\ 13.2.2]{HS16}.
\end{proof}

Write $\langle -,-\rangle_G$ (resp.\ $\langle-,-\rangle_{\cG}$) for the natural pairing on $X \times X^\vee$ (resp.\ $\cX \times \cX^\vee$).

There is a natural injective map $\jmath : \cX \hookrightarrow X$ given by
\begin{equation}\label{eq:f and e}
	f_i \longmapsto e_i - e_{2n-i+1} \text{ for } 1 \leq i \leq n,\qquad 
	f_0  \longmapsto e_{n+1} + \cdots + e_{2n},\notag
\end{equation}
with $X_0 = \jmath(\cX)$ by \cite[Prop.\ 6.5]{BDGJW}. If $\rho_G$ and $\rho_{\cG}$ are half the sum of the positive roots for $G$ and $\cG$ respectively, a simple check shows $\jmath(\rho_{\cG}) = \rho_G$. We also have:

\begin{proposition}[\cite{BDGJW}, Proposition 6.6] \label{prop:weyl transfer} There is a map $\cW_{\cG} \to \cW_G$ of Weyl groups, also denoted $\jmath$, such that:
	\begin{itemize}\s
		\item[(i)] $\jmath$ induces an isomorphism $\cW_{\cG} \cong \cW_G^0 \subset \cW_G$;
		
		\item[(ii)] for all  $\sigma \in \cW_{\cG}$ and $\mu \in \cX$, we have $\jmath(\mu^\sigma) = \jmath(\mu)^{\jmath(\sigma)}$.
	\end{itemize}
\end{proposition}

Dually, define also a map $\jmath^\vee: X^\vee \to \cX^\vee$ by sending $\nu \in X^\vee$ to
\[
\jmath^\vee(\nu) \defeq \sum_{i = 0}^{n} \big\langle \jmath(f_i), \nu\big\rangle_G \cdot f_i^*.
\]
Then for all $\mu \in \cX$ and $\nu \in X^\vee$, we have 
\begin{equation}\label{eq:pairing jmath}
	\langle \mu, \jmath^\vee(\nu)\rangle_{\cG} = \langle \jmath(\mu), \nu\rangle_G
\end{equation}
by construction.  Also let $\jmath^\vee : \cW_G^0 \to \cW_{\cG}$ denote the inverse to $\jmath : \cW_{\cG} \cong \cW_G^0$.

\begin{proposition}[\cite{BDGJW}, Proposition 6.7]\label{prop:jmath vee eq}
	For all $\nu \in X^\vee$ and $\sigma \in \cW_G^0$, we have  
	\[
		\jmath^\vee(\nu^\sigma) = \jmath^\vee(\nu)^{\jmath^\vee(\sigma)}.
	\]
\end{proposition}

We take a brief general intermission. For any quasi-split reductive group $\mathrm{G}$ with a fixed choice of Borel pair $(\mathrm{B},\mathrm{T})$, there is a well-known inclusion-preserving correspondence between standard parabolic subgroups $\mathrm{P}$ of $\mathrm{G}$ and subsets $\Delta_{\mathrm{P}}$ of the set $\Delta$ of simple roots (see e.g. \cite[\S2.3]{BW20}). Here $\mathrm{B}$ corresponds to the empty set, and any proper maximal standard parabolic corresponds to $\Delta\backslash \{a\}$ for some simple root $a \in \Delta$. Further, for any such $\mathrm{P}$ we have a Levi subgroup $L_{\mathrm{P}}$, with Weyl group $\cW_{L_{\mathrm{P}}}$, which is naturally a subgroup of $\cW_{\mathrm{G}}$ (namely, the subgroup that preserves the $\Z$-span of $\Delta_{\mathrm{P}}$). 

\medskip

Returning to our specific set-up, note that $\jmath$ acts on simple roots by sending
\[
	b_1 \mapsto a_1 + a_{2n-1}, \ \ b_2 \mapsto a_2 + a_{2n-2}, \ \ \dots, \ \ b_{n-1} \mapsto a_{n-1}+a_{n+1}, \ \ b_n \mapsto a_n.
\]

\begin{definition}\label{def:r_P}
	Let $P \subset G = \GL_{2n}$ be a standard parabolic, corresponding to a subset $\Delta_P \subset \Delta_G$. We say $P$ is a \emph{spin parabolic} if, for any $i$, $a_i \in \Delta_P$ implies $a_{2n-i} \in \Delta_P$; that is, $\Delta_P$ is a union of some of the sets
	\[
	A_1 \defeq \{a_1,a_{2n-1}\}, \ \ A_2 \defeq \{a_2,a_{2n-2}\}, \dots, \ \ A_{n-1} \defeq \{a_{n-1},a_{n+1}\},\ \ A_n \defeq \{a_n\}.
\]
\end{definition}

If $P$ is a spin parabolic, then there is a corresponding parabolic $\cP \subset \cG$, defined by
\[
	b_i \in \Delta_{\cP} \iff A_i \subset \Delta_P.
\]
Under this correspondence the Borel subgroups $B\subset G$ and $\cB \subset \cG$ are identified.

\begin{notation}\label{not:(ni)-spin}
	We call the parabolic $P$ with Levi $\GL_{n_1} \times \cdots \times \GL_{n_r}$ the $(n_1,...,n_r)$-parabolic. Note that $P$ is a spin parabolic if and only if $(n_1,...,n_r)$ is symmetric around the middle (so the (1,4,1)-parabolic is spin, but the (1,3,2)-parabolic is not).
\end{notation}


\subsection{Parahoric $p$-refinements for $G$}
Let $\pi$ be a $p$-spherical RASCAR of $\GL_{2n}(\A)$. We can write $\pi_p = \Ind_B^G \UPS$ as an unramified principal series representation. Here $\UPS = (\UPS_1,\dots,\UPS_{2n})$ is an unramified character of $T(\Qp) \cong (\Qp^\times)^{2n}$, and we use the normalised induction
	\[
	\Ind_B^G\UPS \defeq \Big\{f : G(\Qp) \to \C : f\text{ smooth, }f(bg) = \delta_B^{1/2}\UPS(b)f(g) \ \ \forall b\in B(\Qp)\Big\},
	\]
	where $\delta_B$ is the standard modulus character on $B(\Qp)$. The choice of $\UPS$ is not unique: we may replace $\UPS$ with $\UPS^\sigma$ for any $\sigma \in \cW_G = \mathrm{S}_{2n}$ in the Weyl group of $G$. Precisely, $\sigma$ permutes the constituent characters $\UPS_i$ by $\UPS_i^\sigma = \UPS_{\sigma(i)}$.

\begin{definition}\label{def:spin} 
	We say $\UPS$ is \emph{spin} if 
	\begin{equation}\label{eq:UPS}
		\UPS_1\UPS_{2n} = \UPS_2\UPS_{2n-1} = \cdots = \UPS_n\UPS_{n+1} = \eta_p.
	\end{equation}
\end{definition}

Since $\pi_p$ admits an $(\eta_p,\psi_p)$-Shalika model, using \cite{AG94} and \cite[\S6.1]{BDGJW} we may (and will) choose $\UPS$ to be spin. This is the `Asgari--Shahidi' convention on $\UPS$ described in \cite[\S6.1]{BDGJW}. This is still not unique: we could replace $\UPS$ with $\UPS^\sigma$, for any $\sigma\in\cW_G^0 \subset \cW_G$. 

Note this is \emph{different} from how we chose $\UPS$ in \cite{BDGJW}, where we assumed $\UPS_i\UPS_{n+i} = \eta_p$. The two choices are exchanged by $\tau = \smallmatrd{1}{}{}{w_n} \in \cW_G$ (see \S6.1 and Remark 6.12 \emph{op.\ cit}.).

\medskip

Now let $B \subset P \subset \GL_{2n}$ be a standard parabolic, with associated parahoric subgroup $J_P \defeq \{g \in \GL_{2n}(\Zp) : g \newmod{p} \in P(\F_p)\}$. Note that $J_B = \Iw$ is an Iwahori subgroup of $G(\Qp)$. 

\begin{definition}\label{def:P-refinement}
	\begin{itemize}\s
		\item[--] For $1 \leq r \leq 2n$, let  $t_{p,r} = \smallmatrd{p I_r}{}{}{I_{n-r}} = (e_1^*+\cdots + e_r^*)(p) \in T(\Qp)$. Let $U_{p,r}^P = [J_P t_{p,r} J_P]$ be the associated double coset operator for $J_P$.
		\item[--] Let $\cH_p^P \defeq \Qp[U_{p,r}^P, U_{p,2n}^P : 1 \leq r \leq 2n-1,a_r \not\in \Delta_P]$.
	\end{itemize}
\end{definition}

Formally, $\cH_p^P$ is the free commutative $\Qp$-algebra generated by the symbols $U_{p,r}^P$.  Via the description as double coset operators, \cite[Prop.\ 4.3]{HidP-ord} (with $C = I_{P,1}$ in the notation \emph{op.\ cit}.) shows we can identify $\cH_p^P$ with a commutative subalgebra of the $P$-parahoric Hecke algebra $C^\infty_c(J_P\backslash G(\Qp)/J_P)$ at $p$. Thus $\cH_p^P$ acts on $\pi_p^{J_P}$ by convolution product, with $U_{p,r}^P$ acting as averaging over representatives of $J_Pt_{p,r}J_P/J_P$. 

\begin{definition}Fix an isomorphism $i_p : \C \to \overline{\Q}_p$.
	
		\begin{itemize}
		\item[--]	A \emph{$P$-parahoric $p$-refinement}  of $\pi$, or \emph{$P$-refinement} for short, is a system $\alpha^P : \cH_{p}^P \to \overline{\Q}_p$ of Hecke eigenvalues such that $i_p^{-1}\circ \alpha_P$ appears in $\pi_p^{J_P}$. As the $U_{p,r}$-eigenvalues on $\pi^{J_P}$ are algebraic, this depends only lightly on $i_p$. We denote this as $\tilde\pi^P = (\pi,\alpha^P)$. 
		
		\item[--]  If $P = B$, then we write `$p$-refinement' or `Iwahori $p$-refinement' instead of `$B$-parahoric $p$-refinement'. We drop the superscript $B$, writing $\cH_p \defeq \cH_p^B$, $\alpha \defeq \alpha^B$, $\tilde\pi \defeq \tilde\pi^B$, etc.
	\end{itemize}
\end{definition}

\begin{remarks}\label{rem:restrict refinement}
	\begin{itemize}\s
\item[(i)] The algebra $\cH_p$ is the Hecke algebra considered in \cite{BDGJW}. If $Q$ is the $(n,n)$-parabolic, then $\cH_p^Q = \Qp[U_{p,n},U_{p,2n}]$ is the Hecke algebra considered in \cite{BDW20}.

\item[(ii)] If $P'\subset P$ are two parabolics, then we have a natural injective map 
\[
\cH_p^P\hookrightarrow \cH_p^{P'}, \qquad U_{p,r}^P \mapsto U_{p,r}^{P'}.
\]
Taking $P' = B$,  this allows us to identify $\cH_p^P$ as a subalgebra of $\cH_p$ for all $P$. Via \cite[Cor.\ 3.16]{OST19} (see also Proposition \ref{prop:p-refinement} and \eqref{eq:Up char poly} below), any (Iwahori) $p$-refinement $\tilde\pi = (\pi,\alpha)$ restricts to a unique $P$-parahoric $p$-refinement $\tilde\pi^P = (\pi,\alpha^P)$, with $\alpha^P \defeq \alpha|_{\cH_p^P}$.

\item[(iii)] Part (ii) shows that the $U_{p,r}^P$-eigenvalues appearing in $\pi_p^{J_P}$ are independent of $P$ (that is, they depend only on $r$ and $\pi_p$). In light of this, we will henceforth abuse notation and write just $U_{p,r}$, dropping $P$ from notation. 
\end{itemize}
\end{remarks}

The following describes the possible $p$-refinements in terms of the Weyl group $\cW_G = \mathrm{S}_{2n}$. The \emph{Satake parameter} of an irreducible unramified principle series representation $\pi_p = \Ind_B^G \UPS$ is $\UPS(p) = (\UPS_1(p),\dots, \UPS_{2n}(p)) \in (\C^\times)^{2n}$. It is \emph{regular semisimple} if the $\UPS_i(p)$ are pairwise distinct.

\begin{proposition} \label{prop:p-refinement}
	Suppose the Satake parameter of $\pi_p = \Ind_B^G\UPS$ is regular semisimple. 
	\begin{itemize}\s 
		\item[(i)] There is a bijection (that depends on $\UPS$)
		\begin{equation}\label{eq:Psi_UPS}
		\Psi_\UPS : \{\text{Iwahori }p\text{-refinements of }\pi\}  \longrightarrow \cW_G,
	\end{equation}
  such that if $\tilde\pi = (\pi,\alpha)$ is a $p$-refinement with $\Psi_{\UPS}(\tilde\pi) = \sigma$, then for each $r$ we have
	\begin{equation}\label{eq:chenevier}
	\alpha(U_{p,r}) =\delta_B^{-1/2}\UPS^\sigma(t_{p,r}) = \prod_{j=1}^r p^{-\tfrac{2n-2j+1}{2}}\UPS_{\sigma(j)}(p) \neq 0.
	\end{equation}
	
	\item[(ii)] If $P$ is a standard parabolic with Levi subgroup $L_P$, there is a bijection
	\[
			\Psi_\UPS^P : \{P\text{-refinements of }\pi\}  \longrightarrow \cW_G/\cW_{L_P},
	\]
	such that if $\tilde\pi^P = (\pi,\alpha^P)$ is a $P$-refinement with $\Psi_{\UPS}^P(\tilde\pi^P) = [\sigma]$ for $\sigma \in \cW_G$, then $\alpha^P(U_{p,r})$ is given by \eqref{eq:chenevier} whenever $U_{p,r} \in \cH_p^P$.
	
	\item[(iii)] If $\tilde\pi^P$ is a $P$-refinement, then the possible extensions to Iwahori level are exactly the $p$-refinements $\tilde\pi$ with $\Psi_{\UPS}(\tilde\pi) = \Psi_{\UPS}^P(\tilde\pi^P) \newmod{\cW_{L_P}}$.
	\end{itemize}
\end{proposition}
\begin{proof}
 (i) is \cite[Lem.\ 4.8.4]{Che04}. (ii) is \cite[Cor.\ 3.16]{OST19}. (iii) is immediate.
\end{proof}

\begin{remark}\label{rem:change of UPS}
	For any $\nu \in \cW_G$ and any $p$-refinement $\tilde\pi$,  we have $\Psi_{\UPS^\nu}(\tilde\pi) = \nu \Psi_{\UPS}(\tilde\pi)$.  In \cite{BDGJW} we denoted $\UPS$ for what would be $\UPS^\tau$ here, where $\tau = \mathrm{diag}(1,w_n)$, where $w_n$ is the longest Weyl element for $\GL_n$. Thus our bijection $\Psi_{\UPS}$ is denoted $\Delta_{\UPS^\tau}$ there.
\end{remark}

\begin{remark}
We will assume regular semisimplicity of the Satake parameter of $\pi_p$ throughout this paper. This is a fairly mild assumption: for example, it conjecturally holds for all unramified local components of any RACAR of $\GL_2(\A)$ (equivalent, by Proposition \ref{prop:p-refinement}, to the two roots of the Hecke polynomial at $p$ being distinct). 

In general, let $\pi$ be a fixed RASCAR of $\GL_N(\A)$. Whilst it is no longer necessarily true that regular semisimplicity holds for \emph{all} unramified primes\footnote{For example, let $E$ be an elliptic curve with good supersingular reduction at $p$, with associated RACAR $\Pi$ of $\GL_2(\A)$.  Then $\pi \defeq \mathrm{Sym}^3\Pi$ is a RASCAR of $\GL_4(\A)$ whose Satake parameter is not regular semisimple at $p$.}, it still holds for a density 1 set of primes. Indeed, via local-global compatibility in the Langlands correspondence (known for essentially self-dual representations, hence for $\pi$) there is a compatible system of Galois representations $\rho_{\pi,\ell}$ attached to $\pi$, regular in the sense of having distinct Hodge-Tate weights (as $\pi$ is regular algebraic), and such that the eigenvalues of $\mathrm{Frob}_p$ correspond to the Satake parameters at $p$ for unramified primes $p$. This claim is then justified in the proof of \cite[Lem.\ 5.3.1(2)]{BGGT14a}.
\end{remark}

The $U_{p,r}$-eigenvalues will not, in general, vary $p$-adic analytically. For $p$-adic interpolation, we must instead use normalised analogues $U_{p,r}^\circ$ of $U_{p,r}$. For this, we must introduce the \emph{weight} of $\pi$. Our convention is that the weight is the unique algebraic character $\lambda$ of $T$ such that $\pi$ contributes to (Betti) cohomology with coefficients in $V_\lambda^\vee$, where $V_\lambda$ is the irreducible representation of $G$ of highest weight $\lambda$. This is summarised in detail in \cite[\S2.5]{BDW20}.

\begin{definition}
	 If $\lambda$ is the weight of $\pi$, we define
	\[
	U_{p,r}^\circ = \lambda(t_{p,r}) U_{p,r}  = p^{\lambda_1+\cdots+\lambda_r}U_{p,r} \in  \cH_p.
	\]
\end{definition}

Let $\pi$ be a RASCAR of weight $\lambda$, and $P$ a spin parabolic. By \cite[Lem.\ 4.9]{Clo90}, $\lambda$ is \emph{pure}, in the sense that there exists $\sw \in \Z$ such that $\lambda_i + \lambda_{2n+1-i} = \sw$ for all $1 \leq i \leq n$.

\begin{definition}\label{def:non-critical slope}
	Let $\tilde\pi^P = (\pi, \alpha^P)$ be a $P$-refinement of $\pi$. We say $\tilde\pi^P$ has \emph{non-$P$-critical slope} if
	\[
	v_p(\alpha^P(U_{p,r}^\circ)) < \lambda_{r} - \lambda_{r+1} + 1 \qquad \text{for all } 1 \leq r \leq 2n-1 \text{ with } a_r \not\in \Delta_P.
	\]
	(Note that $\lambda_r - \lambda_{r+1} = \lambda_{2n-r} - \lambda_{2n-r+1}$ by purity, so the bounds for $U_{p,r}^\circ$ and $U_{p,2n-r}^\circ$ agree).
	
	We say a $p$-refinement $\tilde\pi$ has non-$P$-critical slope if its associated $P$-refinement $\tilde\pi^P$ does. We say $\tilde\pi$ has non-critical slope if it has non-$B$-critical slope.
\end{definition}


\section{$P$-spin refinements}\label{sec:P-spin}

Let $P \subset G = \GL_{2n}$ be a spin parabolic. We now generalise \cite[\S6]{BDGJW} to an arbitrary such $P$. Let $\pi = \Ind_B^G \UPS$ be a RASCAR of $\GL_{2n}(\A)$ that is spherical and regular at $p$, recalling we have fixed a spin $\UPS$ satisfying $\UPS_1\UPS_{2n} = \cdots = \UPS_n\UPS_{n+1} = \eta_p$ \eqref{eq:UPS}. Recall $\Psi_{\UPS}$ from \eqref{eq:Psi_UPS}.

\begin{definition}\label{def:P-spin}
	\begin{itemize}\s
\item	We say an Iwahori $p$-refinement $\tilde\pi = (\pi,\alpha)$ is a \emph{$P$-spin refinement} if 
	\[
		\Psi_{\UPS}(\tilde\pi) \in \cW_G^0 \cdot \cW_{L_P} \subset \cW_G.
	\]
	\item We say a $P$-refinement $\tilde\pi^P$ is $P$-spin if 
	\[
		\Psi_{\UPS}^P(\tilde\pi^P) \in \mathrm{Im}\Big(\cW_{G}^0 \to \cW_G \to \cW_G/\cW_{L_P}\Big) \subset \cW_G/\cW_{L_P}.
	\] 
\end{itemize}
\end{definition}

\begin{lemma}
A $P$-refinement $\tilde\pi^P$ is $P$-spin if and only if all of its extensions to Iwahori $p$-refinements are $P$-spin.	
\end{lemma}
\begin{proof}
Immediate from the definitions and Proposition \ref{prop:p-refinement}(iii).
\end{proof}

\begin{remarks}\label{rem:P-spin ind of UPS} 
\begin{itemize}\s
	\item[(i)] The cases of $B$-spin and $Q$-spin refinements, for $Q$ the $(n,n)$-parabolic, were defined in \cite[Lem.\ 6.12, Rem.\ 6.14]{BDGJW}. 
	
	\item[(ii)] Since any two choices of spin $\UPS$ differ by an element of $\cW_G^0$, this definition is independent of such a choice of $\UPS$ by Remark \ref{rem:change of UPS}. 
\end{itemize}
\end{remarks}

\subsection{$P$-spin refinements via Hecke algebras}   

Recall objects for $\cG = \mathrm{GSpin}_{2n+1}$ (e.g.\ Borel $\cB$, parabolics $\cP$) are written as calligraphic versions of objects for $G = \mathrm{GL}_{2n}$ (e.g. $B,P$).

As $\pi$ is symplectic, it is the functorial transfer of a RACAR $\Pi$ of $\cG(\A)$. Moreover $\Pi_p = \Ind_{\cB}^{\cG}\UPS_{\cG}$ is an unramified principal series for $\cG(\Qp)$, for $\UPS_{\cG}$ an unramified character of $\cT$ satisfying $\jmath(\UPS_{\cG}) = \UPS$ (by \cite[p.177(i)]{AS06} and \cite[Prop.\ 5.1]{AS14}). 

Our primary motivation for $P$-spin refinements is that they interact well with this functoriality, as we will show in Proposition \ref{prop:spin factor}.

\subsubsection{Parahoric refinements for $\mathrm{GSpin}_{2n+1}$}

\begin{definition}
	Let $\cB \subset \cP \subset \cG$ be a parabolic, with parahoric subgroup $\cJ_{\cP} \subset \cG(\Zp)$.
	\begin{itemize}\s
		\item For $1 \leq r \leq n$, let $\cU_{p,r} \defeq [\cJ_{p} \cdot \jmath^\vee(t_{p,r}) \cdot \cJ_{p}]$, where $\jmath^\vee(t_{p,r}) = (f_1^*+\cdots + f_r^*)(p)$. Let $\cV_p \defeq [\cJ_p \cdot f_0^*(p) \cdot \cJ_p]$, which acts on $\Pi^{\cJ_p}_p$ via the central action of $p\in\Qp$. 
		\item  Define a Hecke algebra $\cH_p^{\cG,\cP} \defeq \Qp[\cU_{p,r}, \cV_p : b_r \notin \Delta_{\cP}]$. 
	\item A \emph{$\cP$-parahoric $p$-refinement $\tilde\Pi^{\cP} = (\Pi,\alphaG)$ of $\Pi$} is an eigensystem $\alphaG : \cH_p^{\cG,\cP} \to \overline{\Q}_p$ appearing in $\Pi_p^{\cJ_{\cP}}$. We sometimes write $\cP$-refinement for short.
	\end{itemize}
\end{definition}

\subsubsection{Functoriality for parahoric refinements}
Let $P \subset G$ be a spin parabolic, with associated $\cP \subset \cG$. Note that $a_r \not\in \Delta_P \iff b_r \not\in \Delta_{\cP}$, so 
\[
	\cH_p^{\cG,\cP} = \Qp[\cU_{p,r}, \cV_p : a_r \not\in \Delta_P].
\]
We now relate $P$- and $\cP$-refinements. The map $\jmath^\vee : X^\vee \to \cX^\vee$ induces a map 
\begin{equation}\label{eq:jmath hecke}
	\jmath^\vee : \cH_p^P \longrightarrow \cH_p^{\cG,\cP}
\end{equation}
(cf.\ \cite[\S6.4]{BDGJW}). If $1 \leq r \leq n$ with $a_r \not\in \Delta_P$, then $\jmath^\vee$ sends  
\[
	U_{p,r} \longmapsto \cU_{p,r}, \qquad U_{p,2n-r} \longmapsto \cU_{p,r}\cV_p^{n-r}, \qquad U_{p,2n} \longmapsto \cV_p^n.
\] 
For $1 \leq r \leq 2n$, consider the characteristic polynomials
\[
 \cF_{\cG,r}(T) \defeq \det\big(T - \jmath^\vee(U_{p,r})|\Pi_{p}^{\cJ_{\cP}}\big), \qquad	F_{G,r}(T) \defeq \det\big(T - U_{p,r}|\pi_{p}^{J_{P}}\big).
\]

\begin{lemma}
Let $1\leq r\leq 2n$. If $U_{p,r} \in \cH_p^P$, then $\cF_{\cG,r}(T)$ divides $F_{G,r}(T)$.
\end{lemma}
\begin{proof}
Let $\nu_{p,r} \defeq e_1^* + \cdots + e_r^* \in X^\vee$. By \cite[Cor.\ 3.16]{OST19}, we may write
	\begin{equation}\label{eq:Up char poly}
		F_{G,r}(T)  = \prod_{[\sigma] \in \cW_G/\cW_{L_P}}\Big(T-p^{\langle \rho_G, \nu_{p,r}\rangle_G} p^{\langle \UPS^\sigma,\nu_{p,r}\rangle_G}\Big)
	\end{equation}
 where we identify $\UPS^\sigma(t_{p,r}) = \UPS^\sigma(\nu_{p,r}(p)) = p^{\langle\UPS^\sigma,\nu_{p,r}\rangle_G}$ under the natural extension of $\langle-,-\rangle_G$. 
 
 For $\cG$, \cite[Cor.\ 3.16]{OST19} again gives
	\begin{align*}
		\cF_{\cG,r}(T) &= \prod_{\omega \in \cW_{\cG}/\cW_{\cL_{\cP}}}\Big(T-p^{\langle \rho_{\cG}, \jmath^\vee(\nu_{p,r})\rangle_{\cG}} p^{\langle\UPS^\omega_{\cG},\jmath^\vee(\nu_{p,r})\rangle_{\cG}}\Big)\\
					&= \prod_{[\sigma] \in \cW_{G}^0/\cW_{L_P}^0}\Big(T-p^{\langle \rho_{G}, \nu_{p,r}\rangle_{G}} p^{\langle\UPS^\sigma,\nu_{p,r}\rangle_{G}}\Big),
			\end{align*}
	where we identify $\sigma = \jmath(\omega)$, we write $\cW_{L_P}^0 = \jmath(\cW_{\cL_{\cP}})$, and we have used $\jmath(\rho_{\cG}) = \rho_G$, Proposition \ref{prop:jmath vee eq}, and \eqref{eq:pairing jmath}. 
	
	Now note that $\cW_{L_P}^0 = \cW_{L_P}\cap \cW_G^0$, so that $\cW_{G}^0/\cW_{L_P}^0$ is naturally a subset of $\cW_G/\cW_{L_P}$. 
	It follows immediately that $\cF_{\cG,r}$ divides $F_{G,r}$.
\end{proof}

\begin{definition}\label{def:refinement transfer 1}
		Let $P$ be a spin parabolic and $\tilde\pi^P = (\pi,\alpha^P)$ a $P$-refinement. We say $\tilde\pi^P$ is \emph{the functorial transfer of a $\cP$-refinement $\tilde\Pi^{\cP} = (\Pi,\alphaG)$ of $\Pi$} if $\alpha^P$ factors as
		\[
			\cH_p^P \xrightarrow{\ \ \jmath^\vee \ \ } \cH_p^{\cG,\cP} \xrightarrow{ \ \ \alphaG \ \ } \overline{\Q}_p.
		\]
\end{definition}

\begin{proposition}\label{prop:spin factor}
Let $\tilde\pi^P$ be a $P$-refinement. Then
\[
\tilde\pi^P\text{ is $P$-spin} \iff \tilde\pi^P\text{ is the functorial transfer of some $\tilde\Pi^{\cP}$.}
\]
\end{proposition}

\begin{proof}
Let $\tilde\pi^P = (\pi,\alpha^P)$ with $\Psi_{\UPS}^P(\tilde\pi) = [\sigma] \in \cW_G/\cW_{L_P}$. By the proof of the above lemma, and the fact that 
\[
	\alpha^P(U_{p,r}) = \delta_B^{-1/2}\UPS^\sigma(t_{p,r}) = p^{\langle \rho_{G}, \nu_{p,r}\rangle_{G}} p^{\langle\UPS^\sigma,\nu_{p,r}\rangle_{G}},
\]
 we see that $\alpha^P$ factors through $\jmath^\vee$ if and only if $[\sigma]$ is in $\cW_G^0/\cW_{L_P}^0 \subset \cW_G/\cW_{L_P}$; that is, if and only if $\tilde\pi^P$ is a $P$-spin refinement.
\end{proof}

\subsection{Optimally $P$-spin refinements}
Above, we studied when a $P$-refinement was $P$-spin (for the same $P$). An Iwahori $p$-refinement $\tilde\pi$, however, can be $P$-spin for many different $P$'s. 

\begin{definition}\label{def:optimal}
	We say an Iwahori $p$-refinement $\tilde\pi$ is \emph{optimally $P$-spin} if it is $P$-spin and there is no spin $P'\subsetneq P$ such that it is $P'$-spin.
\end{definition}

\begin{corollary}\label{cor:min spin}
	Let $\tilde\pi = (\pi,\alpha)$ be an Iwahori $p$-refinement. 
	\begin{itemize}\s
		\item[(i)] If $P$ and $P'$ are spin parabolics and $\tilde\pi$ is $P$-spin and $P'$-spin, then $\tilde\pi$ is $P\cap P'$-spin. 
		\item[(ii)] $\tilde\pi$ is optimally $\Pmin$-spin for precisely one spin parabolic $B \subseteq \Pmin \subseteq G$.
	\end{itemize}
\end{corollary}
\begin{proof}
	(i) By Proposition \ref{prop:spin factor}, the associated $P$-refinement $\alpha^P$ and $P'$-refinement $\alpha^{P'}$ both factor through spin Hecke algebras; that is, there are maps
	\[
	\alphaG : \cH_{p}^{\cG,\cP} = \Qp[\cU_{p,r}, \cV_p : a_r \not\in \Delta_P] \to \overline{\Q}_p, \qquad 	\alphaGprime : \cH_{p}^{\cG,\cP'} = \Qp[\cU_{p,r},\cV_p : a_r \not\in \Delta_{P'}] \to \overline{\Q}_p
	\]	
	such that
	\[
	\cH_p^P \xrightarrow{\ \ \jmath^\vee \ \ } \cH_{p}^{\cG,\cP} \xrightarrow{ \ \ \alphaG \ \ } \overline{\Q}_p, \qquad 	\cH_p^{P'} \xrightarrow{\ \ \jmath^\vee \ \ } \cH_p^{\cG,\cP'} \xrightarrow{ \ \ \alphaGprime \ \ } \overline{\Q}_p.
	\]
	These extend to a map
	\[
	\alpha^{\cG, \cP\cap\cP'} : \cH_p^{\cG,\cP\cap\cP'} = \Qp[\cU_{p,r}, \cV_p : a_r \not\in \Delta_P\cap\Delta_{P'}] \to \overline{\Q}_p.
	\]	
	Since $\Delta_P \cap \Delta_{P'} = \Delta_{P\cap P}$, we find $\alpha^{P\cap P'}$ factors through $\alpha^{\cG, \cP\cap\cP'}$, whence $\alpha$ is a $P\cap P'$-spin refinement, as required.
	
	(ii) The unique minimum $\Pmin$ is the intersection of all $P$ such that $\tilde\pi$ is $P$-spin. 
\end{proof}

\subsection{$P$-spin refinements combinatorially}
We now introduce a convenient combinatorial description of $p$-refinements. Let $\tilde\pi$ be a $p$-refinement, with $\Psi_{\UPS}(\tilde\pi) = \sigma$. We represent this by the tuple $\tilde\pi \sim \{\sigma(1)\sigma(2)\cdots\sigma(2n)\}$ (for example, if $\sigma$ is the transposition in $\mathrm{S}_{4}$ exchanging 1 and 2, then we represent $\tilde\pi$ as $\{2134\}$). From this, we can easily read off whether $\tilde\pi$ is $P$-spin.

\begin{definition}\label{def:i-spin}
Let $\sigma \in \cW_G$. For $1 \leq r \leq n$, we say $\sigma$ is \emph{$r$-spin} if
\begin{equation}\label{eq:P-spin criterion}
	\forall i \leq r, \ \exists j \geq 2n+1-r \ \text{s.t.} \ \sigma(i) + \sigma(j) = 2n+1.
\end{equation}
We say a $p$-refinement $\tilde\pi$ is $r$-spin if $\sigma = \Psi_{\UPS}(\tilde\pi)$ is $r$-spin. 
\end{definition}

In particular, to be $r$-spin, in the tuple $\{\sigma(1)\cdots\sigma(2n)\}$, it must be possible to pair off the first $r$ numbers and last $r$ numbers into pairs that sum to $2n+1$. For example:
\begin{itemize}
	\item The $p$-refinement $\tilde\pi \sim \{216345\}$ (for $\GL_6$) is 1-spin (since $\sigma(1) + \sigma(6) = 2+5 = 7$). It is not 2-spin, as $\{2,1\}$ and $\{4,5\}$ cannot be paired off into pairs summing to 7. Similarly it is not 3-spin. 
	\item The $p$-refinement $\tilde\pi' \sim \{132456\}$ is 1-spin and 3-spin, but not 2-spin.
\end{itemize}

\begin{definition}\label{def:X_P}
For a spin parabolic $P$, define $X_P \subset \{1,...,n\}$ by
\[
	i \in X_P  \iff a_i \not\in \Delta_P \iff A_i \not\subset \Delta_P.
	\]
\end{definition}

This defines an inclusion-reversing bijection between spin parabolics $P$ and subsets $X_P \subset \{1,...,n\}$. If $X \subset \{1,...,n\}$, we say $\tilde\pi$ is \emph{$X$-spin} if it is $r$-spin for all $r \in X$.

\begin{proposition}\label{prop:P-spin criterion}
Let $P$ be a spin parabolic and $\tilde\pi$ a $p$-refinement. Then
\begin{equation}\label{eq:P-spin criterion 2}
\tilde\pi \text{ is }P\text{-spin} \iff \tilde\pi \text{ is }X_P\text{-spin}.
\end{equation}
It is optimally $P$-spin if and only if $X_P = \bigcup_{\substack{X \subset \{1,...,n\}\\\tilde\pi\text{ is $X$-spin}}}X$ is maximal with this property.
\end{proposition}

\begin{example}Recall $\Pmin$ is the unique spin parabolic such that $\tilde\pi$ is optimally $\Pmin$-spin. The example $\tilde\pi \sim \{216345\}$ above is 1-spin but not 2- or 3-spin, so $X_{\Pmin} = \{1\}$, hence $\Delta_{\Pmin} = \{a_2,a_3,a_4\}$, i.e $\Pmin$ is the (1,4,1)-parabolic. Similarly $\Delta_{P_{\tilde\pi'}} = \{a_2,a_4\}$, so $P_{\tilde\pi'}$ is the (1,2,2,1)-parabolic.
\end{example}

\begin{proof}
For $1 \leq r \leq n$, let $P_r$ be the $(r,2n-2r,r)$-parabolic. Note that 
\[\textstyle
	P = \bigcap_{\substack{r \in \{1,...,n\}\\a_r \not\in \Delta_P}} P_r, \qquad \text{thus} \qquad X_P = \bigcup_{\substack{r \in \{1,...,n\}\\a_r \not\in \Delta_P}} X_{P_r},
\]
 so by Corollary \ref{cor:min spin}(i),  it suffices to show that 
\begin{equation}\label{eq:P-spin criterion 3}
\tilde\pi \text{ is }P_r\text{-spin} \iff \tilde\pi \text{ is }r\text{-spin}.
\end{equation}

First suppose $\tilde\pi$ is $P_r$-spin, so we can write $\Psi_{\UPS}(\tilde\pi) = \zeta\sigma$, with $\zeta \in \cW_G^0$ and $\sigma \in \cW_{L_{P_r}}$. Note $\sigma \in \cW_{L_{P_r}} = \mathrm{S}_r \times \mathrm{S}_{2n-2r} \times \mathrm{S}_r$ preserves $\{1,...,r\}$ and $\{2n+1-r,...,2n\}$, hence $\sigma$ is $r$-spin.

By Lemma \ref{lem:relative}, as $\zeta \in \cW_G^0$, $\sigma(i) + \sigma(j) = 2n+1$ if and only if $\zeta\sigma(i) + \zeta\sigma(j) = 2n+1$, i.e.
	\begin{equation}\label{eq:sigma vs zetasigma}
		\text{($\sigma$ is $r$-spin) $\iff$ ($\zeta\sigma$ is $r$-spin)}.
	\end{equation}
It follows that $\zeta\sigma$, hence $\tilde\pi$, is $r$-spin, giving $\Rightarrow$ in \eqref{eq:P-spin criterion 3}.

Conversely, suppose $\tilde\pi$ is $r$-spin, and let $\sigma = \Psi_{\UPS}(\tilde\pi) \in \cW_G$.

\begin{claim}
Without loss of generality we may assume $\sigma$ preserves $\{1,...,r\}$. 
\end{claim}

\emph{Proof of claim:} We may renormalise $\UPS$ by elements of $\cW_G^0$, as this preserves both being $P_r$-spin (Remark \ref{rem:P-spin ind of UPS}) and $r$-spin (by Remark \ref{rem:change of UPS} and \eqref{eq:sigma vs zetasigma}). We do so repeatedly.

First, without loss of generality we may take 
\begin{equation}\label{eq:subset 1n}
\{\sigma(1),...,\sigma(r)\} \subset \{1,...,n\}.
\end{equation}
 Indeed, if $\sigma(i) > n$ for $1 \leq i \leq r$, then  there exists $2n+1-r\leq j \leq 2n$ such that $\sigma(i) + \sigma(j) = 2n+1$, so that $\sigma(j) \leq n$; and we may exchange $\sigma(i)$ and $\sigma(j)$ by the transposition $(\sigma(i),\sigma(j)) \in \cW_G^0$.

Given \eqref{eq:subset 1n}, after acting by an element of $S_n \subset \cW_G^0$, we may assume $\{\sigma(1),...,\sigma(r)\} = \{1,...,r\}$, proving the claim. 

\medskip

As $\tilde\pi$ is $r$-spin, if $\sigma$ preserves $\{1,...,r\}$, it must also preserve $\{2n+1-r,...,2n\}$. This means $\sigma \in \mathrm{S}_r \times \mathrm{S}_{2n-2r} \times \mathrm{S}_r = \cW_{L_{P_r}}$, so $\sigma$ (hence $\tilde\pi$) is $P_r$-spin, giving $\Leftarrow$ in \eqref{eq:P-spin criterion 3}, and hence \eqref{eq:P-spin criterion 2}.
	 
	The last statement is immediate as $P \leftrightarrow X_P$ is inclusion-reversing.
\end{proof}

\subsection{The function $\gamma_{\tilde\pi}$}

Finally, we introduce one more combinatorial description of being $P$-spin, which will be useful when we study symplectic families.

\begin{definition}\label{def:gamma}
	Let $\tilde\pi$ be a $p$-refinement and $\sigma = \Psi_{\UPS}(\tilde\pi)$. Define an injective map
	\[
	\gamma_{\tilde\pi} : \{1,...,n\} \longhookrightarrow \{1,...,2n\}
	\]
	by setting $\gamma_{\tilde\pi}(i)$ to be the unique integer such that
	\[
	\sigma(i) + \sigma(2n+1-\gamma_{\tilde\pi}(i)) = 2n+1.
	\]
\end{definition}

\begin{lemma}
	The map $\gamma_{\tilde\pi}$ is independent of the choice of $\UPS$ satisfying $\UPS_i \UPS_{2n+1-i} = \eta_p$.
\end{lemma}
\begin{proof}
	If $\UPS'$ is another such choice, there exists $\nu \in \cW_G^0$ such that $\UPS' = \UPS^\nu$. Remark \ref{rem:change of UPS} says $\Psi_{\UPS'}(\tilde\pi) = \nu\Psi_{\UPS}(\tilde\pi) = \nu\sigma$. By Lemma \ref{lem:relative}, $\gamma_{\tilde\pi}$ is unchanged if we replace $\sigma$ with $\nu\sigma$.
\end{proof}

\begin{lemma}\label{lem:spin relations}
	Let $\tilde\pi$ be a $p$-refinement. For $1 \leq r \leq n$,  we have
	\[
		\tilde\pi\text{ is $r$-spin } \iff \text{ $\gamma_{\tilde\pi}$ sends $\{1,...,r\}$ to itself.}
	\]
\end{lemma}
\begin{proof}
	We know $\gamma_{\tilde\pi}$ preserves $\{1,...,r\}$ if and only if $2n+1-r\leq 2n+1-\gamma_{\tilde\pi}(i) \leq 2n$ for all $i$. By definition of $\gamma_{\tilde\pi}$, this is if and only if the sets $\{\sigma(1),...,\sigma(r)\}$ and $\{\sigma(2n+1-r),...,\sigma(2n)\}$ can be paired off into pairs summing to $2n+1$. But this is the definition of $r$-spin.
\end{proof}

\begin{proposition}\label{prop:gamma spin}
	Let $P$ be a spin parabolic, let $\tilde\pi$ be a $p$-refinement, and $\gamma_{\tilde\pi} : \{1,...,n\} \hookrightarrow \{1,...,2n\}$ the function from Definition \ref{def:gamma}. Then
\[
	\tilde\pi\text{ is $P$-spin } \iff \text{ $\gamma_{\tilde\pi}$ preserves $\{1,...,r\}$ whenever $r \in X_P$.}
\]
Additionally, $\tilde\pi$ is optimally $P$-spin if $\gamma_{\tilde\pi}$ does \emph{not} preserve $\{1,...,r\}$ for all $r \not\in X_P$.
\end{proposition}
\begin{proof}
	Both statements follow by combining Proposition \ref{prop:P-spin criterion} with Lemma \ref{lem:spin relations}. 
\end{proof}

\subsection{Non-critical slope bounds}

We conclude this section by showing that non-critical slope conditions (as in Definition \ref{def:non-critical slope}) interact well with the functoriality described above. We will not use this result in this paper, but it is simple to prove and has wider applications. Let $\pi$ be a RASCAR of weight $\lambda$, and $P$ a spin parabolic.

Suppose $\tilde\pi^P = (\pi, \alpha^P)$ is a $P$-spin $P$-refinement, and let $\alpha^{\cP}$ be the corresponding $\cP$-refinement of $\Pi$ furnished by Proposition \ref{prop:spin factor}. The integrally normalised Hecke operators for $\cG$ are defined as $\cU_{p,r}^\circ = \lambda^{\cG}(\jmath^{\vee}(t_{p,r})) \cU_{p,r}$, where $\lambda^{\cG} = \lambda_1f_1 + \cdots + \lambda_nf_n + (\lambda_n + \lambda_{n+1})f_0$ is the weight of $\Pi$, the unique weight with $\jmath(\lambda^{\cG}) = \lambda$. Note for any $1 \leq r \leq n$, we have $\lambda(t_{p,r}) = \lambda^{\cG}(\jmath^\vee(t_{p,r}))$ by  \eqref{eq:pairing jmath}, so $\alpha^{\cP}(\cU_{p,r}^\circ) = \alpha^P(U_{p,r}^\circ)$ for $1 \leq r \leq n$ with $a_r \not\in \Delta_P$.

The small slope bound for $\cG$ is defined (in terms of the root system) in \cite[Def.\ 4.3]{BW20}. Specifically, we need $v_p(\alpha^{\cP}(\cU_{p,r}^\circ)) < \langle \lambda^{\cG}, \beta_r^*\rangle + 1,$ where $\beta_r^*$ is the corresponding simple coroot. Using the $\cG$-root system from Proposition \ref{prop:cG root system}:
\begin{itemize}\s
	\item If $1 \leq r \leq n-1$, then $\beta_r = f_r - f_{r+1}$, and $\beta_r^* = f_r^* - f_{r+1}^*$. So $\langle \lambda^{\cG}, \beta_r^*\rangle + 1 = \lambda_r-\lambda_{r+1} + 1$.
	\item If $r = n$, then $\beta_n = f_n$, and $\beta_n^* = 2f_n^* - f_0^*$. So $\langle \lambda^{\cG}, \beta_n^*\rangle = \lambda_n - \lambda_{n+1} + 1$.
\end{itemize}

Accordingly, we see $\alpha^{\cP}$ is non-$\cP$-critical slope (in the sense of \cite{BW20}) if and only if $v_p(\alpha^{\cP}(\cU_{p,r}^\circ)) < \lambda_r - \lambda_{r+1}+ 1$ whenever $a_r \not\in \Delta_{\cP}$. In particular:

\begin{proposition}\label{lem:non-critical slope transfer}
	Let $\tilde\pi^P$ be a $P$-spin $P$-refinement, corresponding to a $\cP$-refinement $\tilde\Pi^{\cP}$ of $\Pi$. Then $\tilde\pi^P$ is non-$P$-critical slope if and only if $\tilde\Pi^{\cP}$ is non-$\cP$-critical slope.
\end{proposition}

\begin{proof}
	In \eqref{eq:U_p equality} we will show $\alpha^P(U_{p,r}^\circ) = \eta_0(p)^{n-r}\alpha^P(U_{p,2n-r}^\circ)$, for $\eta_0$ a finite order character. As $\eta_0$ has finite order, $v_p(\eta_0(p)) = 0$. Thus for all $r$ with $a_r \not\in \Delta_P$, we have 
	\[
	v_p(\alpha^{\cP}(\cU_{p,r}^\circ)) = v_p(\alpha(U_{p,r}^\circ)) = v_p(\alpha(U_{p,2n-r}^\circ)).
	\]
	But the non-critical slope bounds for these operators are the same for each $r$.
\end{proof}


\part{Dimensions of Symplectic Components}

In part II, we focus on full Iwahori refinements $\tilde\pi$, and study the families through such refinements in the Iwahori eigenvariety. In particular, we conjecture a classification on the dimension of such symplectic families based on the unique spin parabolic $\Pmin$ such that $\tilde\pi$ is optimally $\Pmin$-spin, prove the upper bound, and prove the lower bound in special cases.

\section{The symplectic locus in the eigenvariety}

\subsection{The eigenvariety}
Recall that $K = K^p\Iw$ is Iwahori at $p$, and let $\sW = \sW_K$ be the \emph{weight space} for $G$ of level $K$ (defined e.g.\ in \cite[\S10.1]{BDGJW}). It is a $2n$-dimensional $\Qp$-rigid space.  Let $\cH = \cH^p \cdot \cH_p$, for $\cH^p = \otimes_{v\nmid p\infty} \cH_v$ the tame Hecke algebra of e.g.\ \cite[Def.\ 2.2]{BW20}.

The central object of study in this paper is the \emph{eigenvariety for $G$}.

\begin{theorem}(\cite[Thm.\ 1.1.2]{Han17}). \label{thm:eigenvariety}
	There exists a canonical separated rigid analytic space $\sE^G_K$, and a locally finite map $w:\sE_K^G \to \sW$,  such that the $L$-points $x \in \sE^G$ with $w(x) = \lambda$ biject with finite-slope systems of $\cH$-eigenvalues in the overconvergent cohomology $\hc{\bullet}(S_K,\sD_\lambda)$.
\end{theorem}

Here $S_K$ is the locally symmetric space for $G$ of level $K$ defined in \cite[\S2.1]{Han17}, $\cD_\lambda$ is the local system of locally analytic distributions of weight $\lambda$ defined in \cite[\S2.2]{Han17}, and $\cH_p$ acts on the cohomology via normalised Hecke operators $U_{p,r}^{\circ}$ \cite[Rem.\ 3.13]{BDW20}. 

A point $x \in \sE_K^G$ is \emph{classical (cuspidal)} if the corresponding system of eigenvalues appears in $\pi_x^K$ for a (cuspidal) automorphic representation $\pi_x$ of $G(\A)$ of weight $w(x)$.  Following \cite{Urb11}, \cite[Conj.\ 1.1.5]{Han17} predicts:

\begin{conjecture}
	Every irreducible component of $\sE_K^G$ containing a non-critical cuspidal classical point of regular weight has dimension $n+1$.
\end{conjecture}

\begin{remarks}
	\begin{itemize}
		\item[(i)]	The notion of non-criticality we take here is \cite[Def.\ 3.2.3]{Han17}.
			
			\item[(ii)] By an irreducible component of a rigid space, we mean in the sense of \cite[Def.\ 2.2.2]{Con99}. Whilst the global definition of irreducible components is complicated, the dimension of any such component can be computed locally, where the definition is much more straightforward: if $\mathrm{Sp}(T) \subset \sE_K^G$ is any affinoid piece, the irreducible components of $\mathrm{Sp}(T)$ are of the form $\mathrm{Sp}(T/\mathfrak{p})$, where $\mathfrak{p}$ is a minimal prime ideal of $T$. If $x \in \mathrm{Sp}(T)$ is a given point, corresponding to a maximal ideal $\m_x \subset T$, then the irreducible components containing $x$ are the components $\mathrm{Sp}(T/\mathfrak{p})$ with $\mathfrak{p}\subset \m_x$.

\item [(iii)] In \cite[Prop.\ B.1]{Han17}, Newton has proved that every component as in the conjecture has dimension at least $n+1$. For $\GL_N$, the natural generalisation of this conjecture -- precisely stated in \cite[Conj.\ 1.1.5]{Han17}, and which Hida and Urban style as a `non-abelian Leopoldt conjecture' -- has been proved for $N\leq 4$ in \cite[Thm.\ 4.5.1]{Han17}, noting that $l(\GL_1), l(\GL_2) = 0$ and $l(\GL_3),l(\GL_4) = 1$. For $N \geq 5$, however, it remains wide open.

			\item[(iv)] This conjecture generalises \cite[Conj.\ 1.1]{HidP-ord}, which considers the $p$-ordinary special case from a similar automorphic perspective. In this setting, one has an analogous conjecture on the Galois side due to Tilouine \cite{Til96}, predicting the dimension of certain deformation rings; and under appropriate $R=T$ theorems, the two conjectures become equivalent. To our knowledge, however,  this analogous Galois conjecture is equally wide open.
		
		\end{itemize}

\end{remarks}

\subsection{The classical and symplectic loci}

\begin{definition}
The \emph{classical cuspidal locus} $\sL_K^G \subset \sE_K^G$ is the Zariski closure of the classical cuspidal points in $\sE_K^G$.
\end{definition}

Let $\sW_0 \subset \sW$ be the $(n+1)$-dimensional \emph{pure weight space}, the Zariski-closure of all pure algebraic weights (that is, dominant weights $\lambda = (\lambda_1,...,\lambda_{2n})$ such that $\lambda_1+\lambda_{2n} = \lambda_2+\lambda_{2n-1} = \cdots = \lambda_n + \lambda_{n+1} = \sw(\lambda)$ for some $\sw(\lambda) \in \Z$). 
By \cite[Lem.\ 4.9]{Clo90} any classical cuspidal point $x$ has weight $w(x) \in \sW_0$, so:

\begin{proposition}
	We have $w(\sL_K^G) \subset \sW_0$. 
\end{proposition}

 Through any point $x \in \sL^G_K$, there is a `trivial' 1-dimensional family, corresponding to twists by the norm. (In the introduction, for more conceptual statements, we removed this trivial variation; but here, for cleaner comparisons to other works, we leave it in). 

\begin{definition}
	Let $x \in \sL_K^G$ be a classical cuspidal point.
	\begin{itemize}\s
		\item An irreducible neighbourhood of $\sL_K^G$ through $x$ is \emph{trivial} if it is exactly 1-dimensional, given by twists by the norm and varying over the weight family $\{w(x) + (\kappa,....,\kappa)\}$.
		\item A \emph{classical family} through $x$ is a non-trivial irreducible neighbourhood $\sC \subset \sL_K^G$  of $x$ that itself contains a Zariski-dense set of classical points\footnote{Note that if the classical points are very Zariski-dense in $\sL_K^G$, then $\sC$ will always contain a Zariski-dense set of classical points.}.
	\item We say a point/eigensystem $x \in \sL_K^G$ is \emph{arithmetically rigid} if it cannot be varied in a classical family (i.e.\ it varies only in a trivial family).
	\end{itemize}
\end{definition}

Little is known, or even precisely conjectured, about the classical cuspidal locus. However, there is a folklore expectation that \emph{all} classical families should come from discrete series, in the sense described in \S\ref{sec:intro philosophy}. In particular, all such families should `come from self-duality'.

Given the above expectation, it is natural to study RACARs $\pi$ of $G(\A)$ that are essentially self-dual. Such RACARs are either orthogonal or symplectic. We focus on the latter. 

\begin{definition}
	Define the \emph{symplectic locus} $\sS_K^G \subset \sL_K^G \subset \sE^G_K$ to be the Zariski closure of all classical cuspidal points $x$ such that $\pi_x$ is symplectic. A \emph{symplectic family through $x$} is a non-trivial irreducible neighbourhood of $x$ in $\sS_K^G$ containing a Zariski-dense set of symplectic points.
\end{definition}

Our main result (Theorem \ref{thm:intro 1} of the introduction) gives upper/lower bounds for the dimensions of symplectic families. We state this in the stronger form we prove in \S\ref{sec:main results}.

\subsection{Parabolic weight spaces}\label{sec:parabolic weight spaces}

 To state the more precise version of Theorem \ref{thm:intro 1} that we actually prove, we must introduce parabolic weight spaces.  

Recall that if $P \subset G$ is a parabolic, then the \emph{$P$-parabolic weight space} is the subspace $\sW^P \subset \sW$ of characters that extend to characters of $L_P$. If $\lambda_\pi \in \sW$ is any fixed weight, we denote its coset
\[
\sW_{\lambda_\pi}^P \defeq \lambda_\pi  + \sW^P \subset \sW,
\]
and call it the \emph{$P$-parabolic weight space through $\lambda_\pi$}. These notions are defined in general, and in detail, in \cite[\S3.1]{BW20}. We also define the pure subspaces $\sW_0^P$ and $\sW_{0,\lambda_\pi}^P$ to be the intersections of $\sW^P$ and $\sW_{\lambda_\pi}^P$ with $\sW_0$. We now compute their dimensions.

\begin{lemma}\label{lem:lambda_i-lambda_i+1constant in families}
	If $\lambda_\pi = (\lambda_{\pi,1},...,\lambda_{\pi,2n})$ and $\lambda = (\lambda_1,...,\lambda_{2n})$ are two weights, then $\lambda \in \cW_{\lambda_\pi}^P$ if and only if
	\begin{equation}\label{eq:lambda_i}
		\lambda_i - \lambda_{i+1} = \lambda_{\pi,i} - \lambda_{\pi,i+1} \qquad \forall i \text{ such that }a_i \in \Delta_P.
	\end{equation}
\end{lemma} 

\begin{proof}
	We have $\lambda \in \sW_{\lambda_\pi}^G$ if and only if $\lambda - \lambda_\pi =: \mu= (\mu_1,...,\mu_{2n}) = (\lambda_1 - \lambda_{\pi,1}, ..., \lambda_{2n}-\lambda_{\pi,2n})$ factors through $L_P$. If $L_P = \GL_{m_1} \times \cdots \times \GL_{m_r}$,  then this happens if and only if $\mu$ factors through $\det_1 \times \cdots \times \det_r$.  This is equivalent to $\mu_1 = \cdots = \mu_{m_1}$, ..., $\mu_{2n-m_r+1} = \cdots = \mu_{2n}$ (i.e.\ the $\mu_i$'s are constant in each Levi factor); or in other words, that $\lambda_i - \lambda_{\pi,i} = \mu_i = \mu_{i+1} = \lambda_{i+1}-\lambda_{\pi,i+1}$ for all $i$ with $a_i \in \Delta_P$. Rearranging gives \eqref{eq:lambda_i}.
\end{proof}

In particular, $\lambda_i - \lambda_{i+1}$ can vary in a $P$-parabolic weight family if and only if $a_i \not\in \Delta_P$. For example, in a $B$-parabolic weight family weights can vary in all directions (since $\Delta_B = \varnothing$). If $Q$ is the $(n,n)$-parabolic, then $\Delta_Q = \Delta_B \backslash \{a_n\}$, so in a $Q$-parabolic family $\lambda_1 -\lambda_2$, ..., $\lambda_{n-1}-\lambda_n$ are fixed, $\lambda_n - \lambda_{n+1}$ can vary, and $\lambda_{n+1}-\lambda_{n+2}$, ..., $\lambda_{2n-1} - \lambda_{2n}$ are fixed, so we get the 2-dimensional variation of \cite{BDW20}.

\begin{lemma}\label{lem:parabolic dimension}
	For any spin parabolic $P$ and $\lambda_\pi \in X_0 \subset \sW_0$, we have $\mathrm{dim}(\sW_{0,\lambda_\pi}^P) = \#X_{P} + 1 $.
\end{lemma}

\begin{proof}
	By Lemma \ref{lem:lambda_i-lambda_i+1constant in families}, each $\lambda_i-\lambda_{i+1}$ is constant in $\sW_{\lambda_\pi}^P$ if and only if $a_i \in\Delta_P$, and each such condition decreases  the dimension by 1; so  
	\[
		\mathrm{dim}(\sW_{\lambda_\pi}^P) = 2n - \#\Delta_P = \#\{1 \leq i \leq 2n-1: a_i \not\in\Delta_P\} + 1.
	\]
	 If $\lambda \in \sW_{0,\lambda_\pi}^P$ and $1 \leq i \leq n-1$, we must have $\lambda_i +\lambda_{2n+1-i} = \lambda_{i+1} + \lambda_{2n-i}$, whence $\lambda_{i} - \lambda_{i+1} = \lambda_{2n-i}-\lambda_{2n+1-i}$. (If $i = n$, this still holds; but then it is vacuous). Thus $\mathrm{dim}(\sW_{0,\lambda_\pi}^P) = \#\{1 \leq i \leq n : a_i \not\in \Delta_P\} + 1 = \#X_P + 1$, as required.
\end{proof}

\subsection{Main results/conjecture: the dimension of symplectic families}\label{sec:main results}

We now precisely state the stronger forms of Theorem \ref{thm:intro 1} that we actually prove. 	Let $\pi$ be a RASCAR of weight $\lambda_\pi$ that is spherical and regular at $p$, and let $\tilde\pi$ be an optimally $\Pmin$-spin $p$-refinement.  In \S\ref{sec:obstructions}, we will show the following `upper bound':

\begin{theorem}\label{thm:shalika obstruction}
Any symplectic family $\sC \subset \sS^G_K$ through $\tilde\pi$ is supported over the $\Pmin$-parabolic pure weight space, i.e.\
	\[
	w(\sC) \subset \sW_{0,\lambda_\pi}^{\Pmin}.
	\]
	In particular, $\mathrm{dim}(\sC) \leq \#X_{\Pmin} + 1$.
\end{theorem}

Note we make no non-criticality assumption here. The second statement is Theorem \ref{thm:intro 1}(i); this follows immediately from the first statement, as $w$ is a locally finite map and $\mathrm{dim}(\sW_{0,\lambda_\pi}^{\Pmin}) = \#X_{\Pmin} + 1$ by Lemma \ref{lem:parabolic dimension}.

Our second main result, a stronger form of Theorem \ref{thm:intro 1}(ii), is a `lower bound'. Away from $p$, let $K_1(\pi)^p \subset G(\A_f^{(p)})$ be the Whittaker new level from \cite{JPSS} (see e.g.\ \cite[(7.2)]{BDW20}). Let $K_1(\tilde\pi) = K_1(\pi)^p\mathrm{Iw}_G$. In \S\ref{sec:lower bound}, we prove:

\begin{theorem}\label{thm:lower bound}
	Suppose that $\tilde\pi$ has non-critical slope and $\lambda_\pi$ is regular. Then there is a unique symplectic family through $\tilde\pi$ in $\sE_{K_1(\tilde\pi)}^G$. This family has dimension exactly $\#X_{\Pmin}+1$, and is \'etale over $\sW_{0,\lambda_\pi}^{\Pmin}$ at $\tilde\pi$. 
\end{theorem}

\begin{remark}\label{rem:philosophy}
	Our guiding expectation is that any classical cuspidal family for $G$ should be a transfer of a discrete series family. Which discrete series families, then, should give rise to the families of Theorem \ref{thm:lower bound}?
	
	Since $\tilde\pi$ is an optimally $\Pmin$-spin $p$-refinement, by Proposition \ref{prop:spin factor}, the associated $\Pmin$-refinement $\tilde\pi^{\Pmin}$ is a functorial transfer of a $\cPmin$-refinement $\tilde\Pi^{\cPmin}$ for $\mathrm{GSpin}_{2n+1}$. Then $\tilde\Pi^{\cPmin}$ should vary in a `spin family' $\sC^{\cG}$ over an $(\#X_{\Pmin} + 1)$-dimensional $\cP$-parabolic weight space $\sW^{\cPmin}_{\cG,\lambda_\Pi}$ for $\cG$ (see e.g.\ \cite[Cor.\ 5.16]{BW20}). The map $\jmath$ from \S\ref{sec:structure gspin} isomorphically identifies $\sW_{\cG,\lambda_\Pi}^{\cPmin}$ and $\sW_{0,\lambda_\pi}^{\Pmin}$, and under Langlands functoriality, we expect that the family of Theorem \ref{thm:lower bound} is exactly a transfer to $G$ of the expected spin family $\sC^{\cG}$. 
	
	\label{rem:image is irreducible} 
	
	If we \emph{suppose} the existence of this $p$-adic functoriality map, then Theorem \ref{thm:shalika obstruction} implies that the image of $\sC^{\cG}$ in the Iwahori-level $\GL_{2n}$-eigenvariety is itself an irreducible component of the symplectic locus (that is, it is not a proper subspace of some larger irreducible component).
\end{remark}

Remark \ref{rem:philosophy}, and the philosophy above, suggest the following.

\begin{conjecture}\label{conj:main conjecture}
Let $\tilde\pi$ be a $p$-refined RASCAR of $\GL_{2n}$. Every symplectic family through $\tilde\pi$ is the transfer of a classical parabolic family for $\mathrm{GSpin}_{2n+1}$, varies over $\sW_{0,\lambda_\pi}^{\Pmin}$, and has dimension $\#X_{\Pmin} + 1$.
\end{conjecture}

\subsection{The dimension of classical families}
We have predicted the dimension of \emph{symplectic} families through symplectic $\tilde\pi$. It is desirable to describe more generally the \emph{classical} families. If the following is true, then these questions are equivalent.

\begin{expectation}
	Every classical family through a $p$-refined RASCAR is symplectic. In particular, Conjecture \ref{conj:main conjecture} describes all \emph{classical} families through RASCARs.
\end{expectation}

We do not state this as a formal conjecture; without further evidence, we do not feel confident to rule out `strange' behaviour in higher dimension, where it is harder to classify all the possible lifts from discrete series. For example, we do not rule out classical cuspidal families through $\tilde\pi$ that are lifts from discrete series but not themselves essentially self-dual.

If we restrict to \emph{essentially self-dual} families -- that is, where the essentially self-dual points are Zariski-dense -- then we are on safer ground. Any such family should be symplectic or orthogonal. The symplectic/orthogonal loci should never intersect at classical cohomological points, meaning every classical essentially self-dual family through a $p$-refined RASCAR should be symplectic. 

In the case of $\GL_4$, we expect every classical family to be essentially self-dual, motivating: 

\begin{conjecture}
		Let $\tilde\pi$ be a $p$-refined RASCAR $\pi$ of $\GL_4$. Every classical family through $\tilde\pi$ is the transfer of a classical family on $\GSp_4$, which varies over a $\Pmin$-parabolic weight space and has dimension $\#X_{\Pmin}+1$.
\end{conjecture}

This could be considered a (symplectic) $\GL_4$ analogue of \cite{CM09} (for Bianchi modular forms) and \cite{APS08} (for $\GL_3)$. It seems at least as difficult.


\section{Weight obstructions to symplectic families}\label{sec:obstructions}

Let $\pi$ be a RASCAR of weight $\lambda_\pi$ that is spherical and regular at $p$, and $\tilde\pi$ an optimally $\Pmin$-spin $p$-refinement. In this section, we prove Theorem \ref{thm:shalika obstruction}. In particular, let $\sC$ be any classical symplectic family through $\tilde\pi$. We show that $\sC$ varies only over $\sW_{0,\lambda_\pi}^{\Pmin}$, so  has dimension at most $\#X_{\Pmin} + 1$.

  Recall from \eqref{eq:UPS} that $\pi_p = \Ind_B^G\UPS$ is unramified principal series, where $\UPS$ is a character with $\UPS_i\UPS_{2n+1-i} = \eta_p$ for all $i$. This fixed a bijection $\Psi_{\UPS} : \{p\text{-refinements}\} \isorightarrow \cW_G$ from the set of $p$-refinements to the Weyl group.

\subsection{Identities between Hecke eigenvalues}\label{sec:spin relations}

Given a $p$-refinement $\tilde\pi = (\pi,\alpha)$, we have so far given several criteria for it being $P$-spin. The most natural, in terms of transfer from $\mathrm{GSpin}_{2n+1}$, is conceptually useful but is hard to check. To study the $P$-spin condition in $p$-adic families, we would prefer a characterisation purely in terms of eigenvalues that is intrinsic to $\GL_{2n}$, with no reference to $\mathrm{GSpin}_{2n+1}$. The following is an easy starting point. By \cite[(5.5)]{GR13}, the Shalika character $\eta_p$ is of the form $\eta_0 |\cdot|^{\sw}$, with $\eta_0$ finite order.

\begin{lemma}\label{lem:U_p equality}
	If $\tilde\pi = (\pi,\alpha)$ is $r$-spin, then
	\begin{equation}\label{eq:U_p equality}
		\eta_0(p)^{n-r}\cdot \alpha(U_{p,r}^\circ) =\alpha(U_{p,2n-r}^\circ). 
	\end{equation}
\end{lemma}

\begin{proof}
	By \eqref{eq:P-spin criterion 3}, $\tilde\pi$ is $P_r$-spin for the $(r,2n-2r,r)$-parabolic $P_r$. Applying Proposition \ref{prop:spin factor} to $\tilde\pi^{P_r}$, we see $\alpha^{P_r}$ factors through $\jmath^\vee : \cH_p^{P_r} \to \cH_p^{\cG,\cP_r}$. Note $\jmath^\vee$ sends $U_{p,r} \mapsto \cU_{p,r}$ and $U_{p,2n-r} \mapsto \cV_p^{n-r} \cU_{p,r}$, and that $\cV_p$ acts on $\Pi$ via $\eta(p)$; so this factorisation implies that 
	\[
		\eta_p(p)^{n-r} \cdot \alpha(U_{p,r}) = \alpha(U_{p,2n-r}).
	\]
	To get the claimed relation for the normalised $U_{p,r}^\circ$'s, recall $U_{p,r}^\circ = p^{\lambda_1+\cdots + \lambda_r}U_{p,r}$. We conclude as 
	\[
		p^{\lambda_1+\cdots+\lambda_{n-r}} = p^{\lambda_1+\cdots+\lambda_r}\cdot p^{(n-r)\sw}, \qquad \text{and} \qquad \eta_p(p) = \eta_0(p)p^{-\sw}.\qedhere
	\]
\end{proof}

However, this statement is certainly not if-and-only-if in general. When $r=n$, for example, the statement \eqref{eq:U_p equality} is vacuous, so is satisfied by all $\tilde\pi$. It is desirable to find analogous relations that \emph{exactly} characterise the $r$-spin (hence $P$-spin) refinements.  For this, we will use the canonical function $\gamma_{\tilde\pi} : \{1,...,n\} \hookrightarrow \{1,...,2n\}$ attached to $\tilde\pi$  in Definition \ref{def:gamma}, which -- by Proposition \ref{prop:gamma spin} -- exactly determines when $\tilde\pi$ is $P$-spin.

 For any $p$-refinement $\alpha$, by Proposition \ref{prop:p-refinement}, $\alpha(U_{p,r}^\circ) \neq 0$ for all $r$. We will repeatedly use the following simple observation.

\begin{lemma}\label{lem:UPS in terms of Up circ}
	Let $\tilde\pi$ be a $p$-refinement and let $\sigma = \Psi_{\UPS}(\tilde\pi)$. Then
	\begin{align}\label{eq:theta_r}
		\UPS_{\sigma(r)} (p) &=  p^{\tfrac{2r-2n-1}{2}}\cdot \frac{\alpha(U_{p,r})}{\alpha(U_{p,r-1})}\\
		&= p^{\tfrac{2r-2n-1}{2}}\cdot p^{-\lambda_r}\cdot \frac{\alpha(U_{p,r}^\circ)}{\alpha(U_{p,r-1}^\circ)}.\notag
	\end{align}
	Here, by convention, $\alpha(U_{p,0}) = \alpha(U_{p,0}^\circ) \defeq 1$.
\end{lemma}
\begin{proof}
	The first equality follows from Proposition \ref{prop:p-refinement}(i), which says for any $r$, we have $\alpha(U_{p,r}^\circ) = \delta_B^{-1/2}(t_{p,r}) \cdot p^{\lambda_1+\cdots+\lambda_r} \cdot \UPS_{\sigma(1)}(p)\cdots\UPS_{\sigma(r)}(p).$ The second equality follows as $U_{p,r}^\circ = p^{\lambda_1+\cdots+\lambda_r}U_{p,r}$.
\end{proof}

Crucially, by definition of $\gamma_{\tilde\pi}$,  \eqref{eq:UPS} tells us $\UPS_{\sigma(i)} \cdot \UPS_{\sigma(2n+1-\gamma_{\tilde\pi}(i))} = \eta_p$.  For ease of notation, let $\alpha_r \defeq \alpha(U_{p,r})$.

\begin{lemma}\label{lem:U_p relation}
	For each $1 \leq s \leq n$, we have
\begin{equation}\label{eq:s relation}
	\alpha_s \cdot \prod_{i=1}^s p^{\frac{2n-2\gamma_{\tilde\pi}(i)+1}{2}}\frac{\alpha_{2n+1-\gamma_{\tilde\pi}(i)}}{\alpha_{2n-\gamma_{\tilde\pi}(i)}} = \delta_B^{-1/2}(t_{p,s})\cdot \eta_p(p)^s.
\end{equation}
As $\pi_p$ is regular, $\gamma_{\tilde\pi}$ is the unique map $\{1,...,n\} \hookrightarrow \{1,...,2n\}$ with this property.
\end{lemma}
\begin{proof}
We know $\alpha_s = \delta_B^{-1/2}(t_{p,s})\UPS_{\sigma(1)}(p)\cdots\UPS_{\sigma(s)}(p)$. By Lemma \ref{lem:UPS in terms of Up circ}, the left-hand side is 
	\[
		\delta_B^{-1/2}(t_{p,s}) \UPS_{\sigma(1)}(p) \cdots \UPS_{\sigma(s)}(p) \cdot \prod_{i=1}^s \UPS_{\sigma(2n+1-\gamma_{\tilde\pi}(i))}(p) = \delta_B^{-1/2} \prod_{i=1}^s \big[\UPS_{\sigma(i)}\UPS_{\sigma(2n+1-\gamma_{\tilde\pi}(i))}\big](p).
	\]
	 We deduce \eqref{eq:s relation} since $\UPS_{\sigma(i)}\UPS_{\sigma(2n+1-\gamma_{\tilde\pi}(i))} = \eta_p$ for each $i$.
	 
	 \medskip
	 
	 It remains to prove uniqueness. Suppose $\gamma: \{1,...,n\} \hookrightarrow \{1,...,2n\}$ is another function such that \eqref{eq:s relation} holds (with $\gamma$ in place of $\gamma_{\tilde\pi}$) for $1 \leq s \leq n$. Regularity of $\pi_p$ means all the $\UPS_i(p)$'s are distinct. Dividing \eqref{eq:s relation} for $s$ by \eqref{eq:s relation} for $s-1$ gives \[
	 	\UPS_{\sigma(s)}\cdot \UPS_{\sigma(2n+1-\gamma(s))}(p) = \eta_p(p) =  \UPS_{\sigma(s)}\cdot \UPS_{\sigma(2n+1-\gamma_{\tilde\pi}(s))}(p).
	 \]
	 Regularity implies $\sigma(2n+1-\gamma(s)) = \sigma(2n+1-\gamma_{\tilde\pi}(s))$, so $\gamma(s) = \gamma_{\tilde\pi}(s)$, and $\gamma = \gamma_{\tilde\pi}$.
\end{proof}

\begin{proposition}\label{prop:spin relations circ}
	For each $1 \leq s \leq n$, we have
	\begin{equation}\label{eq:hecke relations circ}
		\alpha(U_{p,s}^\circ) \cdot \prod_{i=1}^s p^{\frac{2n-2\gamma_{\tilde\pi}(i)+1}{2}}\cdot p^{\lambda_{\gamma_{\tilde\pi}(i)}-\lambda_i} \cdot\frac{\alpha(U_{p,2n+1-\gamma_{\tilde\pi}(i)}^\circ)}{\alpha(U_{p,2n-\gamma_{\tilde\pi}(i)}^\circ)} = \delta_B^{-1/2}(t_{p,s})\cdot \eta_0(p)^s.
	\end{equation}
If $\pi_p$ is regular, then $\gamma_{\tilde\pi}$ is the unique map $\{1,...,n\} \hookrightarrow \{1,...,2n\}$ with this property.
\end{proposition}

\begin{proof}
	The direct analogue of \eqref{eq:s relation} with normalised eigenvalues is 
	\[
	\alpha(U_{p,s}^\circ) \cdot \prod_{i=1}^sp^{\frac{2n-2\gamma_{\tilde\pi}(i)+1}{2}}\cdot p^{-\lambda_{2n+1-\gamma_{\tilde\pi}(i)}} \cdot \frac{\alpha(U_{p,2n+1-\gamma_{\tilde\pi}(i)}^\circ)}{\alpha(U_{p,2n-\gamma_{\tilde\pi}(i)}^\circ)} = p^{\lambda_1+\cdots+\lambda_s} \cdot \delta_B^{-1/2}(t_{p,s})\cdot\eta_p(p)^s.
	\]
	To get the stated form, we use that $\lambda_{\gamma_{\tilde\pi}(i)} + \lambda_{2n+1-\gamma_{\tilde\pi}(i)} = \sw$ and $\eta_p(p) = \eta_0(p)p^{-\sw}$.
\end{proof}

\subsection{Zariski-density of $p$-refined spherical points}

In our proofs of Theorems \ref{thm:shalika obstruction} and \ref{thm:lower bound}, we will require a Zariski-dense set of classical points with good properties. This is furnished by the following. Note we do \emph{not} require RASCARs here, only RACARs.

\begin{proposition}\label{prop:ZD spherical}
	Let $\sC \subset \sL_K^G$ be a classical family containing a classical point corresponding to a $p$-refined RACAR that is spherical and regular at $p$. Then $\sC$ contains a Zariski-dense set of classical points corresponding to $p$-refined RACARs that are spherical and regular at $p$.
\end{proposition}

\begin{proof}
	Any classical point $y \in \sC$ corresponds to an eigensystem $\alpha_y$ appearing in a RACAR $\pi_y$ such that $\pi_{y,p}$ is Iwahori-spherical (admits non-zero Iwahori-invariant vectors). By \cite[Prop.\ 2.6]{Cas80}, any such $\pi_{y,p}$ is a $\GL_{2n}(\Qp)$-submodule of an unramified principal series representation $\Ind_B^G\UPS_y$, for an unramified character $\UPS_y = (\UPS_{y,1},...,\UPS_{y,2n})$. First we prove that $\Ind_B^G\UPS_y$ is irreducible for a Zariski-dense set of $y \in \sC$, as then $\pi_{y,p} = \Ind_B^G\UPS_y$ is spherical.
	
	For convenience, drop the subscript $y$. Let $\sigma = \Psi_\UPS(\tilde\pi)$; without loss of generality, replace $\UPS$ with $\UPS^\sigma$ and assume $\sigma = \mathrm{id}$. By \cite[Thm.\ 4.2]{BZ77}, $\Ind_B^G\UPS$ is reducible if and only if there exist $r,s$ such that $\UPS_r = \UPS_{s}|\cdot|$. As the $\UPS_i$ are unramified, this happens if and only if 
	$p \cdot \UPS_r(p) = \UPS_s(p).$
	Using Lemma \ref{lem:UPS in terms of Up circ} with $\sigma = 1$, this is equivalent to
	\begin{equation}\label{eq:reduced criterion}
		p\cdot p^{r-s}\cdot p^{\lambda_s-\lambda_r} \cdot \alpha(U_{p,r}^\circ)\cdot \alpha(U_{p,s-1}^\circ) = \alpha(U_{p,s}^\circ) \cdot \alpha(U_{p,r-1}^\circ).
	\end{equation}
	Since the $\alpha(U_{p,i}^\circ)$ are all analytic and non-zero on $\sC$, the locus $\sC_{r,s}$ in $\sC$ where \eqref{eq:reduced criterion} is satisfied is a Zariski-closed subspace (with weight support only over subsets where $\lambda_r-\lambda_s$ is constant). However, by assumption $\sC$ contains a $p$-refined spherical point, so $\sC_{r,s} \neq \sC$, whence $\sC_{r,s} \subset \sC$ is a proper subspace of smaller dimension. 
	
	Any classical point $y$ where $\Ind_B^G \UPS_y$ is reducible must live in $\bigcup_{r\neq s}\sC_{r,s}$. Since there are only finitely many possible pairs $(r,s)$, this union is a proper subspace of $\sC$ of smaller dimension. It follows that $\Ind_B^G\UPS_y$ is irreducible for a Zariski-dense set of $y$, and each of these $y$ corresponds to a $p$-refined $p$-spherical RACAR.
	
	\medskip
	
	It remains to check a Zariski-dense subset of these $y$ are regular. Note such a $y$ is not regular, then there exist $r \neq s$ such that $\UPS_r(p) = \UPS_s(p)$. Arguing as above, this happens if and only if
	\[
	p^{r-s}\cdot p^{\lambda_s-\lambda_r} \cdot \alpha(U_{p,r}^\circ)\cdot \alpha(U_{p,s-1}^\circ) = \alpha(U_{p,s}^\circ) \cdot \alpha(U_{p,r-1}^\circ),
	\]
	again cutting out a closed subspace in $\sC$. We conclude that there are a Zariski-dense set of $p$-regular points as before.
\end{proof}

\begin{remark}
	In any positive-dimensional component of $\sC_{r,s}$ we must have $\lambda_r-\lambda_s$ constant. It follows that any everywhere-ramified family must vary over some parabolic weight space $\sW_{0,\lambda}^{P}$ for some non-minimal $B \subsetneq P \subset G$. In particular, we recover that any classical family over the full pure weight space $\sW_0$ contains a Zariski-dense set of spherical points.
\end{remark}

\subsection{Proof of Theorem \ref{thm:shalika obstruction}}

Let $\tilde\pi$ be an optimally $P$-spin $p$-refined RASCAR such that $\pi_p$ is spherical and regular, and let $\sC$ be a symplectic family though $\tilde\pi$. To prove Theorem \ref{thm:shalika obstruction}, we must show that $w(\sC) \subset \sW_{0,\lambda_\pi}^P$.

Let $\fX$ be the set of classical points in $\sC$ that correspond to $p$-refined RASCARs $\tilde\pi_y$ such that $\pi_{y,p}$ is spherical and regular. By Proposition \ref{prop:ZD spherical}, the set $\fX$ is Zariski-dense in $\sC$.

For each $y\in \fX$, let $\gamma_y : \{1,...,n\} \hookrightarrow \{1,...,2n\}$ be the function for $\tilde\pi_y$ from Definition \ref{def:gamma}. 

\begin{lemma}\label{prop:transfer}
	The function $\gamma_y$ is constant as $y$ varies in $\fX$.  
\end{lemma}

\begin{proof}
	There are only finitely many functions $\gamma : \{1,...,n\} \hookrightarrow \{1,...,2n\}$, so there must exist such a function $\gamma$ and a Zariski dense subset $\fY \subset \fX \subset \sC$ such that $\gamma_z = \gamma$ for all $z \in \fY$.
	
	By Proposition \ref{prop:spin relations circ}, at every $y$ in $\fY$, the Hecke relations 
	\begin{equation}\label{eq:obstruction}
		\alpha_y(U_{p,s}^\circ) \cdot \prod_{i=1}^s p^{\frac{2n-2\gamma_{\tilde\pi}(i)+1}{2}}\cdot p^{\lambda_{y,\gamma(i)}-\lambda_{y,i}} \cdot\frac{\alpha_y(U_{p,2n+1-\gamma(i)}^\circ)}{\alpha_y(U_{p,2n-\gamma(i)}^\circ)} = \delta_B^{-1/2}(t_{p,s})\cdot \eta_0(p)^s
	\end{equation}
	are satisfied for all $1 \leq s \leq n$, where $w(y) = \lambda_y$. Since $U_{p,r}^\circ$ defines an analytic function on $\sC$, and these relations hold for the Zariski-dense $\fY$, they hold over all of $\sC$. In particular, they hold at every point $y \in \fX$. Since the points in $\fX$ are regular, the unicity statement in Proposition \ref{prop:spin relations circ} says $\gamma_y = \gamma$ for all $y \in \fX$. 
\end{proof}

\begin{lemma}\label{lem:every point P-spin}
	Every point $y \in \fX$ is optimally $P$-spin.
\end{lemma}
\begin{proof}
	Let $y \in \fX$, and let $P_y$ be the unique spin parabolic such that $\tilde\pi_y$ is optimally $P_y$-spin. By Proposition \ref{prop:gamma spin}, $P_y$ is determined by the function $\gamma_y$. By Lemma \ref{prop:transfer}, the function $\gamma_y$ is constant over $\fX$; thus $P_y$ is also constant over $\fX$. But $\fX$ contains $\tilde\pi$, which by assumption is optimally $P$-spin. Thus $P_y = P$ for all $y \in \fX$.
\end{proof}

\begin{lemma}\label{lem:weight obstruction}
	For $1 \leq i \leq n$, if $a_i \in \Delta_{P}$, then $\lambda_{y,i} - \lambda_{y,i+1}$ is constant as $y$ varies in $\fX$. 
\end{lemma}
\begin{proof}
	Let $\gamma$ be the function from the proof of Lemma \ref{prop:transfer}. We showed that the relation \eqref{eq:obstruction} holds over all of $\sC$, and for all $1 \leq s \leq n$. As the $\alpha_y(U_{p,r}^\circ)$ vary analytically with $y$, for this to be true for all $s$, the term $p^{\lambda_{y,\gamma(i)} - \lambda_i}$ must be constant for all $1 \leq i \leq n$. This forces $\lambda_{y,\gamma(i)} - \lambda_{y,i}$ to be constant.
	
	Now, suppose $a_i \in \Delta_P$. Then $i \not\in X_P$. Now, since the points of $\fX$ are optimally $P$-spin, by Proposition \ref{prop:gamma spin} we know that $\gamma$ does not preserve $\{1,...,i\}$. In particular, there exists some $m \in \{1,...,i\}$ such that $\gamma(m) > i$. Also, by dominance, we have $\lambda_{m} \geq \lambda_{i} \geq \lambda_{i+1} \geq \lambda_{\gamma(m)}$. Thus if $\lambda_{y,\gamma(m)} - \lambda_{y,m}$ is constant, as $y$ varies over $\fX$, then so is $\lambda_{y,i} - \lambda_{y,i+1}$.
\end{proof}

Finally we prove Theorem \ref{thm:shalika obstruction}. If $a_i \in \Delta_P$, then either:
\begin{itemize}\s
\item[(1)] $1 \leq i \leq n$. Lemma \ref{lem:weight obstruction}, and Zariski-density of $\fX$, imply $\lambda_i-\lambda_{i+1}$ is constant over $w(\sC)$. 
\item[(2)] or $n+1 \leq i \leq 2n-1$; then $1 \leq 2n-i \leq n$. As $P$ is a spin parabolic $a_i \in \Delta_P$ if and only if $a_{2n-i} \in \Delta_P$, so by (1) $\lambda_{2n-i} - \lambda_{2n-i+1}$ is constant.  As $w(\sC)$ is in the pure weight space, this implies $\lambda_{i}-\lambda_{i+1}$ is constant. 
\end{itemize}
By Lemma \ref{lem:lambda_i-lambda_i+1constant in families}, this means that $w(\sC) \subset \sW_{0,\lambda_\pi}^P$, as claimed. \qed



\section{Existence of $P$-spin families}\label{sec:lower bound}

We have obtained an upper bound on the dimension of symplectic families. We now prove Theorem \ref{thm:lower bound}, constructing families realising this bound through non-critical slope refinements.

\subsection{$B$-spin families}
Let $\pi$ be a RASCAR of regular weight that is spherical and regular at $p$. Let $K_1(\tilde\pi)$ be as before Theorem \ref{thm:lower bound}. In \cite{BDW20} and \cite{BDGJW} we proved:

\begin{theorem}\label{thm:BDGJW}
Let $\tilde\pi$ be a non-critical $B$-spin refinement. There is a unique family $\sC$ through $\tilde\pi$ in $\sE^G_{K_1(\tilde\pi)}$ that varies over the pure weight space $\sW_0$. Moreover $\sC$ is an $(n+1)$-dimensional classical symplectic family \'etale over $\sW_0$ at $\tilde\pi$ in which the classical symplectic points are very Zariski-dense. 
\end{theorem}

Recall we say a subset $X \subset \sC$ is \emph{very Zariski-dense} if for every $x \in X$, there is a basis of affinoid neighbourhoods $V \subset \sC$ of $x$ such that $X \cap V$ is Zariski-dense in $V$.

\begin{proof}
When $K_1(\pi) = G(\widehat{\Z})$, this is \cite[Thm.\ 13.6]{BDGJW}. One can treat general $K_1(\pi)$ following exactly the strategy of \cite[\S7.5,7.6]{BDW20}. 
\end{proof}

	\begin{lemma}\label{lem:every point spherical}
		We may shrink $\sC$ so that every classical point $y \in V$ corresponds to a $B$-spin $p$-refined RASCAR $\tilde\pi_y$ such that $\pi_{y,p} = \Ind_B^G \UPS_y$ is a regular and spherical, with $\Psi_{\UPS_y}(\tilde\pi_y) = \Psi_{\UPS}(\tilde\pi)$.
	\end{lemma}
In other words: `each classical point is a $p$-refined $p$-spherical RASCAR, and for each such point, and all the refinements are in the same position in the Weyl group.' 

\begin{proof}
By Proposition \ref{prop:ZD spherical} and its proof, all the classical points corresponding to RACARs that are ramified at $p$ live inside a proper closed subspace of the eigenvariety, and since $x$ is not in this closed subspace, we can shrink the neighbourhood $\sC$ to avoid it completely. Then every classical $y$ is unramified principal series at $p$.

In this $\sC$, every $y$ is (optimally) $B$-spin by Lemmas \ref{prop:transfer} and \ref{lem:every point P-spin}; so $\Psi_{\UPS_y}(\tilde\pi_y) \in \cW_G^0$. By Remarks \ref{rem:change of UPS} and \ref{rem:P-spin ind of UPS}, we can thus conjugate $\UPS_y$ so that $\Psi_{\UPS_y}(\tilde\pi_y) = \Psi_{\UPS}(\tilde\pi)$.
\end{proof}

\begin{wrapfigure}[13]{r}{0.4\textwidth}
	
	\vspace{-14pt}
	
	\hspace{-10pt}
	\includegraphics[width=6cm]{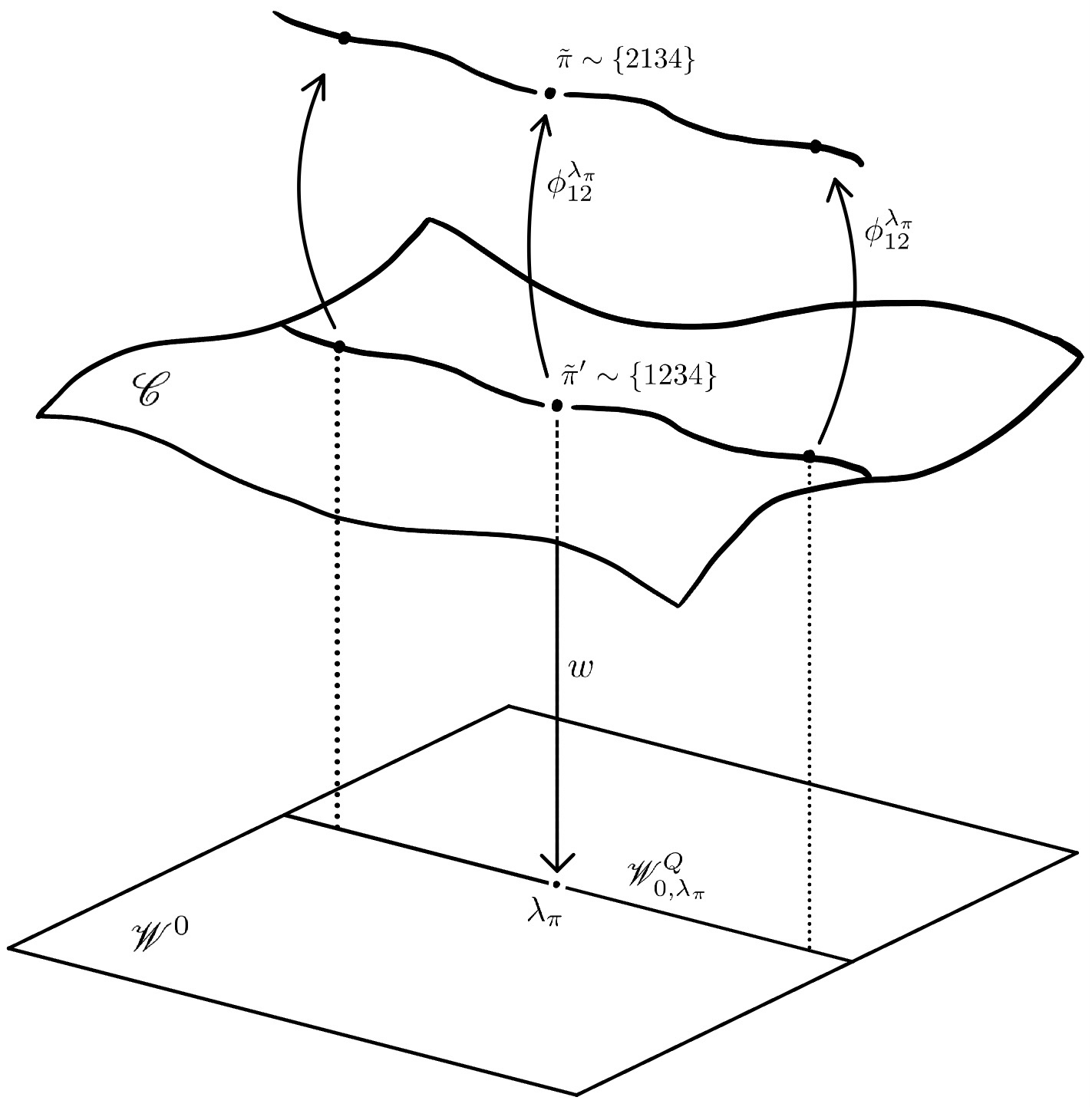}
\end{wrapfigure}

\subsection{Refinement-switching} \label{sec:relating eigenvalues}

To produce $P$-spin families, we take the part of the $B$-spin family supported over the $P$-parahoric weight space, and systematically switch between refinements for each classical point in the family. For $\GL_4$, this is pictorially represented in the figure right, and we shall now explain the notation.

To enact this strategy, we need to able to pass between optimally $P$-spin and optimally $B$-spin refinements, and to relate eigenvalues as we do so.  Recall the notion of being $r$-spin from Definition \ref{def:i-spin}, and $X$-spin from Definition \ref{def:X_P}. The following lemma shows you can always `improve' the spin-ness with a controlled transposition.

\begin{lemma}\label{lem:adjacent spin 1}
	Suppose $\tilde\pi$ is optimally $X$-spin, for $X \subset \{1,...,n\}$.
	\begin{enumerate}[(i)]\s
		\item Let $1 \leq i \leq n-1$, and suppose: (a) $(i-1) \in X$ or $i=1$, and (b) $i \notin X$. 
		Let 
		\[
			k \defeq \left\{\begin{array}{cl} 2n-i &: i-1\text{ is maximal in }X,\\
				\mathrm{min}\{i' \in X: i' > i-1\} &: \text{else}.
				\end{array}\right.
		\]
		
		Then there exists $i+1 \leq j \leq k$ such that the $p$-refinement $\tilde\pi'$ with
		\[
		\Psi_{\UPS}(\tilde\pi') = \Psi_{\UPS}(\tilde\pi) \cdot (i,j)
		\]
		is $X\cup\{i\}$-spin. 
		
		\item If $\tilde\pi$ is $(n-1)$-spin, then it is $n$-spin (i.e.\ if $n-1 \in X$, then $n \in X$). 
	\end{enumerate}
\end{lemma}

\begin{proof}
	(i)	Let $\sigma = \Psi_{\UPS}(\tilde\pi)$, and let $j$ be the unique integer such that $\sigma(j) + \sigma(2n+1-i) = 2n+1$. 
	
	\medskip
	
	\textbf{Step 1: Inequalities on $j$.} 	For any $r \in X$, since $\tilde\pi$ is $r$-spin, the sets  $\{\sigma(1),...,\sigma(r)\}$ and $\{\sigma(2n+1-r),...,\sigma(2n)\}$ pair off so that the sum of each pair is $2n+1$. In particular,
	\begin{myquote}
		$\sigma(j)$ is in one of these two sets $\iff$ $\sigma(2n+1-i)$ is in the other. 
	\end{myquote}
	Then:
	\begin{itemize}\s
		\item Apply $(\dagger)$ with $r = i-1$. As $\sigma(2n+1-i) \not\in \{\sigma(2n+2-i),...,\sigma(2n)\}$, we know $\sigma(j) \notin \{\sigma(1),...,\sigma(i-1)\}$. So  $j \not\in \{1,...,i-1\}$, i.e.\ $i \leq j$. 
		
		\item As $\tilde\pi$ is $(i-1)$-spin but \emph{not} $i$-spin, $\sigma(i) + \sigma(2n+1-i) \neq 2n+1$, so $j \neq i$; hence $i+1 \leq j$.

		\item As $i \leq n-1$, we have $\sigma(2n+1-i) \not\in \{\sigma(1),...,\sigma(i-1)\}$, so $j \leq 2n+1-i$. But $j \neq 2n+1-i$ clearly, so $j \leq 2n-i$ (always).
		\item If $i-1$ is maximal in $X$, then $k=2n-i$ and we are done. Otherwise $k$ is the next smallest element of $X$; as  $i < k$ and $\sigma$ is $k$-spin, we have $\sigma(2n+1-i) \in \{\sigma(2n+1-k),...,\sigma(2n)\}$, so $(\dagger)$ implies $j \leq k$.
	\end{itemize}

	\medskip
	
	\textbf{Step 2: $\tilde\pi'$ is $X$-spin.}	Now, let $\zeta = (i,j)$. If $r \in X$, then either we have 
	\[
	r < i \text{ and } j \leq r < 2n+1-r, \qquad\text{ or } \qquad \text{both }i,j \leq r.
	\]
	Either way, $\zeta$ preserves  $\{1,...,r\}$ and $\{2n+1-r,...,2n\}$. In particular, we have
	\begin{align*}
		\{\sigma(1),...,\sigma(r)\} &= \{\sigma\zeta(1),...,\sigma\zeta(r)\},\\
		\{\sigma(2n+1-r),...,\sigma(2n)\} &= \{\sigma\zeta(2n+1-r),...,\sigma\zeta(2n)\},
	\end{align*}
	so $\sigma\zeta$ is $r$-spin since $\sigma$ is. Since this is true of all $r \in X$, we conclude $\sigma\zeta = \Psi_{\UPS}(\tilde\pi')$ is $X$-spin.
	
	\medskip
	
	\textbf{Step 3: $\tilde\pi'$ is $X\cup \{i\}$-spin.} 	 By above, $\sigma\zeta$ is $(i-1)$-spin. Moreover, by construction $\sigma\zeta(i) + \sigma\zeta(2n+1-i) = 2n+1$, so additionally $\sigma\zeta$ is $i$-spin. As it is $X$-spin and $i$-spin, $\sigma\zeta = \Psi_{\UPS}(\tilde\pi')$ is $X\cup\{i\}$-spin, as claimed.
	
	\bigskip

	(ii) If $\tilde\pi$ is $(n-1)$-spin, then by definition, for each $r \leq n-1$, there is $s \geq n+2$ such that $\sigma(r) + \sigma(s) = 2n+1$. This accounts for $n-1$ of the $n$ pairs with this property, and forces $\sigma(n) + \sigma(n+1) = 2n+1$ to be the $n$th and last. Thus $\tilde\pi$ is also $n$-spin.
\end{proof}

We now relate the Hecke eigenvalues of $\tilde\pi$ and $\tilde\pi'$ from the previous lemma. Recall that by Proposition \ref{prop:p-refinement}, since $\UPS_i(p) \neq 0$ for all $i$, $\alpha$ is finite slope, i.e.\ $\alpha(U_{p,i}^\circ) \neq 0$ for all $i$. 

\begin{lemma}\label{lem:adjacent spin 2}
	Let $\tilde\pi = (\pi,\alpha)$ and $\tilde\pi' = (\pi,\alpha')$ be two $p$-refinements, with 
	\[
	\Psi_{\UPS}(\tilde\pi') = \Psi_{\UPS}(\tilde\pi) \cdot (i,j),
	\]
	where $(i,j) \in \mathrm{S}_{2n}$ is a transposition with $i < j$. Then for all $r$, 
	\[
	\alpha'(U_{p,r}^\circ) = \left\{\begin{array}{cc} p^{j-i}p^{\lambda_i-\lambda_j} \frac{\alpha(U_{p,j}^\circ)}{\alpha(U_{p,j-1}^\circ)} \cdot \frac{\alpha(U_{p,i-1}^\circ)}{\alpha(U_{p,i}^\circ)}\cdot\alpha(U_{p,r}^\circ)&: i \leq r < j\\
		\alpha(U_{p,r}^\circ) &: \text{otherwise},
	\end{array}\right.
	\]
	where $\pi$ has weight $\lambda = (\lambda_1,...,\lambda_{2n})$ and we use the shorthand that ``$\alpha(U_{p,0}^\circ)$'' $:= 1$.
\end{lemma}

\begin{proof}
	Let $\sigma = \Psi_{\UPS}(\tilde\pi)$. By Proposition \ref{prop:p-refinement} the definition of $U_{p,r}^\circ$ we have
	\[
	\alpha(U_{p,r}^\circ) = \delta_B^{-1/2}(t_{p,r}) \cdot p^{\lambda_1 + \cdots + \lambda_r} \cdot \UPS_{\sigma(1)}(p)\cdots\UPS_{\sigma(r)}(p).
	\]
	Now $\alpha'(U_{p,r}^{\circ})$ can be described in the same way, except with $\sigma$ replaced with $\sigma(i,j)$. When $r < i$ or $r \geq j$, this is identical to $\alpha(U_{p,r}^\circ)$; when $i \leq r < j$, this means $\UPS_{\sigma(i)}(p)$ is replaced by $\UPS_{[\sigma(i,j)](i)}(p) = \UPS_{\sigma(j)}(p)$ in the product. Via Lemma \ref{lem:UPS in terms of Up circ}, in this case
	\begin{align*}
		\alpha'(U_{p,r}^\circ) &= \alpha(U_{p,r}^\circ) \cdot \UPS_{\sigma(j)}(p) \cdot \UPS_{\sigma(i)}(p)^{-1}\\
		&= \alpha(U_{p,r}^\circ) \cdot \left[p^{-\lambda_j}p^{(2j-2n-1)/2}\frac{\alpha(U_{p,j}^\circ)}{\alpha(U_{p,j-1}^\circ)}\right]\cdot \left[p^{-\lambda_i}p^{(2i-2n-1)/2}\frac{\alpha(U_{p,i}^\circ)}{\alpha(U_{p,i-1}^\circ)}\right]^{-1},
	\end{align*}
	which simplifies to the claimed expression.
\end{proof}

We will use Lemma \ref{lem:adjacent spin 2} to define maps between families on the eigenvariety. This requires adding inverses to the Hecke algebra.

\begin{definition}
Let $\cH^{\mathrm{frac}} = \cH_p^{\mathrm{frac}} \cdot \cH^p$, where
\[
	\cH_p^{\mathrm{frac}} \defeq \Qp[U_{p,r}^\circ, (U_{p,r}^\circ)^{-1} : 1 \leq r \leq 2n].
\]
\end{definition}

Now fix $K = K_1(\tilde\pi)$ from before Theorem \ref{thm:lower bound}. Let $\sE = \sE_K^G$ from Theorem \ref{thm:eigenvariety}, defined by the action of $\cH$ on overconvergent cohomology. Let also $\sE' = \sE'_K$ be the eigenvariety defined by the same eigenvariety datum, but using instead the action of $\cH^{\mathrm{frac}}$ on the \emph{finite-slope} overconvergent cohomology.

\begin{lemma}\label{lem:E = E'}
We have $\sE = \sE'$.
\end{lemma}
\begin{proof}
Both eigenvarieties are defined by writing down local pieces $\sE_{\Omega,h} = \mathrm{Sp}(\bT_{\Omega,h})$ and $\sE'_{\Omega,h} = \mathrm{Sp}(\bT_{\Omega,h}')$, where $\bT_{\Omega,h}$ (resp.\ $\bT_{\Omega,h}'$) is the image of $\cH \otimes \cO_\Omega$ (resp $\cH^{\mathrm{frac}}\otimes \cO_\Omega$) in  $\mathrm{End}_{\cO_\Omega}(\hc{\bullet}(S_K,\sD_\Omega)^{\leq h})$. As each $U_{p,r}^\circ$ acts invertibly on the slope $\leq h$ cohomology (see e.g.\ \cite[\S2.3.1]{Urb11}), the image of $U_{p,r}^\circ$ in $\bT_{\Omega,h}$ is invertible; and hence $\bT_{\Omega,h} = \bT_{\Omega,h}'$, so $\sE_{\Omega,h} = \sE_{\Omega,h'}$.

Both $\sE$ and $\sE'$ are defined by the same gluing of the same local pieces, so they are equal.
\end{proof}

\begin{definition}
For $\lambda = (\lambda_1,...,\lambda_{2n}) \in X^*(T)$, and $i < j$, define a map
\[
\phi_{ij}^\lambda : \cH \longrightarrow \cH^{\mathrm{frac}}
\]
to be the identity map on all operators away from $p$, and at $p$ by
\[
	\phi_{ij}^\lambda(U_{p,r}^\circ) = \left\{\begin{array}{cc} p^{j-i}p^{\lambda_i-\lambda_j} \frac{U_{p,j}^\circ}{U_{p,j-1}^\circ} \cdot \frac{U_{p,i-1}^\circ}{U_{p,i}^\circ}\cdot U_{p,r}^\circ&: i \leq r < j\\
		U_{p,r}^\circ &: \text{otherwise},
	\end{array}\right.
	\]
\end{definition}

\begin{lemma}\label{lem:Hecke map}
Let $\pi$ have weight $\lambda_\pi$, and let $\tilde\pi = (\pi,\alpha)$ and $\tilde\pi' = (\pi,\alpha')$ be $p$-refinements with 
\[
\Psi_\UPS(\tilde\pi') = \Psi_{\UPS}(\tilde\pi)\cdot (i,j)
\]
as elements of $\cW_G$. Then $\alpha' = \alpha \circ \phi_{ij}^{\lambda_\pi}$ and $\alpha' \circ \phi_{ij}^{\lambda_\pi} = \alpha$.
\end{lemma}

\begin{proof}
 Note also $\Psi_{\UPS}(\tilde\pi') \cdot (i,j) = \Psi_{\UPS}(\tilde\pi)$. Both statements are then direct from Lemma \ref{lem:adjacent spin 2}.
\end{proof}

\subsection{From $P$-spin to $B$-spin}
Let $\tilde\pi = (\pi,\alpha)$ be an optimally $P$-spin non-critical slope refinement. 

\begin{proposition}\label{prop:P to B}
	\begin{itemize}\s
	\item[(i)] There exists an element $\tau = (i_1,j_1) \cdots (i_k,j_k) \in \cW_G$, where $k \leq n - \#X_P$, and a $B$-spin $p$-refinement $\tilde\pi' = (\pi, \alpha')$ with 
	\[
		\Psi_{\UPS}(\tilde\pi') = \Psi_{\UPS}(\tilde\pi) \cdot \tau.
	\]

	\item[(ii)] The refinement $\tilde\pi'$ from (i) has non-critical slope.
	
	\item[(iii)] We have  $\alpha' \circ \phi_\tau^{\lambda_\pi} = \alpha$, where for any classical $\lambda$ we let
	\[
		\phi_\tau^{\lambda} \defeq \phi_{i_k,j_k}^{\lambda} \circ \cdots \circ \phi_{i_1,j_1}^{\lambda} : \cH \longrightarrow \cH^{\mathrm{frac}}.
	\]

	\item[(iv)] We have $\phi_\tau^\lambda = \phi_\tau^{\lambda_\pi}$ for any classical $\lambda \in \cW_{\lambda_\pi}^P$. 
	\end{itemize}
\end{proposition}

\begin{proof}
	(i)   We iterate Lemma \ref{lem:adjacent spin 1}.  Let $X_P = \{I_1,...,I_{\#X_P}\}$. Let $1 \leq i_1 \leq n$ be minimal with $i_1 \not\in X_P$. Then there exists some $r$ such that $I_r < i_1 < I_{r+1}$ (where $I_0 \defeq 0$ and $I_{\#X_P+1} \defeq 2n-I_{\#X_P}$). By Lemma \ref{lem:adjacent spin 1}, there exists $I_r < i_1 < j_1 \leq I_{r+1}$ and an $(X_P \cup \{i_1\})$-spin $\tilde\pi^{(1)}$ satisfying
	\[
		\Psi_\UPS(\tilde\pi^{(1)}) = \Psi_\UPS(\tilde\pi) \cdot (i_1,j_1).
	\] 
Iterating this process $k \leq n-\#X_P$ times, we obtain a $p$-refinement $\tilde\pi' = \tilde\pi^{(k)}$ which is $\{1,...,n\}$-spin with $\Psi_{\UPS}(\tilde\pi') = \Psi_\UPS(\tilde\pi) \cdot (i_1,j_i)\cdots (i_k,j_k)$. By Proposition \ref{prop:P-spin criterion} $\tilde\pi'$ is $B$-spin.
	
	\medskip
	
	(ii) From Definition \ref{def:non-critical slope}, $\tilde\pi'$ has non-critical slope if
	\begin{equation}\label{eq:non-critical slope}
		v_p(\alpha'(U_{p,i}^\circ)) < \lambda_i - \lambda_{i+1} + 1, \qquad 1 \leq i \leq 2n-1
	\end{equation}
	By assumption this is true for $\alpha$.  To see it for $\alpha'$:
	\begin{itemize}\s
		\item[(ii-1)] If $i \geq n$: by the proof of \cite[Thm.\ 4]{Roc20} (more precisely, the sentence following the second displayed equation), as $\tilde\pi$ has non-critical slope, it is $n$-spin. In particular, $n \in X_P$. By construction this forces $1 \leq i_r,j_r \leq n$ for all $n$. By Proposition \ref{prop:p-refinement}, we see $\alpha'(U_{p,i}^\circ) = \alpha(U_{p,i}^\circ)$. So $\alpha'(U_{p,i}^\circ)$ is non-critical slope as $\alpha(U_{p,i}^\circ)$ is.
		
		\item[(ii-2)] If $i < n$: as $\tilde\pi'$ is $i$-spin, we have $v_p(\alpha'(U_{p,i}^\circ)) = v_p(\alpha'(U_{p,2n-i}^\circ))$ by Lemma \ref{lem:U_p equality}. This is non-critical slope by (ii-1).
	\end{itemize}

(iii) This follows from iterating Lemma \ref{lem:Hecke map}.

\medskip

(iv) By Lemma \ref{lem:lambda_i-lambda_i+1constant in families}, we know $\lambda_i - \lambda_{i+1}$ is constant in $\cW_{\lambda_\pi}^P$ whenever $i \not\in X_P$. In the map $\phi_{i_r,j_r}^{\lambda}$, the only dependence on $\lambda$ is in the term
\begin{equation}\label{eq:product of ps}
	p^{\lambda_{i_r} - \lambda_{j_r}} = p^{\lambda_{i_r} - \lambda_{i_r+1}} \cdots p^{\lambda_{j_r-1} - \lambda_{j_r}}.
\end{equation}
By construction, we know that $I_s < i_r \leq j_r \leq I_{s+1}$ fall between two adjacent elements of $X_P$, so that $i_r, i_r+1, ..., j_r-1 \not\in X_P$. Thus all of the terms in the product \eqref{eq:product of ps} are constant as $\lambda$ varies in $\cW_{\lambda_\pi}^P$. The result follows.
\end{proof}

\subsection{From $B$-spin to $P$-spin}

Let $\tilde\pi$ and $\tilde\pi'$ be as in Proposition \ref{prop:P to B}. By Theorem \ref{thm:BDGJW} and Lemma \ref{lem:E = E'}, there exists a unique $(n+1)$-dimensional symplectic family $\sC' \subset \sE'$ through $\tilde\pi'$. Assume $\sC'$ is as in Lemma \ref{lem:every point spherical}, and let 
\[
	\sC'_P \defeq \sC' \times_{\sW} \sW_{0,\lambda_\pi}^P
\]
 be the $(\#X_P+1)$-dimensional subspace varying only over $\sW_{0,\lambda_\pi}^P$. By Lemma \ref{lem:every point spherical}, every classical point $y' \in \sC'_P$ corresponds to a $p$-refined $\tilde\pi_y' = (\pi_y, \alpha_y')$ with $\pi_{y,p} = \Ind_B^G\UPS_y$ spherical and regular. Let $\tilde\pi_y = (\pi_y,\alpha_y)$ be the unique $p$-refinement with
 \[
 \Psi_{\UPS_y}(\tilde\pi_y') = \Psi_{\UPS_y}(\tilde\pi_y) \cdot \tau,
 \]
 for $\tau$ as in Proposition \ref{prop:P to B}.
 
 \begin{lemma}
The refinement $\tilde\pi_y$ is optimally $P$-spin and we have
\begin{equation}\label{eq:relate alpha' and alpha}
	\alpha_y' \circ \phi_\tau^{\lambda_\pi} = \alpha_y.
\end{equation}
 \end{lemma}

\begin{proof}
	By Lemma \ref{lem:every point spherical}, we know $\Psi_{\UPS_y}(\tilde\pi_y') = \Psi_{\UPS}(\tilde\pi')$. In particular, we have
	\[
		\Psi_{\UPS_y}(\tilde\pi_y) = \Psi_{\UPS_y}(\tilde\pi_y') \cdot \tau^{-1} = \Psi_{\UPS}(\tilde\pi') \cdot \tau^{-1} = \Psi_{\UPS}(\tilde\pi),
	\]
	so that $\tilde\pi_y$ is optimally $P$-spin. The identity \eqref{eq:relate alpha' and alpha} follows by iterating Lemma \ref{lem:Hecke map} as in Proposition \ref{prop:P to B}(iii). Here we use (iv) of that result to see $\phi_\tau^\lambda = \phi_\tau^{\lambda_\pi}$.
\end{proof}

\begin{lemma}\label{lem:y' nc slope}
	For a Zariski-dense set of classical $y' \in \sC'_P$, the $p$-refinement $\tilde\pi_y$ is non-critical slope, and thus corresponds to a classical $P$-spin point $y \in \sE$.
\end{lemma}
\begin{proof}
	Up to shrinking $\sC'_P$, we may assume that the slope of each $U_{p,i}^\circ$ is constant along $\sC'_P$. As $\phi_{\tau}^{\lambda_\pi}(U_{p,i}^\circ)$ is a product of $U_{p,i}^\circ$'s and terms constant over $\sW_{0,\lambda_\pi}^P$, the slope of $\alpha_y(U_{p,i}^\circ) = \alpha_y' \circ \phi_{\tau}^{\lambda_\pi}(U_{p,i}^\circ)$ is constant, equal to $v_p(\alpha(U_{p,i}^\circ))$, for all $i$ and for any classical $y' \in \sC'_P$. 
	
	By assumption $\tilde\pi$ is non-critical slope (for $\lambda_\pi$). For a Zariski-dense set of classical weights $\lambda_y \in w(\sC'_P)$, the non-critical slope condition \eqref{eq:non-critical slope} for $\lambda_y$ is strictly weaker than for $\lambda_\pi$; so above all such weights, the points $\tilde\pi_y$ are non-critical slope.
\end{proof}
 
 Here we are using the very-Zariski-density of classical weights in the pure weight space, which allows us to shrink $\sC_P'$ whilst maintaining a Zariski-density of classical (symplectic) points.

\subsection{Proof of Theorem \ref{thm:lower bound}}
Let us take stock. We started with a non-critical slope $P$-spin refinement $\tilde\pi$, and via an element $\tau$ in the Weyl group, associated to it a non-critical slope $B$-spin refinement $\tilde\pi'$. This varies in a unique $(n+1)$-dimensional family $\sC'\subset \sE' = \sE$. Applying $\tau^{-1}$ to each $p$-refined classical point $y'$ in $\sC_P'$ gives another $P$-spin point $y \in \sE$. We now show this association can be interpolated over $\sW_{0,\lambda_\pi}^P$.

\begin{proposition}\label{prop:symplectic existence}
	There exists a finite map $t : \sC'_P \to \sE$ over $\sW_{0,\lambda_\pi}^P$ which interpolates the association $y' \mapsto y$. Thus there exists an $(\#X_P+1)$-dimensional symplectic family through $\tilde\pi$.
\end{proposition}

\begin{proof}
We use an interpolation idea that originally dates back to Chenevier \cite{Che05}. The precise version we use is \cite[Thm.\ 3.2.1]{JoNew}, which says: suppose we have eigenvariety data $\cD_1, \cD_2$, using Hecke algebras $\cH_1, \cH_2$, giving eigenvarieties $\sE_1, \sE_2$. Suppose there is a map $\phi : \cH_2 \to \cH_1$ and a Zariski-dense set of points $y_1 \in \sE_1$ with $\alpha_{y_1} \circ \phi$ appearing as a point $y_2 \in \sE_2$. Then there is a finite map $\sE_1 \to \sE_2$ interpolating the transfer $y_1 \mapsto y_2$. We need only explain why our situation fits this.

Let $\Omega_P \defeq w(\sC'_P)$. The part of the eigenvariety $\sE$ over $\Omega_P$ is constructed from an eigenvariety datum
\[
	\cD_2 = (\Omega_P, \sZ, \sH, \cH, \psi)
\]
in the sense of \cite[Def.\ 3.1.1]{JoNew}. Also \cite[Cor.\ 3.1.5]{JoNew} allows us to realise $\sC_P'$ inside the eigenvariety attached to an eigenvariety datum 
\[
\cD_1 = (\Omega_P, \sZ_{\sC_P'}, \sH', \cH^{\mathrm{frac}}, \psi),
\]
where we shrink the weight space to be $P$-parabolic, and the Fredholm hypersurface to isolate the component containing $\sC_P'$. 

The map of Hecke algebras is $\phi_{\tau}^{\lambda_\pi} : \cH \to \cH^{\mathrm{frac}}.$ For a Zariski-dense set of classical $y' \in \sC_P'$, corresponding to eigensystems $\alpha'$, by Lemma \ref{lem:y' nc slope} the eigensystem $\alpha' \circ \phi_\tau^{\lambda_\pi}$ appears in $\sE$, and we deduce existence of $t$ by \cite{JoNew}.
	
	Now $t(\sC_P') \subset \sE$ is the required symplectic family through $\tilde\pi$.
\end{proof}

We have now proved existence of an $(\#X_P+1)$-dimensional symplectic family $\sC$ through any non-critical slope $P$-spin point in $\sE_{K_1(\tilde\pi)}^G$. Theorem \ref{thm:lower bound} claims that this family is unique and \'etale over its image in weight space, an affinoid $\Omega_P$ in $\sW_{0,\lambda_\pi}^{P}$ (noting $P = \Pmin$).

To complete the proof, key is the observation that at level $K_1(\tilde\pi)$, with appropriate signs, the $\tilde\pi$-isotypic part of the top-degree compactly-supported cohomology is 1-dimensional (as in e.g.\ \cite[Prop.\ 7.18]{BDW20}). Then as in Proposition 7.19 \emph{op.\ cit}., there exists an ideal $I \subset \cO_{\Omega_P,\lambda_\pi}$ such that we have a relation
\begin{equation}\label{eq:flat}
\cO_{\sE_{K_1(\tilde\pi)}^G,\tilde\pi} = \cO_{\Omega_P,\lambda_\pi}/I
\end{equation}
between the local rings. It suffices to prove $I=0$, since then $\cO_{\sE_{K_1(\tilde\pi)}^G,\tilde\pi}$ is free of rank one over $\cO_{\Omega_P,\lambda_\pi}$, and in particular $\sE_{K_1(\tilde\pi)}^G \to \Omega_P$ is \'etale at $\tilde\pi$; so $\sC$ must be the \emph{unique} family through $\tilde\pi$, and is itself \'etale over $\Omega_P$ at $\tilde\pi$. 
	
	Note that $\cO_{\Omega_P,\lambda_\pi}$ is a regular local ring, and in particular it is reduced. Thus if $I\neq 0$, then $\Omega_{P,\lambda_\pi}/I$ has dimension strictly less than $\dim(\Omega_{P,\lambda_\pi}) = \dim(\Omega_P) = \#X_P+1$. By \eqref{eq:flat}, it would follow that every component of $\sE_{K_1(\tilde\pi)}^G$ through $\tilde\pi$ has dimension $< \dim(\Omega_P) = \#X_P+1$. But this contradicts the existence of the ($\#X_P+1)$-dimensional component $\sC$. Thus $I=0$, $\sC$ is unique and $w : \sC \to \Omega_P$ is \'etale at $\tilde\pi$. This completes the proof of Theorem \ref{thm:lower bound}. \qed

\begin{remark}\label{rem:infinite fern}
For $\GL_2$, the \emph{infinite fern} (see \cite{GM98}) is the image of the Coleman--Mazur eigenvariety in an unobstructed deformation space of residual Galois representations. If $\pi$ is a $p$-spherical RACAR of $\GL_2$, then there are two $p$-refinements $\pi_\alpha, \pi_\beta$, each varying in Coleman families; but both $\pi_\alpha,\pi_\beta$ have the same underlying Galois representation, so have the same image in the infinite fern, and the images of their families in the infinite fern cross at this point.

The proof here suggest that, given a hypothetical `infinite fern' $\sI$ for $\GL_{2n}$, there would be a picture with higher-dimensional intersections. Consider e.g.\ $\GL_4$; then the image of the $\GL_4$-eigenvariety in $\sI$ through $\pi$ should comprise 24 surfaces (the Iwahori families), intersecting at 6 lines (the $Q$-parahoric families), which all intersect at a single point (corresponding to $\pi$). Our expectation is that 8 of the surfaces (through the $B$-spin refinements) comprise classical points, and these intersect at 4 lines (corresponding to 4 classical families at $Q$-parahoric level).  

A higher-dimensional `infinite fern' for \emph{polarised} Galois representations of $\GL_n$ over CM fields is the main focus of \cite{HS-fern}. 
\end{remark}


\section{Explicit examples for $\GL_4$}\label{sec:examples}

We now illustrate the theory concretely for $\GL_4$, and give an explicit example. There are 4 spin parabolics in $G$: $B$, the (2,2)-parabolic $Q$, the (1,2,1)-parabolic $Q'$, and $G$ itself. Suppose $\pi$ is a RASCAR of $\GL_4$ with $\pi_p$ spherical, the transfer of a RACAR $\Pi$ on $\mathrm{GSp}_4$, and let $\cF \in \Pi$ be a Siegel newform of level prime to $p$.

There are 6 $Q$-refinements of $\pi_p$ (Hecke eigensystems in the $Q$-parahoric invariants of $\pi_p$), corresponding to elements of $\cW_G/\cW_{L_Q}$. These are combinatorially represented by decomposing $\{1,2,3,4\}$ into an ordered disjoint union $A_1 \sqcup A_2$, where $\#A_1 = \#A_2 = 2$ (cf.\ \cite[\S3.3]{DJR18}). Exactly four of these are `$Q$-spin', factoring through Klingen refinements of $\cF$:
\begin{equation}\label{eq:Q-spin}
	\{1,2\}\sqcup\{3,4\}, \ \ \  \{1,3\} \sqcup \{2,4\}, \ \ \  \{24\}\sqcup\{13\}, \  \ \ \{34\}\sqcup\{12\},
\end{equation}
whilst $\{14\}\sqcup\{23\}$ and $\{23\}\sqcup\{14\}$ do not factor. These four are the refinements satisfying the combinatorial criterion \cite[Def.\ 3.5(ii)]{DJR18}.

There are 24 Iwahori $p$-refinements, each lying above a unique $Q$-refinement. Each $Q$-refinement $A_1\sqcup A_2$ has 4 further Iwahori refinements, corresponding to orderings on $A_1$ and $A_2$; e.g.\ above $\{1,2\} \sqcup\{3,4\}$ are \{1234\}, \{2134\}, \{1243\}, \{2143\}. The table below lists all the Iwahori $p$-refinements $\tilde\pi$, together with the smallest parabolic $P \subset G$ such that $\tilde\pi$ is $P$-spin.

\begin{center}
	\begin{tabular}{c|c}
		\textbf{$\tilde\pi$ optimally:} & $\Psi_{\UPS}(\tilde\pi)$ \\
		\hline
		$B$-spin & \{1234\}, \{1324\}, \{2143\}, \{2413\}, \{3142\}, \{3412\}, \{4231\}, \{4321\}\\
		$Q$-spin & \{2134\}, \{3124\}, \{1243\}, \{4213\}, \{1342\}, \{4312\}, \{2431\}, \{3421\}\\
		$G$-spin & \{2314\}, \{3214\}, \{1423\}, \{4123\}, \{1432\}, \{4132\}, \{2341\}, \{3241\}
	\end{tabular}
\end{center}

(Any $Q'$-spin refinement is automatically a $B$-spin refinement by Lemma \ref{lem:adjacent spin 1}(ii)). We conjecture that the dimension of the symplectic locus through the optimally $B$-spin, $Q$-spin and $G$-spin refinements is 3, 2 and 1 respectively; we have proved this for non-critical slope $\tilde\pi$.

\begin{example}
	From the tables at \url{www.smf.compositio.nl}, there is a unique non-endoscopic Siegel modular form $\cF$ on $\mathrm{GSp}_4$ of level 1 that transfers to a RASCAR $\pi$ on $\GL_4$ of weight $\lambda = (12,1,-1,-12)$; and $\pi$ is everywhere spherical.
	
	At $p = 11$, by examining the Newton polygon, one sees this $\pi$ admits a parahoric-ordinary $Q$-refinement $\tilde\pi^Q$, corresponding to an ordinary Klingen refinement of $\cF$. We can normalise $\UPS$ so that this $Q$-refinement is $\{1,2\} \sqcup \{3,4\}$.
	
	The 4 Iwahori refinements above $\tilde\pi^Q$ are $\{1234\}, \{2134\}, \{1243\}, \{2143\}$. For $\lambda = (12,1,-1,-12)$, the  non-critical slope bounds \eqref{eq:non-critical slope} are $v_p(U_{p,1}) < 12$, $v_p(U_{p,2}) < 3$, $v_p(U_{p,3}) < 12$. We see:
	\begin{itemize}
		\item $\{1234\}$ is $B$-spin. Its $U_{p,i}$-eigenvalues have slopes $v_p(U_{p,1}) = v_p(U_{p,3}) = 11$ and $v_p(U_{p,2}) = 0$. This is non-critical slope, varying in a unique 3-dimensional symplectic family. 
		
		\item $\{2134\}$ is optimally $Q$-spin. The slopes are $v_p(U_{p,1}) = 11$, $v_p(U_{p,2}) = 0$, $v_p(U_{p,3}) = 1$. This is non-critical slope, varying in a 2-dimensional symplectic family, inside a 3-dimensional component of the eigenvariety.
	\end{itemize}
	Similarly $\{1243\}$ and $\{2143\}$ are non-critical slope, optimally $Q$-spin and $B$-spin respectively.
\end{example}

\part{$p$-refined Friedberg--Jacquet Integrals}

In Part III, we focus on parahoric $P$-refinements $\tilde\pi^P$. We give a conjectural classification of the $P$-spin $P$-refinements in terms of non-vanishing of twisted global period integrals, and prove various results towards this by using the results of Part II. Our conjecture generalises \cite[Expectation 7.2]{BDGJW}, which we prove in some cases. 


\section{$p$-refined Friedberg--Jacquet integrals: Statements}\label{sec:p-refined FJ}

Let $\pi$ be a RACAR of $G(\A)$. For $\varphi \in \pi$ and Hecke characters $\chi,\eta$, let
\begin{equation}\label{eq:period integral}
Z_H(\varphi,\chi,s) \defeq \int_{\A^\times H(\Q)\backslash H(\A)} \varphi\left[\matrd{h_1}{}{}{h_2}\right]\chi|\cdot|^{s-\tfrac{1}{2}}\left(\frac{\det(h_1)}{\det(h_2)}\right)\eta^{-1}\big(\det(h_2)\big) dh,
\end{equation}
where $H = \GL_n\times\GL_n$. In \cite[Prop.\ 2.2]{FJ93} (with \cite{AS14}) Friedberg--Jacquet proved:

\begin{theorem}\label{thm:FJ}
	Let $\pi$ be a RACAR of $G(\A)$. Let $\chi, \eta$ be algebraic Hecke characters, with $\chi$ finite order. Then for any $s \in \C$, the following are equivalent:
	\begin{enumerate}[(1)]\s
		\item There exists $\varphi \in \pi$ such that $Z_H(\varphi,\chi,s+1/2) \neq 0$.
		\item All of the following hold:
		\begin{itemize}\s
			\item[--]$\pi$ is a functorial transfer of some $\Pi$ on $\mathrm{GSpin}_{2n+1}(\A)$ with central character $\eta$, 
			\item[--] $L(\pi\times\chi, s+1/2) \neq 0$.
		\end{itemize}
	\end{enumerate}
\end{theorem}

In this section, we conjecture a local `$p$-refined' analogue of this theorem, that naturally arises from the question of $p$-adic interpolation (and $p$-adic $L$-functions).

	\subsection{Context for the conjecture}
	
	Our study is motivated by $p$-adic interpolation of period integrals. We briefly set up this problem. 
	
	Let $\Sigma_p$ denote the set of Dirichlet characters of (non-trivial) $p$-power conductor, and let $J$ be a finite set of integers. Let $\{A_{\chi,j} : \chi \in \Sigma_p, j \in J\} \subset \C$ be a set of complex numbers, and $i_p : \C \isorightarrow \overline{\Q}_p$ a fixed choice of isomorphism. We say that this set is \emph{$p$-adically interpolable} if there exists a locally analytic $p$-adic distribution $\mu$ on $\Zp^\times$, of growth/order $h < \# J$ (in the sense of \cite[\S6.2.4]{Eigenbook}), such that
	\[
		\int_{\Zp^\times} \chi(x) x^j = i_p(A_{\chi,j}) \qquad \text{for all $\chi \in \Sigma_p, \  j\in J$}.
	\]	
	If such a $\mu$ exists, it is uniquely determined by these interpolation and growth properties.  
	
As an example of this, let  $E/\Q$ be an elliptic curve with good ordinary reduction at $p$, let $J = \{0\}$, and let $A_{\chi,0}$ be the algebraic part of the critical $L$-value $L(E,\chi,1)$, with the appropriate modified Euler factors defined by Coates--Perrin-Riou \cite{coates89}. This is interpolable by a $p$-adic measure on $\Zp^\times$ (a distribution of growth 0), the \emph{$p$-adic $L$-function} of $E$, as proved in \cite{MSD74}.
	
	It is natural to ask if values of period integrals can be $p$-adically interpolated. Precisely, for $\pi$ as in Theorem \ref{thm:FJ}, can one interpolate the set
	\[
	\{Z_H(\varphi,\chi,j+1/2) : \chi \in \Sigma_p, j \in J\}
	\]
	for appropriate $\varphi$ and $J$? This question is only meaningful for RASCARs, where it was first studied in \cite{AG94}, and subsequently in \cite{DJR18,Geh18,BDW20,BDGJW}. It turns out that as stated, the answer is (likely to be) \emph{no}: there are problems locally at $p$. However, one may obtain a positive answer after making a standard modification at $p$, for which we need some notation.

		\begin{notation}\label{not:u}
			\begin{itemize}\s
				\item[--] Let $\tilde\pi^P = (\pi,\alpha^P)$ be a $P$-refinement. We say $\varphi \in \tilde\pi^P$ (resp.\ $\varphi_p \in \tilde\pi_p^P$) if $\varphi \in \pi^{J_P}$ (resp.\ $\varphi_p \in \pi_p^{J_P}$) is an $\alpha^P$-eigenvector for $\cH_p^P$. 
				
				\item[--] Let $u = \smallmatrd{1}{-w_n}{0}{1} \in \GL_{2n}(\Qp)$, where $w_n$ is the longest Weyl element in $\GL_n(\Qp)$ (i.e.\ the matrix with 1s along the antidiagonal, and 0s elsewhere).  If $P$ is the $(m_1,...,m_r)$-parabolic (see Notation \ref{not:(ni)-spin}), let 
				\[
				t_P = \mathrm{diag}(p^{r-1}\mathrm{I}_{m_1}, ..., p\mathrm{I}_{m_{r-1}},\mathrm{I}_{m_r}) \in T(\Qp).
				\]
			\end{itemize}
			For any $\beta \geq 1$, we view $ut_P^\beta \in G(\Qp) \subset G(\A)$ in the obvious way.
		\end{notation}

The following is proved in \cite{BDW20,BDGJW}, generalising results from \cite{AG94,Geh18,DJR18}.

\begin{theorem}\label{thm:interpolable}
	Let $\pi$ be a RASCAR of $G(\A)$ of weight $\lambda = (\lambda_1,\dots, \lambda_{2n})$. Let $J = \{j \in \Z : -\lambda_{n+1} \leq j \leq -\lambda_n\}$. Let $P$ be a spin parabolic, and $\tilde\pi^P$ be a non-$P$-critical slope $P$-refinement.  
	
	For any $\varphi \in \tilde\pi^P$, the set
	\[
	\Big\{ C_\beta Z_H(ut_p^\beta \cdot \varphi, \chi, j+\tfrac12) : \chi \in \Sigma_p \text{ of conductor $p^\beta$}, \ \ j \in J\Big\}
	\]
	is $p$-adically interpolable, for $C_\beta$ an explicit non-zero volume term that depends only on $\beta$.
\end{theorem}

Whilst this theorem is not stated as written in the aforementioned works, it is an implicit step in the constructions of $p$-adic $L$-functions contained therein. Its relevance stems from Friedberg--Jacquet's decomposition and study of global period integrals via local integrals,  as follows.

Write $\pi = \otimes_v \pi_v$, and consider $\varphi = \otimes\varphi_v$ a pure tensor. Then \cite[Prop.\ 2.3, \S3]{FJ93} shows that
	\[
	Z_H(\varphi,\chi,s) = \prod_v \zeta_v(\varphi_v,\chi_v,s),
	\]
	where 
	\begin{equation}\label{eq:local FJ}
	\zeta_v(\varphi_v,\chi_v,s) \defeq \int_{\GL_n(\Q_v)}\cS_{\psi_v}^{\eta_v}(\varphi_v)\left[\matrd{x}{}{}{1}\right]\chi_v|\cdot|^{s-\tfrac{1}{2}}\Big(\det x\Big) dx
	\end{equation}
	is the local Friedberg--Jacquet integral attached to $\pi_v$. Here $\cS_{\psi_v}^{\eta_v}$ is an intertwining of $\pi_v$ into its Shalika model (see e.g.\ \cite[\S2.6]{BDW20}). 
	
	Let $\ell\neq p$ be a finite prime, and $\varphi_\ell \in \pi_\ell$. By \cite[Prop.\ 3.1]{FJ93}, for each unramified quasi-character $\chi_\ell: F_\ell^\times \to \C^\times$, there exists a holomorphic function $r_\ell(\varphi_\ell,\chi_\ell,s)$ such that  
	\[	
	\zeta_\ell(\varphi_\ell,\chi_\ell,s) = r_\ell(\varphi_\ell,\chi_\ell,s) \cdot L(\pi_\ell\times\chi_\ell,s).
\]
 Moreover there exists $\varphi_\ell^{\mathrm{FJ}} \in \pi_\ell$ such that $r_\ell(\varphi_\ell^{\mathrm{FJ}},\chi_\ell,s) = 1$. If $\pi_\ell$ is spherical, we may take $\varphi_\ell^{\mathrm{FJ}}$ spherical \cite[Prop.\ 3.2]{FJ93}. 
	
	At infinity, by \cite{Sun19} there exists a vector $\varphi_\infty^{\mathrm{FJ}} \in \pi_\infty$ such that $\zeta_\infty(\varphi_\infty^{\mathrm{FJ}},\chi_\infty,s) \neq 0$.

	For these `good' choices of local vectors away from $p$, the interpolation theorem above then implies an interpolation of the values 
		\[
	\Big\{ C_\beta \zeta_\infty(\varphi_\infty,\chi_\infty,j+\tfrac12) \cdot L^{(p)}(\pi\times\chi, j+\tfrac12) \cdot \zeta_p(ut_P^\beta \cdot \varphi_p, \chi_p, j+\tfrac12) : \chi \in \Sigma_p \text{ of conductor $p^\beta$}, \ j \in J\Big\}.
	\]
	We are left to study the integral $\zeta_p(ut_P^\beta \cdot \varphi_p, \chi_p, s)$. In particular, we must ask for which eigenvectors $\varphi_p$ this integral is non-zero; else the interpolation theorem is vacuous. Where it is non-zero, we obtain a $p$-adic interpolation of $L$-values (that is, a $p$-adic $L$-function). The non-vanishing of this integral is the subject of our conjecture.

\subsection{Local statement of the conjecture}

	The local Friedberg--Jacquet integrals in \eqref{eq:local FJ} can be defined more generally for $\pi_p$ any irreducible admissible representation of $G(\Q_p)$ admitting an $(\eta_p,\psi_p)$-Shalika model, for  $\psi_p$ and $\eta_p$ smooth characters of $\Q_p$ and $\Q_p^\times$ respectively. We will specialise further to the case where $\pi_p = \Ind_B^G \UPS$ is an unramified principal series representation with regular semisimple Satake parameter. 
	
By \cite[Prop.\ 1.3]{AG94} and \cite[p.177(i)]{AS06}, we see that $\pi_p$ is a functorial transfer of an unramified principal series representation $\Pi_p$ of $\cG(\Q_p)$. Note that all our definitions and properties of $P$-refinements used only local data at $p$, and hence make sense for $\pi_p$.

\begin{conjecture}\label{conj:local FJ}
	Let $\pi_p$ be an unramified principal series representation of $\GL_{2n}(\Qp)$ with regular semisimple Satake parameter, admitting an $(\eta_p,\psi_p)$-Shalika model.  
	Let $\tilde\pi_p^P$ be a $P$-refinement of $\pi_p$. Let $\chi_p$ be a finite order character of $\Qp^\times$ of conductor $p^\beta > 1$. 
	
	\medskip

	The following are equivalent:
	\begin{enumerate}[(1)]\s
		\item There exists an eigenvector $\varphi_p\in \tilde\pi_p^P$ such that $\zeta_p(ut_P^\beta \cdot \varphi_p,\chi_p,s+1/2) \neq 0$ for some $s \in \C$.
		\item Both of the following hold:
		 \begin{itemize} \item[(i)] $P$ is contained in the $(n,n)$-parabolic, 
			\item[(ii)] and $\tilde\pi_p^P$ is $P$-spin. 
		\end{itemize}
	\end{enumerate}
\end{conjecture}

In the remainder of the paper, we prove a number of results towards this conjecture. In particular, we use local methods to prove (2) $\Rightarrow$ (1) and (1) $\Rightarrow$ (2-i) always hold. We also use our (global) results on classical symplectic families to prove partial results towards (1) $\Rightarrow$ (2-ii) for $\pi_p$ that arise as the local component of a nice enough RASCAR $\pi$. See \S\ref{sec:results towards conjecture} for precise statements of these results.

\subsection{A global version of the conjecture}

To apply our global results, it is convenient to state a global analogue of Conjecture \ref{conj:local FJ}. Whilst strictly weaker, it has the benefit of being more obviously analogous to Friedberg--Jacquet's original result (Theorem \ref{thm:FJ}).

\begin{definition}\label{def:refinement transfer}
	Let $\tilde\pi^P$ be a $P$-refined RACAR of $G(\A)$ for $P\subset G$ a spin parabolic, with associated $\cP \subset \cG$. We say $\tilde\pi^P$ is a \emph{functorial transfer of a $\cP$-refined $\tilde\Pi^{\cP}$ on $\mathrm{GSpin}_{2n+1}(\A)$} if $\pi$ is the functorial transfer of $\Pi$, and $\tilde\pi^P$ is the functorial transfer of $\tilde\Pi^{\cP}$ in the sense of Definition  \ref{def:refinement transfer 1}.
\end{definition}

\begin{conjecture}\label{conj:FJ}
	 Let $P \subsetneq G$ be a proper spin parabolic, with associated $\cP \subset \mathrm{GSpin}_{2n+1}$. Let $\tilde\pi^P = (\pi,\alpha^P)$ be a $P$-refined RACAR of $G(\A)$. Assume $\pi_p$ is spherical with regular semisimple Satake parameter. Let $\chi, \eta$ be algebraic Hecke characters, with $\chi$ finite order of conductor $p^\beta > 1$. For any $s \in \C$, the following are equivalent:
		\begin{enumerate}[(1$'$)]\s
		\item There exists $\varphi \in \tilde\pi^P$ such that $Z_H(ut_P^\beta \cdot \varphi,\chi,s+1/2) \neq 0$.
		\item All of the following hold:
		\begin{itemize}\s 
						\item[(i)] $P$ is contained in the $(n,n)$-parabolic,
			\item[(ii)] $\tilde\pi^P$ is a functorial transfer of some $\tilde\Pi^{\cP}$ on $\mathrm{GSpin}_{2n+1}(\A)$ with central character $\eta$, 
			\item[(iii)]$L(\pi\times\chi, s+1/2) \neq 0$.
		\end{itemize}
	\end{enumerate}
\end{conjecture}

\begin{proposition}\label{prop:local vs global}
	Let $\pi = \otimes \pi_v$ be a RASCAR of $G(\A)$ such that $\pi_p$ is as in Conjecture \ref{conj:local FJ}.  
	
	Then Conjecture \ref{conj:local FJ} for $\pi_p$ is equivalent to Conjecture \ref{conj:FJ} for $\pi$.
\end{proposition}

\begin{proof}
	This largely follows from the discussion after Theorem \ref{thm:interpolable}, which implies that for pure tensors $\varphi = \otimes_v\varphi_v$, we have
	\begin{equation}\label{eq:euler product}
	Z_H(ut_P^\beta\cdot \varphi,\chi,s+1/2) =  \zeta_p(ut_P^\beta\cdot\varphi_p,\chi_p,s+1/2) \prod_{v\neq p}\zeta_v(\varphi_v,\chi_v,s+1/2).
	\end{equation}	
	
	Logically, we must show that 
	\[
	\Big[(1) \iff (2)\Big] \iff \Big[(1')\iff(2')\Big],
	\]
	where (1), (2) are in the local conjecture, and (1$'$), (2$'$) are in the global conjecture. First we draw some implications between the various conditions.
	
	\begin{itemize}
		\item $(1') \Rightarrow (1)$. If there exists $\varphi$ such that $Z_H(ut_P^\beta\cdot\varphi,\chi,s+1/2) \neq 0$, then we may without loss of generality replace $\varphi$ with a pure tensor $\varphi = \otimes_v\varphi_v$. As the product in \eqref{eq:euler product} is holomorphic, non-vanishing of the left-hand side implies non-vanishing of the factor $\zeta_p(ut_P^\beta\cdot\varphi_p,\chi_p,s+1/2)$, so (1) holds.
	
	\item $(1) \Rightarrow (1')$.  If (1) holds, there exists $\varphi_p \in \tilde\pi_p^P$ such that $\xi_p(ut_P^\beta\varphi_p,\chi_p,s+1/2) \neq 0$ for some $s\in\C$. We shall prove in Corollary \ref{cor:non-vanishing for s} in the next section that this implies non-vanishing for \emph{all} $s\in \C$. (This corollary is proved via purely local methods, so there is no circular argument here). Fix some $s$ such that $L(\pi\times\chi,s+1/2) \neq 0$, hence $L^{(p)}(\pi\times\chi,s+1/2)\neq 0$ (where we delete the factor at $p$). Let $\varphi \defeq \varphi_\infty^{\mathrm{FJ}}\otimes \varphi_p \otimes \bigotimes_{\ell\nmid p\infty}\varphi_\ell^{\mathrm{FJ}}$, for $\varphi_v^{\mathrm{FJ}}$ the test vectors chosen after Theorem \ref{thm:interpolable}. Then that discussion, combined with \eqref{eq:euler product}, shows
	\begin{equation}\label{eq:L non zero}
		Z_H(ut_P^\beta\cdot\varphi, \chi,s+1/2) = \zeta_p(ut_P^\beta\cdot\varphi_p,\chi_p,s+1/2) \cdot \zeta_\infty(\varphi_\infty^{\mathrm{FJ}},\chi_\infty,s+1/2)\cdot L^{(p)}(\pi\times\chi,s+1/2) \neq 0,
	\end{equation}
	so (1$'$) holds.
	
\item $(2') \Rightarrow (2)$. Conditions (2-i) and (2$'$-i) are identical, and (2-ii) and (2$'$-ii) are equivalent by Proposition \ref{prop:spin factor}. As (2$'$) consists of this and one further condition, it implies (2).

\item Additionally, we see from the previous point that if we have (2) and $L(\pi\times\chi,s+1/2) \neq 0$ for some $s \in \C$, then $(2')$ holds (for this $s$).
\end{itemize}

Now suppose the global conjecture holds. If (1) holds then (1$'$) holds, so (2$'$) holds by the conjecture, which implies (2) holds locally. Conversely if (2) holds, then pick some $s\in\C$ such that $L(\pi\times\chi,s+1/2)\neq 0$; then (2$'$) holds, so (1$'$) holds by the global conjecture. Thus (1) holds. Thus the local conjecture is true.

\medskip

Finally, suppose the local conjecture holds. In the global conjecture, if (2$'$) holds, then (2) holds, so (1) holds by the local conjecture, so (1$'$) holds.

If (1$'$) holds for $s\in \C$, then (1) holds, so (2) holds by the local conjecture. Thus (2$'$-i) and (2$'$-ii) hold from above. It remains to show (2$'$-iii), that $L(\pi\times\chi,s+1/2) \neq 0$. But this follows directly from (1$'$) by Theorem \ref{thm:FJ}. In particular, the global conjecture is true.
\end{proof}

For the rest of the paper, and in light of the methods used in the above proof, we assume our prime-to-$p$ level $K^p \subset \GL_{2n}(\A_f^{(p)})$ fixes $\otimes_{\ell\neq p}\varphi_\ell^{\mathrm{FJ}}$, which is possible by \cite[Prop.\ 3.2]{FJ93}.

\subsection{Results towards the conjectures}\label{sec:results towards conjecture}
We collect together our results towards the above conjectures. In \S\ref{sec:proof one direction} we show:

\begin{theorem}\label{thm:2 implies 1}
	Implication (2) $\Rightarrow$ (1) holds in Conjecture \ref{conj:local FJ}.
\end{theorem}

In Corollary \ref{cor:zeta vanishing}, we also show that (1) $\Rightarrow$ (2-i). In particular, to prove Conjecture \ref{conj:FJ} in full, it remains to show (1) $\Rightarrow$ (2-ii).
	
	Our results towards this are global; in particular, we prove cases of (1$'$) $\Rightarrow$ (2$'$-ii) in Conjecture \ref{conj:FJ}. As an example, we prove Theorem \ref{thm:intro 2}(ii) of the introduction: that the conjecture holds in full if we can further refine $\tilde\pi^P$ to a non-critical slope $B$-refinement. To state our (stronger) precise result, we require more terminology.

 \begin{itemize}\s
 \item Fix a prime-to-$p$ level $K^p \subset \GL_{2n}(\A_f^{(p)})$. For a parabolic $P$, we let $K_P \defeq K^p J_P \subset \GL_{2n}(\A_f)$, where $J_P$ is the $P$-parahoric subgroup.
 	
 \item For any open compact $K \subset \GL_{2n}(\A_f)$, let $S_K$ denote the $\GL_{2n}$-locally symmetric space of level $K$ (see \cite[\S2.3]{BDW20}). 

\item For any parabolic $P$, let $\cD_\lambda^P$ be the module of weight $\lambda$ $P$-parahoric distributions for $G$, defined in \cite[\S3.2]{BW20}. We have $\cD_\lambda^G = V_\lambda^\vee$ is the dual of the algebraic induction of $\lambda$, and $\cD_\lambda^B = \cD_\lambda$ is the usual module of (Iwahori) locally analytic distributions. We have attached $p$-adic local systems $\sV_\lambda^\vee/\sD_\lambda^P$ on $S_{K_P}$ (e.g.\ \cite[\S2.3.2]{BDW20}).
 
 \item The \emph{top degree eigenvariety} was constructed in \cite[\S5]{BW20}, following \cite{Han17}. It is built from modules $\hc{t}(S_{K_B},\sD_\Omega)$, where $\Omega \subset \sW$ is a weight affinoid and $\sD_\Omega$ is a local system of locally analytic distributions over $\Omega$ (as in \cite[Def.\ 3.11]{BDW20}; see \cite[\S3.2]{BW20}). Here $t = n^2+n-1$ is the top degree for cuspidal cohomology.
 
 \item We say \emph{$\tilde\pi^P$  appears in the top degree eigenvariety} if there exists an Iwahori refinement $\tilde\pi$ above $\tilde\pi^P$, and a neighbourhood $\Omega \subset \sW_{0,\lambda_\pi}^P$ of $\lambda_\pi$, such that the  natural specialisation map
\begin{equation}\label{eq:rlambda}
	r_{\tilde\pi} : \hc{t}(S_{K_B},\sD_{\Omega}) \xrightarrow{\ r_{\lambda_\pi}\ } \hc{t}(S_{K_B},\sV_{\lambda_\pi}^\vee) \twoheadrightarrow \hc{t}(S_{K_B},\sV_{\lambda_\pi}^\vee)_{\tilde\pi}
\end{equation}
(induced by $r_{\lambda_\pi} : \sD_\Omega \xrightarrow{\mod \m_{\lambda_\pi}} \sD_{\lambda_\pi} \twoheadrightarrow \sV_{\lambda_\pi}^\vee$, and then projection onto the $\tilde\pi$-eigenspace) is surjective. This implies that the $\tilde\pi$-localisation in $\hc{t}(S_{K_B},\sD_\Omega)$ is non-zero, and hence there is a point in the top degree  eigenvariety corresponding to $\tilde\pi$.

\item We say $\tilde\pi^P$ is \emph{$P$-strongly-interior} if the $P$-parahoric boundary overconvergent cohomology vanishes  $\h^\bullet_\partial(S_{K_P},\sD_{\lambda_\pi}^P)_{\tilde\pi^P} = 0$ (see Def.\ 5.13 \emph{op.\ cit}.). 
\end{itemize}

\begin{theorem}\label{thm:p-refined FJ}
Suppose $\pi$ has regular weight, that $\tilde\pi^P$ appears in the top degree eigenvariety, and that $\tilde\pi^P$ is $P$-strongly-interior. Then (1$'$) $\Rightarrow$ (2$'$) holds in Conjecture \ref{conj:FJ}.
\end{theorem}

\begin{remarks}\label{rem:p-refined FJ}
\begin{enumerate}[(i)]\s
\item We cautiously suggest the conditions on $\tilde\pi^P$ should hold for all $\tilde\pi^P$ (whence Conjecture \ref{conj:FJ}  would hold in full). Unconditionally, $\tilde\pi^P$ is $P$-strongly-interior if it is non-$P$-critical slope; see \cite[Lem.\ 5.14]{BW20}. It appears in the top-degree eigenvariety if there exists a non-$B$-critical Iwahori-refinement $\tilde\pi$ above $\tilde\pi^P$ (see \cite[Def.\ 4.1]{BW20} and \cite[Prop.\ 7.8]{BDW20}).  Hence Theorem \ref{thm:p-refined FJ} implies Theorem \ref{thm:intro 2}(ii) from the introduction.
	
\item When $P=B$, this proves \cite[Expectation 7.2]{BDGJW} for $\tilde\pi$ satisfying the conditions of Theorem \ref{thm:p-refined FJ}, thus for non-critical slope $\tilde\pi$ of regular weight (see Remark 7.3 \emph{op.\ cit}.). 
\end{enumerate}
\end{remarks}

\section{Proof of Theorem \ref{thm:2 implies 1}}\label{sec:proof one direction}

In this section, we give the proof of Theorem \ref{thm:2 implies 1} (that $(2) \Rightarrow (1)$ in Conjecture \ref{conj:local FJ}). Our proof is constructive; if (2) is satisfied, we describe explicitly an eigenvector with non-vanishing local zeta integral. If $P = B$ or the $(n,n)$-parabolic $Q$, then Theorem \ref{thm:2 implies 1} was proved in \cite[Cor.\ 7.15]{BDGJW} and  \cite[Prop.\ 3.4, Lem.\ 3.6]{DJR18} respectively.

Our proof for general $P$ is closely modelled on the approach in \cite{BDGJW}, and we refer to specific places \emph{op.\ cit}.\ for more detail. Recall $\cS_{\psi_p}^{\eta_p}$ is an intertwining of $\pi_p$ into its Shalika model, and for any $\varphi_p \in \pi_p$, we let $W_{\varphi_p} \defeq \cS_{\psi_p}^{\eta_p}(\varphi_p)$. Then we:
\begin{itemize}\s
	\item[(i)] Express $\zeta_p(ut_P^\beta \cdot \varphi_p, \chi_p, s+1/2)$ as a non-zero multiple of a value of $W_{\varphi_p}$;
	\item[(ii)] Show that if $P \subset Q$ and $\tilde\pi_p^P$ is a $P$-spin refinement, there exists $\varphi_p \in \tilde\pi_p^P$ where this specific value of $W_{\varphi_p}$ is non-zero.
\end{itemize}

\subsection{The local zeta integral}

\begin{proposition}\label{prop:zeta vanishing} 
	Let $\varphi_p \in \pi_p^{\Iw}$, and let $W_{\varphi_p} = \cS_{\psi_p}^{\eta_p}(\varphi_p)$. Let $\chi_p$ be a character of conductor $p^\beta >1$. Let $t = \smallmatrd{z_1}{}{}{z_2} \in T(\Qp)$, and 
	\[
	\nu_\beta(t) \defeq p^{-\beta} z_2^{-1} w_n z_1.
	\]
	Then for all $s$,
	\begin{equation}\label{eq:zeta vanishing}
		\Big[\zeta_p(u^{} t \cdot \varphi_p, \chi_p, s) \neq 0\Big] \iff \Big[W_{\varphi_p}\smallmatrd{\nu_\beta(t)}{}{}{1} \neq 0\Big].
	\end{equation}
\end{proposition}

\begin{proof}
	By equivariance of the Shalika intertwining, for any $g \in G(\Qp)$ we have 
		\[
		W_{ut \cdot \varphi_p}(g) = \cS_{\psi_p}^{\eta_p}(ut \cdot \varphi_p)(g) = \Big[ut \cdot \cS_{\psi_p}^{\eta_p}(\varphi_p)\Big](g) = W_{\varphi_p}\left(gut\right).
		\]
		In particular, setting $g = \smallmatrd{x}{}{}{1}$, by definition we have
		\[
		\zeta_p(ut\cdot \varphi_p,\chi_p,s) = \int_{\GL_n(\Qp)}W_{\varphi_p}\left[\matrd{x}{}{}{1}ut\right] \chi_p|\cdot|^{s-\tfrac{1}{2}}\big(\det x\big) dx.
		\] 
		Recalling $u = \smallmatrd{1}{-w_n}{0}{1}$, one can check that
\begin{equation}\label{eq:zeta rewrite}
\tbyt{x}{}{}{1} u^{} t = \tbyt{z_2}{}{}{z_2} \tbyt{1}{-z_2^{-1} x w_n z_2}{}{1} \tbyt{z_2^{-1} x z_1}{}{}{1}.
\end{equation}
By definition of the Shalika model, we have
\begin{align*}
W_{\varphi_p}\bigg[ \tbyt{z_2}{}{}{z_2} \tbyt{1}{-z_2^{-1} x w_n z_2}{}{1} &\tbyt{z_2^{-1} x z_1}{}{}{1}\bigg] =\\ &\eta_p(\det z_2)\psi_p\Big(\operatorname{tr}(-z_2^{-1} x w_n z_2) \Big) W_{\varphi_p}\tbyt{z_2^{-1} x z_1}{}{}{1}.
\end{align*}
In particular, combining this with \eqref{eq:zeta rewrite} shows
\[
\zeta_p(u^{}t \cdot \varphi_p, \chi_p, s) = \eta_p(\operatorname{det}z_2) \int_{\GL_n(\mathbb{Q}_p)} \psi_p\Big(\operatorname{tr}(-z_2^{-1} x w_n z_2) \Big) W_{\varphi_p}\tbyt{z_2^{-1} x z_1}{}{}{1} \chi_p|\cdot|^{s-\tfrac{1}{2}}\Big(\operatorname{det}x\Big) dx.
\]
Let $y = -z_1^{-1}w_n x z_1$, and let $\omega = -z_2^{-1}w_n z_1 = -p^{\beta}\nu_{\beta}(t)$. As $\mathrm{tr}(-z_2^{-1}xw_nz_2) = \mathrm{tr}(y)$, changing variables and noting $dx = dy$, we see
\begin{equation}\label{eq:zeta Q 1}
\zeta_p(u^{} t \cdot \varphi_p, \chi_p, s) = (\star) \cdot \cQ, \qquad (\star) \neq 0,
\end{equation}
where we define
\[
\cQ \defeq \int_{\GL_n(\mathbb{Q}_p)} \psi_p( \operatorname{tr}(y) ) I(\omega y) dy,
\]
for $I$ the function $\GL_n(\Qp) \to \C$ defined by
\[
	I(y) = W_{\varphi_p}\matrd{y}{}{}{1} \chi_p|\cdot|^{s-\tfrac12}(\operatorname{det} y).
\]

By \eqref{eq:zeta Q 1}, to prove \eqref{eq:zeta vanishing} it suffices to prove 
\[
	\cQ \neq 0 \iff W_{\varphi_p}\smallmatrd{\nu_\beta(t)}{}{}{1} \neq 0.
\]

We want to reduce the support of the integral $\cQ$. Let $M = \GL_n(\mathbb{Q}_p) \cap M_{n}(\mathbb{Z}_p)$. By \cite[Lem.\ 5.1]{BDGJW}, the support of $I(\omega y)$ (hence $\cQ$) is contained in $\omega^{-1} M$.

As in \cite[Not.\ 5.3]{BDGJW}, let $A$ denote the set of all diagonal $n \times n$-matrices of the form
		\[
		\gamma = \operatorname{diag}(c_{11}, \dots, c_{nn}), \qquad c_{ii} \in \Z_p^\times.
		\]
Let $B_{\beta}$ denote the additive group of all $n \times n$-matrices $\delta$ with
		\[
		\delta_{i, j} = \left\{ \begin{array}{cc} c_{i, j} & \text{ if } i < j \\ 0 & \text{ if } i = j \\ p^{\beta} c_{i, j} & \text{ if } i > j \end{array} \right., \qquad c_{ij} \in \Zp.
		\]
Let $\alpha = \gamma + \delta$, with $\gamma \in A$, $\delta \in B_{\beta}$. Note that $\det(\alpha) = \det(\gamma)$, that $|\det \alpha| = 1$, and that $\alpha \in \mathrm{Iw}_n(p^\beta)$ is in the depth $p^\beta$ Iwahori subgroup of $\GL_n(\Zp)$; in particular, 
\begin{align*}
I(y\alpha^{-1}) &= W_{\varphi_p}\left[\matrd{y}{}{}{1}\matrd{\alpha^{-1}}{}{}{1}\right]\chi_p|\cdot|^{s-\tfrac12}(\det (y\alpha^{-1}))\\ &= \chi_p(\det \alpha^{-1}) W_{\varphi_p}\matrd{y}{}{}{1}\chi_p|\cdot|^{s-\tfrac12}(\det y) = \chi_p(\det \gamma^{-1})I(y).
\end{align*}
Thus for any $\alpha = \gamma+\delta \in A+B_\beta$, we have
\begin{align}
\cQ = \int_{\omega^{-1} M} \psi_p\big( \operatorname{tr}(y) \big) I\big(\omega y\big) dy &= \chi_p(\operatorname{det} \gamma) \int_{\omega^{-1}M} \psi_p\Big( \operatorname{tr}(y) \Big) I\Big(\omega y \alpha^{-1}\Big) dy \notag\\
&= \chi_p(\det \gamma)\int_{\omega^{-1} M} \psi_p\Big(\operatorname{tr}(x \gamma)\Big) \psi_p\Big(\operatorname{tr}(x \delta)\Big) I\big(\omega x\big) dx, \label{eq:Q via alpha}
\end{align}
where we make the change of variables $x = y\alpha^{-1}$. If $x \in \GL_n(\Qp)$, then for each $\delta \in B_\beta$, we have
	\[
	\psi_p(\mathrm{tr}(x\delta)) = \prod_{i>j}\psi_p(x_{i,j}c_{j,i}) \cdot \prod_{i<j} \psi_p(x_{i,j}c_{j,i}p^\beta).
	\]
	From this, we see that:
	\begin{itemize}
		\item[(1)] For a fixed $x \in \GL_n(\Qp)$, the function
		\[ B_\beta \longrightarrow \C, \qquad \delta \longmapsto \psi_p(\operatorname{tr}(x \delta))\]
		is the trivial function if and only if
\begin{equation}\label{eq:Mbeta'}
x_{i, j} \in \left\{ \begin{array}{cc} p^{-\beta} \mathbb{Z}_p & \text{ if } i < j \\ \mathbb{Z}_p & \text{ if } i > j \end{array} \right..
\end{equation}
Let $M_\beta'$ be the subset of $x \in \GL_n(\mathbb{Q}_p)$ satisfying \eqref{eq:Mbeta'}. 

\item[(2)] For a fixed $\delta$ with every $c_{i,j}$ sufficiently divisible by $p$, then $\psi_p(\operatorname{tr}(x \delta)) = 1$  for all $x \in \omega^{-1}M$. Denote the subset of such $\delta \in B_\beta$ by $B_\beta'$, noting it has finite index in $B_\beta$.
\end{itemize}

Since \eqref{eq:Q via alpha} holds for any $\gamma+\delta \in A+B_\beta$, we can average over $B_\beta$ and use character orthogonality, as in \cite[Cor.\ 5.5]{BDGJW}. The right-hand side in \eqref{eq:Q via alpha} depends on $\delta \in B_\beta$ only up to $B_\beta'$, so for any fixed $\gamma \in A$ we have
\begin{align*}
	\cQ &= \chi_p(\det \gamma) \tfrac{1}{[B_\beta:B_\beta']} \sum_{\delta \in B_\beta/B_\beta'} \int_{\omega^{-1} M} \psi_p\Big(\operatorname{tr}(x \gamma)\Big) \psi_p\Big(\operatorname{tr}(x \delta)\Big) I\big(\omega x\big) dx\\
	&= \chi_p(\det \gamma)  \int_{\omega^{-1} M} \psi_p\Big(\operatorname{tr}(x \gamma)\Big)  I\big(\omega x\big)\bigg[\tfrac{1}{[B_\beta:B_\beta']} \sum_{\delta \in B_\beta/B_\beta'}\psi_p\Big(\operatorname{tr}(x \delta))\Big)\bigg] dx\\
	&= \chi_p(\det \gamma)  \int_{\omega^{-1} M \cap M_\beta'} \psi_p\Big(\operatorname{tr}(x \gamma)\Big)  I\big(\omega x\big) dx,
\end{align*}
using character orthogonality and observation (1) above in the last step.

Now we average over $\gamma \in A$. By the expression above, we have
	\begin{align*}
		\cQ &= \vol(A)^{-1}\int_{A}  \chi_p(\det \gamma)  \left[\int_{\omega^{-1} M \cap M_\beta'} \psi_p\Big(\operatorname{tr}(x \gamma)\Big)  I\big(\omega x\big) dx\right] d\gamma\\
		& = \vol(A)^{-1} \int_{\omega^{-1} M \cap M_\beta'} \left[\int_{A}  \chi_p(\det \gamma) \psi_p\Big(\operatorname{tr}(x \gamma)\Big)d\gamma\right]  I\big(\omega x\big) dx.
	\end{align*}
  We have
\[
\chi_p(\operatorname{det}\gamma) \psi_p(\operatorname{tr}(x\gamma)) = \prod_{i=1}^n \chi_p(c_{i, i}) \psi_p( x_{i, i} c_{i, i} ).
\]
Since $A = (\Zp^\times)^n$,  we then have
\begin{equation}\label{eq:gauss sum}
	\int_{A} \chi_p(\operatorname{det}\gamma) \psi_p\Big(\operatorname{tr}(x\gamma)\Big) d^{\times}\gamma = \prod_{i=1}^n\chi_p(p^\beta x_{i,i})^{-1} \int_{\Zp^\times}\chi_p\big(p^\beta x_{i,i}c_{i,i}\big)\psi_p\big(x_{i,i}c_{i,i}\big)dc_{i,i}.
\end{equation}
Recalling $\chi_p$ has conductor $p^\beta > 1$, it is standard that each integral in the right-hand product is zero unless $x_{i,i} \in p^{-\beta}\Zp^\times$; and in this case, the integral is an explicit, non-zero multiple of the (non-zero) Gauss sum $\tau(\chi)$. Hence when each $x_{i,i} \in p^{-\beta}\Zp^\times$, the equation \eqref{eq:gauss sum} has the form $(\star') \prod_{i=1}^n \chi_p(p^\beta x_{i, i})^{-1}$, with $(\star') \neq 0$ an explicit scalar depending only on $\chi, p$ and $\beta$.

Let $M_{\beta}'' \subset M_{\beta}'$ be the subset of $x \in M_\beta'$ where $x_{i,i} \in p^{-\beta}\Zp^\times$.  Note that $M_{\beta}'' = p^{-\beta} \operatorname{Iw}_n(p^{\beta})$. 
\[
\cQ = (\star'') \int_{\omega^{-1}M \cap M_{\beta}''} \prod_{i=1}^n \chi_p(p^{\beta} x_{i, i})^{-1} \cdot I(\omega x) dx, \qquad (\star'') \neq 0.
\]
Write $x' = p^{\beta} x$ for $x \in \omega^{-1}M \cap M_{\beta}''$. Then $\chi_p(\operatorname{det}x') = \prod_{i=1}^n \chi_p(p^{\beta} x_{i, i})$, as $\chi_p$ has conductor $p^{\beta}$. If $\nu = \nu_\beta(t) = -p^{-\beta} \omega,$ then we find
\begin{align}
	\cQ &= (\star'') \int_{\omega^{-1}M \cap M_\beta''} \chi_p(\det p^\beta x)^{-1} I(-p^\beta x) dx = \int_{\nu^{-1}M \cap \mathrm{Iw}(p^\beta)} \chi(\det x')^{-1} I(-\nu x') dx'\notag\\
&= (\star'') \int_{\nu^{-1}M \cap \operatorname{Iw}_n(p^{\beta})} \chi_p(\det x')^{-1} W_{\varphi_p}\matrd{-\nu x'}{}{}{1} \chi_p|\cdot|^{s-\tfrac{1}{2}}(\det -\nu x') dx' \notag \\ 
&= (\star'')\chi_p|\cdot|^{s-\tfrac{1}{2}}(\det -\nu)\int_{\nu^{-1}M \cap \operatorname{Iw}_n(p^{\beta})} W_{\varphi_p}\matrd{\nu}{}{}{1} dx' \notag \\
&=  (\star''') \operatorname{Vol}(\nu^{-1}M \cap \operatorname{Iw}_n(p^{\beta}) ) \cdot W_{\varphi_p}\matrd{\nu}{}{}{1},\label{eq:zeta vanish 2}
\end{align}
where $(\star''') \neq 0$ depends only on $\chi$, $t$, $p$, and $s$. In the penultimate equality we use Iwahori-invariance of $W_{\varphi_p}$. 

We consider two cases:
\begin{enumerate}[(1)]\s
	\item If $\nu \not\in M$, then $W_{\varphi_p}\smallmatrd{\nu}{}{}{1} = 0$, thus $\cQ = 0$. In particular, both sides of \eqref{eq:zeta vanishing} are 0, so Proposition \ref{prop:zeta vanishing} holds.
	\item If $\nu \in M$, then $\nu \operatorname{Iw}_n(p^{\beta})$ is a compact open subset of $\GL_n(\mathbb{Q}_p)$, and it is contained in $M$. This means $\operatorname{Iw}_n(p^{\beta}) \subset \nu^{-1}M$, so the volume above is $\operatorname{Vol}(\operatorname{Iw}_n(p^{\beta}))$ which is non-zero. Then Proposition \ref{prop:zeta vanishing} follows from \eqref{eq:zeta vanish 2}. \qedhere
\end{enumerate}
\end{proof}

\begin{corollary}\label{cor:non-vanishing for s}
	If $\zeta_p(ut\cdot\varphi_p,\chi_p,s_0) \neq 0$ for some $s_0 \in \C$, then $\zeta_p(ut\cdot \varphi_p,\chi_p,s) \neq 0$ for \emph{all} $s \in \C$.
\end{corollary}
\begin{proof}
	Non-vanishing of $W_{\varphi_p}\smallmatrd{\nu_\beta(t)}{}{}{1}$ is independent of $s$.
\end{proof}

\begin{corollary}\label{cor:zeta vanishing}
	If $P$ is a spin parabolic and $P$ is not contained in the $(n,n)$-parabolic, then for all $\varphi_p \in \pi_p^{\Iw}$    and $s \in \C$, we have
	\[
		\zeta_p(ut_P^\beta \cdot \varphi_p, \chi_p,s) = 0.
	\]
\end{corollary}
\begin{proof}
	We apply Proposition \ref{prop:zeta vanishing} with $t = t_P^\beta$, which we write as $\smallmatrd{z_1}{}{}{z_2}$ as above.
	
	 Suppose $P$ has type $(n_1,...,n_k)$. As $P$ is spin, $(n_1,...,n_k)$ is symmetric, whence
	\begin{equation}\label{eq:t_P swap}
		t_P = p^{k-1}w_{2n} t_P^{-1} w_{2n}.
	\end{equation}
Equation \eqref{eq:t_P swap} implies that $z_2 = p^{\beta(k-1)} w_n z_1^{-1} w_n$. Thus, for $\nu_\beta(t_P^\beta)$ as above, we have
\begin{equation}\label{eq:v t_P}
\nu_\beta(t_P^\beta) = p^{-\beta} z_2^{-1} w_n z_1 = p^{-\beta k} w_n z_1^2. 
\end{equation}
Let $[k/2]$ be the floor of $k/2$. Then $p^{2\beta [k/2]}$ is the largest power of $p$ which divides $z_1^2$ (so that one remains in $M_n(\Zp)$). Hence $\nu_\beta(t_P^\beta) \in M_n(\Zp)$ if and only if $k$ is even. As $P$ is spin, this happens if and only if $P$ is contained in the $(n,n)$-parabolic. Since (by \cite[Lem.\ 5.1]{BDGJW}) the support of $W_{\varphi_p}\smallmatrd{y}{}{}{1}$ is in $M \subset M_n(\Zp)$, the statement follows by Proposition \ref{prop:zeta vanishing}. 
\end{proof}


\subsection{Non-vanishing for $P$-spin eigenvectors}

Let  $\tilde\pi^P_p$ be a $P$-spin $P$-refinement. Suppose $P \subset Q$, the $(n,n)$-parabolic. We now construct $\varphi_p \in \tilde\pi_p^P$ such that $W_{\varphi_p}\smallmatrd{\nu_\beta(t_P^\beta)}{}{}{1} \neq 0$.

\subsubsection{Explicit eigenvectors}
We first give eigenvectors in principal series representations, generalising \cite[\S7.1]{BDGJW}. Throughout $\pi_p = \Ind_B^G\UPS$ is irreducible with regular semisimple Satake parameter, with $\UPS$ spin in the sense of Definition \ref{def:spin}.

We recap (but slightly modify) some notation from \cite{BDGJW}. Let $\cW_n$ be the Weyl group of $\GL_n$.  From now on we always view Weyl elements of $\cW_G$ (resp.\ $\cW_n$) as elements of $G(\mathbb{Z}_p)$ (resp.\  $\GL_n(\mathbb{Z}_p)$).  Recall $w_n$ is the longest element in $\cW_n$, and  $\tau = \smallmatrd{1}{}{}{w_n} \in \cW_G$. 
\begin{itemize}\s
	\item For any $w,\nu \in \cW_G$, let $f^{\nu}_w \in \operatorname{Ind}_B^G\theta^\nu$ be the (unique) Iwahori-invariant function supported on $B(\mathbb{Q}_p) w \operatorname{Iw}_G$ with $f_w^{\nu}(w) = p^{n(n-1)}$.
	
	\item  For $\rho \in \cW_n$, let $w(\rho) = \smallmatrd{}{w_n}{\rho}{}$, and (noting the difference to \cite[Def.\ 7.6]{BDGJW}) let
	\[
	F^\nu_\rho = f^{\nu}_{w(\rho)} \in \Ind_B^G (\UPS^\nu).
	\]
\end{itemize}

The relevance of these vectors is captured by \cite[Prop.\ 7.4]{BDGJW}, where we showed:
\begin{proposition}
	 Let $\tilde\pi_{\nu} = (\pi,\alpha_\nu) \defeq \Psi_{\UPS}^{-1}(\nu)$. Then $f_{w_{2n}}^\nu = F_{w_n}^\nu \in \Ind_B^G \UPS^\nu$ is an Iwahori-invariant $\alpha_\nu$-eigenvector.
\end{proposition}

We now define parahoric-level analogues. Recall $\cW_{L_P}$ is the Weyl group of the levi $L_P$ of $P$.  For $w \in \cW_G$, let $[w] \in \cW_G / \cW_{L_P}$ denote the corresponding coset.

Since $P \subset Q$, it is a $(k_1, \dots, k_r, k_r, \dots, k_1)$-parabolic for some $k_i$ with $k_1+\dots + k_r = n$. Let $\cW_{\mathbf{k}} \subset \cW_n$ denote the Weyl group associated with the Levi of the $(k_1, \dots, k_r)$-parabolic in $\GL_n$. For $\rho \in \cW_n$, let $[\rho]' \in \cW_n/\cW_{\mathbf{k}}$ denote the corresponding coset. 

\begin{itemize}	\s	
	\item For $w, \nu \in \cW_G$, let $h^{\nu}_{[w]} \in \operatorname{Ind}_B^G\theta^{\nu}$ denote the $J_P$-invariant function supported on $B(\mathbb{Q}_p) w J_P$ normalised so that $h^{\nu}_{[w]}(w) = p^{n(n-1)}$.  Writing $B(\Qp) w J_P$ as a union of sets of the form $B(\Qp) w' \operatorname{Iw}_G$,  we have
	\[
	h^{\nu}_{[w]} = \sum_{w' \in \cW_G, \ [w'] = [w]} f^{\nu}_{w'}.
	\]
	In particular, $h^{\nu}_{[w]} = h^{\nu}_{[w']}$ if $[w] = [w']$.

	\item For $\rho \in \cW_n$, we set
	\[
	H^{\nu}_{[\rho]'} = h^{\nu}_{[w(\rho)]}. 
	\]
\end{itemize}

\begin{proposition}\label{prop:parahoric eigenvector}
	Let $\tilde\pi_\nu^P = (\pi, \alpha_\nu^P) \defeq (\Psi_{\UPS}^P)^{-1}([\nu])$. Then $h_{[w_{2n}]}^\nu = H_{[w_n]'}^\nu \in \Ind_B^G \UPS^\nu$ is a $J_P$-invariant $\alpha_\nu^P$-eigenvector.
\end{proposition}
\begin{proof}
	Identical to \cite[Prop.\ 7.4]{BDGJW} or \cite[Lem.\ 3.6]{DJR18}.
\end{proof}
If $\nu = 1$, we drop the superscript $\nu$, and simply write $f_w, F_\rho, h_{[w]}, H_{[\rho]'}$.

We return to our fixed $P$-spin $P$-refinement $\tilde\pi^P = (\pi,\alpha^P)$. 

\begin{lemma}\label{lem:choice of UPS eigenvector}
We may choose a spin $\UPS$ so that $\varphi_p \defeq H_{[w_n]'} \in \tilde\pi_p^P$ is an $\alpha^P$-eigenvector.
\end{lemma}
\begin{proof}
	By definition $\Psi_{\UPS}^P(\tilde\pi^P) = [\sigma] \in \cW_G/\cW_{L_P},$ for some $\sigma \in \cW_G^0$. After renormalising $\UPS$ by $\sigma$ (as in Remarks \ref{rem:change of UPS} and \ref{rem:P-spin ind of UPS}) we may assume $\sigma = 1$; as $\sigma \in \cW_G^0$ such a $\UPS$ is still spin by Definition \ref{def:spin}. The result follows from Proposition \ref{prop:parahoric eigenvector}.
\end{proof}

\subsubsection{Intertwining maps} 
We now have an eigenvector $H_{[w_n]'} \in \Ind_B^G\UPS$. To transfer this into the Shalika model $\cS_{\psi_p}^{\eta_p}(\tilde\pi_p)$, we must write down an explicit Shalika intertwining. 

If $\Theta$ is an unramified character satisfying $\Theta_i\Theta_{n+i} = \eta_p$ for all $i$, Ash--Ginzburg \cite[(1.3)]{AG94} have constructed such an explicit $\cS : \Ind_B^G \Theta \to \cS_{\psi_p}^{\eta_p}(\pi_p)$, given by
\begin{equation}\label{eq:AG}
	\cS(f)(g) \defeq \int_{\GL_n(\Zp)}\int_{M_n(\Qp)} f\left[\smallmatrd{}{1}{1}{}\smallmatrd{1}{X}{}{1}\smallmatrd{k}{}{}{k}g\right] \psi^{-1}(\mathrm{tr}(X))\eta^{-1}(\det(k)) dXdk.
\end{equation}

Here we encounter a problem: our choice of $\UPS$ does not satisfy the Ash--Ginzburg condition; rather, $\UPS^\tau$ does, where $\tau = \mathrm{diag}(1,w_n)$. We know $\Ind_B^G\UPS$ and $\Ind_B^G \UPS^\tau$ are isomorphic, but to use \eqref{eq:AG}, we must compute what this isomorphism does to the eigenvector $\varphi_p$ from Lemma \ref{lem:choice of UPS eigenvector}. We do so by generalising \cite[\S7.3]{BDGJW}, using work of Casselman.

	Let $\nu  = \smallmatrd{1}{}{}{\nu'} \in \cW_G$ and $s = \smallmatrd{1}{}{}{s'} \in \cW_G$ be a simple reflection. Suppose that $s$ corresponds to the simple transposition $(a , a+1)$ for $a \geq n+1$. Set $\theta(s) \defeq \theta_a(p)/\theta_{a+1}(p)$ and
	\begin{equation}\label{eq:c_s}
	c_s(\theta^{\nu}) \defeq \frac{1 - p^{-1} \theta^{\nu}(s)}{1-\theta^{\nu}(s)} .
	\end{equation}
	Note $c_s(\theta^{\nu})$ is well-defined as $\theta^{\nu}$ is regular, and \emph{always} non-zero as $\operatorname{Ind}_B^G\theta^{\nu}$ is irreducible. 
	
Let $l$ denote the Bruhat length function on $\cW_G$. 	Then Casselman \cite[Thm.\ 3.4]{Cas80} shows that there are intertwinings $T^\nu_s \colon \operatorname{Ind}_B^G \theta^{\nu} \to \operatorname{Ind}_B^G \theta^{\nu s^{-1}}$
 with the following property:
	\begin{equation}\label{eq:casselman}
	T^{\nu}_s ( f^{\nu}_w ) = \left\{ \begin{array}{cc} p^{-1} f^{\nu s^{-1}}_{sw} + (c_s(\theta^{\nu}) - 1) f^{\nu s^{-1}}_w & \text{ if } l(sw) > l(w) \\ f^{\nu s^{-1}}_{sw} + (c_s(\theta^{\nu}) - p^{-1}) f_w^{\nu s^{-1}} & \text{ if } l(sw) < l(w) \end{array} \right.
	\end{equation}
The eigenvector $H_{[w_n]'}$ is a \emph{sum} of $f_w^\nu$'s as $w$ ranges over a coset in $\cW_G/\cW_{L_P}$. The following allows us to apply a case of \eqref{eq:casselman} consistently to $f_w^\nu$ for every $w$ in a $\cW_{L_P}$-coset.

\begin{lemma} \label{SimpleReflectionLemma}
	Let $s \in \cW_G$ be a simple reflection, and let $w \in \cW_G$. Then exactly only one of the following possibilities can occur:
	\begin{enumerate}[(1)]\s
		\item $s w \cW_{L_P} = w \cW_{L_P}$, whence left multiplication by $s$ permutes $w \cW_{L_P}$;
		\item $s w \cW_{L_P} \neq w \cW_{L_P}$ and $l(sv) < l(v)$ for all $v \in w \cW_{L_P}$;
		\item $s w \cW_{L_P} \neq w \cW_{L_P}$ and $l(sv) > l(v)$ for all $v \in w \cW_{L_P}$.
	\end{enumerate}
\end{lemma}
\begin{proof}
If $s w \cW_{L_P} = w \cW_{L_P}$, (1) occurs; so suppose $s w \cW_{L_P} \neq w \cW_{L_P}$. 

Let $w_{\operatorname{min}}$ and $v_{\operatorname{min}}$ be the unique minimal length representatives in $w \cW_{L_P}$ and $s w \cW_{L_P}$ respectively; properties of such elements are described in \cite[\S1.10]{Hum90}. As $s$ is simple, we must have $l(sw_{\min}) = l(w_{\min}) \pm 1$; so we have two possibilities:

	\medskip
	
\textbf{Possibility 1:} $l( s w_{\operatorname{min}} ) = l(w_{\min}) - 1 < l( w_{\operatorname{min}})$.

	As $sw_{\min} \in sw\cW_{L_P}$, there is a unique $x \in \cW_{L_P}$ such that $s w_{\operatorname{min}} = v_{\operatorname{min}} \cdot x$. We have $l(s w_{\operatorname{min}}) = l(v_{\operatorname{min}}) + l(x)$. As $l(x) \geq 0$, we have 
	\begin{equation}\label{eq:vmin wmin 1}
	l(v_{\operatorname{min}}) \leq l(sw_{\min}) < l(w_{\operatorname{min}}).
	\end{equation}
	
	On the other hand, we can write $v_{\operatorname{min}} = s y$ for some $y \in w \cW_{L_P}$. Again, we either have $l(v_{\operatorname{min}}) = l(y) \pm 1$. We also have $l(w_{\operatorname{min}}) \leq l(y)$ by minimality of $l(w_{\operatorname{min}})$.  If $l(v_{\min}) = l(y) + 1$, then $l(w_{\operatorname{min}}) \leq l(y) < l(v_{\operatorname{min}})$, contradicting \eqref{eq:vmin wmin 1}. Hence $l(v_{\operatorname{min}}) = l(y) - 1$. But then
	\[
	l(v_{\min}) < l(w_{\operatorname{min}}) \leq l(y) = l(v_{\operatorname{min}}) + 1 .
	\]
	This can only happen if $l(y) = l(w_{\operatorname{min}}) = l(v_{\min}) + 1$. Therefore $y = w_{\operatorname{min}}$ (by uniqueness of the minimal length representative), and $v_{\operatorname{min}} = s w_{\operatorname{min}}$. 
	
	Now take any $v \in w \cW_{L_P}$. There are unique $X,Y \in \cW_{L_P}$ such that $v = w_{\operatorname{min}} X$ and $sv = v_{\operatorname{min}} Y = s w_{\operatorname{min}} Y$. By uniqueness, we must have $X=Y$. Finally, we now see that
	\[
	l(s v) = l(s w_{\operatorname{min}} ) + l(X) < l(w_{\operatorname{min}}) + l(X) = l(v),
	\]
	so we are in case (2) of the Lemma.
	
	\medskip
	
	\textbf{Possibility 2:}  $l( s w_{\operatorname{min}} ) = l(w_{\min}) + 1 > l( w_{\operatorname{min}})$. 

We break this up into three further cases:
	\begin{itemize}
		\item[(a)] If $l(v_{\operatorname{min}}) > l(w_{\operatorname{min}})$, then we must have $l(v_{\operatorname{min}}) \geq l(s w_{\operatorname{min}})$. Minimality of $l(v_{\min})$ forces equality, hence $v_{\operatorname{min}} = s w_{\operatorname{min}}$ by uniqueness of the minimal length representative. Then if $v \in w \cW_{L_P}$, as above we must have $v = w_{\operatorname{min}} X$ and $sv = s w_{\operatorname{min}} X$ for some (unique) $X \in \cW_{L_P}$. Hence for any $v \in w\cW_{L_P}$, we have
		\[
		l(sv) = l(sw_{\operatorname{min}}) + l(X) > l(w_{\operatorname{min}}) + l(X) = l(v),
		\]
		whence we are in case (3) of the lemma.
		\item[(b)] If $l(v_{\operatorname{min}}) = l(w_{\operatorname{min}})$, then let $w_{\operatorname{min}} = s_1 \cdots s_r$ and $v_{\operatorname{min}} = s_1' \cdots s_r'$ be reduced word expressions for these elements. We can write $s w_{\operatorname{min}} = v_{\operatorname{min}} \cdot t$ for some unique $t \in \cW_{L_P}$. Moreover, since $l(v_{\operatorname{min}}) + l(t) = l(sw_{\operatorname{min}}) = l(w_{\operatorname{min}}) + 1$, we see that $l(t) = 1$ and hence $t$ is simple. We must therefore have $w_{\operatorname{min}} < sw_{\operatorname{min}} = v_{\operatorname{min}} t$ in the (strong) Bruhat order. 
		
		Since $s_1' \cdots s_r' t$ is a reduced word for $v_{\operatorname{min}} t$, we find that $s_1 \cdots s_r$ occurs inside this word. If 
		\[
		s_1 \cdots s_r = s_1' \cdots \hat{s}_i' \cdots s'_r t
		\]
		where $\hat{\cdot}$ denotes omission of the term, then we see that $w_{\operatorname{min}} \in s_1' \cdots \hat{s}_i' \cdots s'_r \cW_{L_P}$, which contradicts the fact that $w_{\operatorname{min}}$ is a minimal length representative. Hence we must have $s_1 \cdots s_r = s_1' \cdots s_r'$, hence $w_{\operatorname{min}} = v_{\operatorname{min}}$. But this contradicts the assumption that $sw \cW_{L_P} \neq w \cW_{L_P}$. So this case can never occur.
		\item[(c)] If $l(v_{\operatorname{min}}) < l( w_{\operatorname{min}})$, then write $v_{\operatorname{min}} = s y$ for some $y \in w \cW_{L_P}$. Arguing as in Possibility 1, this would imply $y = w_{\operatorname{min}}$, hence $l(s w_{\operatorname{min}}) < l(w_{\operatorname{min}})$, which is a contradiction to the premise of Possibility 2. Thus (c) also never occurs.
	\end{itemize}
Case (a) must thus occur, giving case (3) of the lemma, completing the proof.
\end{proof}

\begin{lemma} \label{lem:casselman}
	There exists an intertwining 
	\[
	M_{\tau} \colon \operatorname{Ind}_B^G \theta \to \operatorname{Ind}_B^G \theta^{\tau} 
	\]
	such that
	\[
	M_{\tau}(H_{[w_n]'}) = H^{\tau}_{[1]'} + \sum_{\substack{x \in \cW_n/\cW_{\mathbf{k}} \\ x \neq [1]'}} c_x H^{\tau}_x
	\]
	for some $c_x \in \mathbb{C}$ (note the sum may be empty).
\end{lemma}

\begin{proof}
	Let $\rho \in \cW_n$, and $s = \smallmatrd{1}{}{}{s'} \in \cW_G$ a simple reflection. We apply \eqref{eq:casselman} in two cases:
	\begin{enumerate}[(1)]
		\item Suppose $s w(\rho) \cW_{L_P} = w(\rho) \cW_{L_P}$. Then by Lemma \ref{SimpleReflectionLemma}(1), there exist $w_1, \dots, w_b \in w(\rho) \cW_{L_p}$ such that 
		\[
		w(\rho) \cW_{L_p} = \{ w_1, \dots, w_b, sw_1, \dots, sw_b \}
		\]
		with all the elements in the set distinct. Then we have
		\[
		T^{\nu}_s( f^{\nu}_{w_i} + f^{\nu}_{sw_i} ) = c_s(\theta^{\nu}) (f^{\nu s^{-1}}_{w_i} + f^{\nu s^{-1}}_{sw_i})
		\]
		hence $T^{\nu}_s( H^{\nu}_{[\rho]'} ) = c_s(\theta^{\nu}) H^{\nu s^{-1}}_{[\rho]'}$.
		\item Suppose $s w(\rho) \cW_{L_P} \neq w(\rho) \cW_{L_P}$. Then by parts (2) and (3) in Lemma \ref{SimpleReflectionLemma}, we have
		\[
		T^{\nu}_s (H^{\nu}_{[\rho]'}) = \left\{ \begin{array}{cc} p^{-1} H^{\nu s^{-1}}_{[s'\rho]'} + (c_s(\theta^{\nu}) - 1) H^{\nu s^{-1}}_{[\rho]'} & \text{ if } l(sw(\rho)) > l(w(\rho)) \\ H^{\nu s^{-1}}_{[s'\rho]'} + (c_s(\theta^{\nu}) - p^{-1}) H_{[\rho]'}^{\nu s^{-1}} & \text{ if } l(sw(\rho)) < l(w(\rho)). \end{array} \right.
		\]
	\end{enumerate}
Crucially the only terms that appear here are of the form $H_x^{\nu s^{-1}}$ for $x \in \cW_n/\cW_{\mathbf{k}}$.

	Now write $w_n = s'_1 \cdots s_c'$, so $\tau = s_c^{-1} \cdots s_1^{-1}$ with $s_i = \smallmatrd{1}{}{}{s_i'}$. We may assume that the factorisation of $w_n$ is chosen such that $s_c' \cdots s_{b+1}'$ is the minimal length representative of the coset $w_n \mathcal{W}_{\mathbf{k}} \subset \mathcal{W}_n$ and $s_i'$ ($i=1, \dots, b$) are simple reflections in $\mathcal{W}_{\mathbf{k}}$, for some integer $1 \leq b \leq c$. Composing, we have
	\[
	M_\tau = T^{s_c^{-1} \cdots s_2^{-1}}_{s_1} \circ \cdots \circ T^{s_c^{-1}}_{s_{c-1}} \circ T^{1}_{s_c} : \pi_p = \Ind_B^G \UPS \longrightarrow \Ind_B^G (\UPS^\tau).
	\]
	Iterating the formulae, we see $M_\tau(H_{[w_n]'})$ is a linear combination of $H^\tau_x$'s for $x \in \cW_n/\cW_{\mathbf{k}}$. The coefficient of $H_{[1]'}^\tau$ is the product of $\prod_{i=1}^b c_{s_i}(\theta^{s_c^{-1} \cdots s_{i+1}^{-1}})$ and a power of $p$, and we saw after \eqref{eq:c_s} that this product is non-zero. Therefore, we may renormalise $M_\tau$ to make this coefficient equal to $1$.
\end{proof}

\subsubsection{Non-vanishing}
With set-up as above, choose a spin $\theta$ so that $H_{[w_n]'} \in \Ind_B^G\UPS$ is an eigenvector for $\tilde\pi^P$. We now show that $H_{[w_n]'}$ does not vanish under the composition
\[
	\Ind_B^G \UPS \mathrel{\mathop{\xrightarrow{\qquad\qquad}}^{M_{\tau}}_{\text{Lemma \ref{lem:casselman}}}}\Ind_B^G \UPS^\tau \mathrel{\mathop{\xrightarrow{\qquad\qquad}}^{\cS}_{\text{\eqref{eq:AG}}}}\cS_{\psi_p}^{\eta_p}(\pi_p) \xrightarrow{\ \smallmatrd{\nu_\beta(t_P^\beta)}{}{}{1}} \C.
\] 

Write  $t_P^\beta = \operatorname{diag}(z_1, z_2)$ as before. Recalling $P$ is the $(k_1,...,k_r,k_r,...,k_1)$-parabolic, by \eqref{eq:v t_P} we have $\nu_\beta(t_P^\beta) = p^{-2\beta r}w_n z_1^2$, for $r$ as \emph{op.\ cit}. Note $z \defeq p^{-\beta r}z_1$ has coefficients in $\mathbb{Z}_p$ (as $P \subset Q$; see the proof of Corollary \ref{cor:zeta vanishing}). 

\begin{lemma}(cf.\ \cite[Prop.\ 7.9]{BDGJW}). \label{FactorisationLemma}
Let $\delta \in \cW_n$.	We have 
	\[
	\smallmatrd{}{1}{1}{}\smallmatrd{1}{X}{}{1}\smallmatrd{k}{}{}{k}\smallmatrd{w_n z^{2}}{}{}{1} \in B(\mathbb{Q}_p) \smallmatrd{}{w_n}{\delta w_n}{} J_P
	\]
	if and only if:
	\begin{itemize}\s
		\item $[\delta w_n]' = [1]'$,
		\item $k \in B_n(\mathbb{Z}_p) w_n J_{\mathbf{k}'}$, where $J_{\mathbf{k}'}$ is the parahoric in $\GL_n$ of type $\mathbf{k}' = (k_r, \dots, k_1)$,
		\item and $k^{-1}X \in w_n z^{2} M_n(\mathbb{Z}_p)$.
	\end{itemize}
\end{lemma}
\begin{proof}
	The proof closely follows that of \cite[Prop.\ 7.9]{BDGJW}, and we merely indicate the small differences here. The ``if'' direction is identical to \emph{op.\ cit}.
	
	For the ``only if'' direction, we again start from (7.10) \emph{op.\ cit}.\ (where now the matrix $\smallmatrd{a}{b}{c}{d}$ is in $J_P$). If we can show $[\delta w_n]' = [1]'$,  then the remaining conditions follow as in (1)--(4) following (7.10) \emph{op.\ cit}. If $P = Q$, then $[\delta w_n]' = [1]'$ is always satisfied. Suppose then that $P \neq Q$ (hence $r > 1$), and that $[\delta w_n]' \neq [1]'$, i.e.\ $\delta w_n \not\in \cW_{\mathbf{k}}$. 
	
	We have the following analogue of Claim 7.12 \emph{op.\ cit}.: let $Y_P \defeq \{k_1,k_1+k_2, ..., k_1+\cdots + k_{r-1}\}$. Then $\cW_{\mathbf{k}} = \cap_{m \in Y_P} \cW_{(m,n-m)}$. Thus $\delta w_n \not\in \cW_{(m,n-m)}$ for some $m \in Y_P$, whence
	\begin{equation}\label{eq:nothing}
	B_n(\mathbb{Q}_p) \delta w_n \overline{J}_m \cap B_n(\mathbb{Q}_p) \overline{J}_m = \varnothing
	\end{equation}
	where $\overline{J}_m$ is the opposite parahoric in $\GL_n(\Zp)$ of type $(m, n-m)$. 
	
	Now factorise $z^{2} = t_{p, m} \mu$. 
	Via the same proof of the analogous statement in \cite{BDGJW}, we can show $k w_n \mu \in B_n(\mathbb{Q}_p) \overline{J}_{m} \cap B_n(\mathbb{Q}_p) \delta w_n \overline{J}_m$, a contradiction to \eqref{eq:nothing}. We deduce $[\delta w_n]' = [1]'$, and hence the lemma.
\end{proof}

Recall $\tilde\pi_p^P = (\pi_p,\alpha^P)$ is a $P$-spin $P$-refinement, with $P \subset Q$. We finally obtain:

\begin{proposition}\label{prop:eigenvector non-vanishing}
	The element $\cS(M_\tau(H_{[w_n]'}))$ is an $\alpha^P$-eigenvector in $\cS_{\psi_p}^{\eta_p}(\pi_p)$, and
	\[
	\mathcal{S}(M_{\tau}(H_{[w_n]'}))\smallmatrd{\nu_\beta(t_P^\beta)}{}{}{1} \neq 0.
	\]
\end{proposition}
\begin{proof}
	This is an $\alpha^P$-eigenvector by Lemma \ref{lem:choice of UPS eigenvector} and Hecke-equivariance of $M_\tau$ and $\cS$. Non-vanishing follows exactly the same proof as \cite[Prop.\ 7.12]{BDGJW}. Precisely, we show that
	\[
	\mathcal{S}(M_{\tau}(H_{[w_n]'}))\smallmatrd{\nu_\beta(t_P^\beta)}{}{}{1} = \mathcal{S}( H^{\tau}_{[1]'} )\smallmatrd{\nu_\beta(t_P^\beta)}{}{}{1} \neq 0.
	\]
	Here the first equality holds as Lemma \ref{lem:casselman} expresses $M_\tau(H_{[w_n]'})$ as a linear combination of $H_x^\tau$'s, and Lemma \ref{FactorisationLemma} shows that the the integrand of $\cS$ (in \eqref{eq:AG}) vanishes on each of these except $H_{[1]'}^\tau$. Non-vanishing is a direct calculation.
\end{proof}

\subsection{Proof of Theorem \ref{thm:2 implies 1}}
We must show that if $P$ is contained in the $(n,n)$-parabolic, and $\tilde\pi_p^P$ is a $P$-spin refinement, then there exists $\varphi_p \in \tilde\pi_p^P$ such that $\zeta_p(ut_P^\beta \cdot \varphi_p, \chi_p, s) \neq 0$. By Proposition \ref{prop:zeta vanishing}, it suffices to prove $W_{\varphi_p}\smallmatrd{\nu_\beta(t_P^\beta)}{}{}{1} \neq 0$, where $W_{\varphi_p}  = \cS_{\psi_p}^{\eta_p}(\varphi_p)$ for some Shalika intertwining $\cS_{\psi_p}^{\eta_p}$. Since the $\alpha^P$-eigenspaces in $\pi_p^{\Iw}$ and $\cS_{\psi_p}^{\eta_p}(\pi^{\Iw})$ are both 1-dimensional, it suffices to exhibit \emph{any} $\alpha^P$-eigenvector in the Shalika model with this non-vanishing property. Such an eigenvector is given by Proposition \ref{prop:eigenvector non-vanishing}. \qed

\section{Proof of Theorem \ref{thm:p-refined FJ}}\label{sec:other direction}

Finally we use our study of the symplectic locus to prove a result towards the remaining implication $(1') \Rightarrow (2')$ in Conjecture \ref{conj:FJ}.  
If the hypotheses of Theorem \ref{thm:p-refined FJ} are satisfied, this furnishes a `good' choice of Iwahori refinement $\tilde\pi$ above $\tilde\pi^P$. Key to our proof is:

\begin{proposition}\label{prop:p-refined FJ}
Suppose (1$'$) of Conjecture \ref{conj:FJ} holds. There is an $(\#X_P+1)$-dimensional symplectic family $\sC$ through $\tilde\pi$ in the $\GL_{2n}$-eigenvariety $\sE_{K_B}^G$, varying over $\sW_{0,\lambda_\pi}^P$.	
\end{proposition}

\emph{Proof of Theorem \ref{thm:p-refined FJ}, given Proposition \ref{prop:p-refined FJ}:}  
Suppose (1$'$) is satisfied in Conjecture \ref{conj:FJ}. By Corollary \ref{cor:zeta vanishing}, and \eqref{eq:euler product}, we see $P$ must be contained in the $(n,n)$-parabolic. By Theorem \ref{thm:FJ}, we deduce that $L(\pi \times \chi,s+1/2) \neq 0$, and that $\pi$ is symplectic. Thus to deduce (2$'$) in Conjecture \ref{conj:FJ} it suffices to prove $\tilde\pi$ (hence $\tilde\pi^P$) is $P$-spin. 

Let $\Omega \defeq w(\sC)$, open of maximal dimension in $\sW_{0,\lambda_\pi}^P$. If $\tilde\pi$ is not $P$-spin, then it is optimally $P'$-spin for some spin parabolic $P' \not\subset P$. Then Theorem \ref{thm:shalika obstruction} shows that $w(\sC) \subset \sW_{0,\lambda_\pi}^{P'}$, hence 
\[
\Omega = w(\sC) \subset  \sW_{0,\lambda_\pi}^{P'} \cap  \Omega \subsetneq \Omega,
\]
a contradiction; so $\tilde\pi$ is $P$-spin.  \qed

\medskip

The proof of Proposition \ref{prop:p-refined FJ} occupies the rest of this section.

\subsection{Big evaluation maps: $p$-adic interpolation of branching laws}\label{sec:branching laws}

Our proof closely follows \cite[Thm.\ 13.6]{BDGJW}, which treated the case $P=B$; and \cite[Thm.\ 7.6(a--c)]{BDW20}, which treated the analogous result in the $(n,n)$-parabolic eigenvariety. These works constructed \emph{evaluation maps} on overconvergent cohomology groups, over affinoids $\Omega$ in the weight space, valued in torsion-free $\cO_\Omega$-modules. Non-vanishing of these maps puts strong constraints on the structure of the overconvergent cohomology, and was shown to produce symplectic families in the eigenvariety. We refer the reader to these works for any undefined notation. 

Let $K = K^pK_p \subset G(\A_f)$ be open compact, with $K_p \subset J_P$ inside the $P$-parahoric subgroup. As in \cite[\S2.10]{BDW20}, choices at infinity fix for all $K$ (non-canonical) embeddings
\begin{equation}\label{eq:cohomology}
\pi_f^K \longhookrightarrow \hc{t}(S_K,\sV_{\lambda_\pi}^\vee(\overline{\Q}_p))_{\tilde\pi}, \qquad \varphi \mapsto \phi_\varphi,
\end{equation}
where the subscript $\tilde\pi$ denotes the $\tilde\pi$-eigenspace.
 
For a dominant weight $\lambda = (\lambda_1,...,\lambda_{2n})$, let 
\[
\mathrm{Crit}(\lambda) \defeq \{j \in \Z : -\lambda_{n+1}\geq j \geq -\lambda_n\}.
\]
 In \cite[\S4]{BDGJW}, to the data of $\lambda, P, \chi, j \in \mathrm{Crit}(\lambda)$, and $\eta = \eta_0|\cdot|^{\sw(\lambda)}$ with $\eta_0$ finite order, we constructed parahoric evaluation maps
\begin{equation}\label{eq:classical evaluation}
	\cE_{\lambda_\pi, P,\chi}^{j,\eta_0} : \hc{t}(S_K, \sV_\lambda^\vee(\overline{\Q}_p)) \longrightarrow \overline{\Q}_p.
\end{equation}
Let $\varphi^{(p)} = \otimes_{\ell\neq p}\varphi_\ell^{\mathrm{FJ}}$, as in the proof of Proposition \ref{prop:local vs global}. Then for any $\varphi_p \in \pi_p$, by \cite[Thm.\ 4.16]{BDGJW} we have
\begin{equation}\label{eq:critical value}
	\cE_{\lambda_\pi,P,\chi}^{j,\eta_0}\left(\phi_\varphi\right) = A_{\lambda_\pi,P,\chi}^{j} \cdot L\Big(\pi\times\chi,j+\tfrac{1}{2}\Big) \cdot \zeta_p\Big(ut_P^\beta\cdot \varphi_p, \chi_p, j+\tfrac12\Big),
\end{equation}
where $\varphi = \varphi^{(p)}\otimes \varphi_p \in \pi_f$ and $A_{\lambda,P,\chi}^{j}$ is a non-zero scalar.

In the rest of \S\ref{sec:branching laws} we will prove the following existence of a `big evaluation map', interpolating \eqref{eq:classical evaluation} as $\lambda$ varies over an $(\#X_P + 1)$-dimensional affinoid $\Omega = \mathrm{Sp}(\cO_\Omega) \subset \sW_{0,\lambda_\pi}^P$, which we henceforth fix. 

\begin{proposition}\label{prop:galois evs commute main diagram}
	Let $\beta \geq 1$, $\chi$ a Dirichlet character of conductor $p^\beta$, $\eta_0$ a Dirichlet character, and $j \in \mathrm{Crit}(\lambda_\pi)$. Then for any classical $\lambda \in \Omega$, we have $j_\lambda \defeq j-\sw(\lambda-\lambda_\pi)/2 \in \mathrm{Crit}(\lambda)$, and there exists an $\cO_\Omega$-module map $\cE_{\Omega,P,\chi}^{j,\eta_0} : \hc{t}(S_K,\sD_\Omega^P) \to \cO_\Omega$ such that for all classical $\lambda \in \Omega$, we have a commutative diagram
	\begin{equation}\label{eq:galois evs}
	\xymatrix@C=35mm{
		\hc{t}(S_K,\sD_\Omega^P) \ar[d]^-{\cE_{\Omega,P,\chi}^{j,\eta_0}}\ar[r]^-{r_\lambda} &    
		\hc{t}(S_K,\sV_\lambda^\vee(\overline{\Q}_p))  \ar[d]^-{\cE_{\lambda,P,\chi}^{j_\lambda,\eta_0}}\\
		\cO_\Omega 
		\ar[r]^-{\newmod{\m_\lambda}} &
		\overline{\Q}_p.}
	\end{equation}
\end{proposition}

\subsubsection{Recap of classical evaluation maps}
 Let $\iota : H \to G$ be the map $(h_1,h_2) \mapsto \smallmatrd{h_1}{}{}{h_2}$. The classical evaluation maps $\cE_{\lambda,P,\chi}^{j,\eta_0}$ were constructed as the composition of:

\begin{construction}\label{construction}\begin{enumerate}[(1)]\s
		\item Pull back classes twisted by $t_P^\beta$ under the map $\iota : H \to G$,
		\item Trivialise $\iota^*\sV_{\lambda}^\vee$ on each connected component and integrate over fundamental classes,
		\item Pass to scalars via a branching law for the critical integer $j$,
		\item Take the sum over connected components, weighted by $\chi$ and $\eta_0$.
	\end{enumerate}
\end{construction}

When $P = Q$ (resp.\ $P=B$), the construction of \eqref{eq:galois evs} was done in \cite[\S5-6]{BDW20} (resp.\ \cite[\S11-12]{BDGJW}). In that construction, we replaced the coefficients $\sV_\lambda^\vee$ in Construction \ref{construction} with $\sD_\Omega$. Of the four steps, the compatability of steps (1) and (2) for $\sD_\Omega$ and $\sV_\lambda^\vee$ is easy via \cite[Lemma 4.8]{BDGJW}, particularly Lemma 4.8. Step (4) is the same in both cases. This leaves (3), which we handle by an interpolation of branching laws.

\subsubsection{Explicit branching laws}
For integers $j_1,j_2$, let $V_{(j_1,j_2)}^H$ denote the 1-dimensional $H(\Zp)$-representation given by the character $\det_1^{j_1}\cdot\det_2^{j_2}$. Then we have \cite[Prop.\ 6.3.1]{GR2}, \cite[Lem.\ 5.2]{BDW20}
\[
j \in \mathrm{Crit}(\lambda) \ \  \iff \ \  \mathrm{dim}\  \mathrm{Hom}_{H(\Zp)}\big(V_\lambda^\vee, V^H_{(j,-\sw(\lambda)-j)}\big) = 1.
\] 
Via step (3) of Construction \ref{construction}, the map $\cE_{\lambda,P,\chi}^{j,\eta_0}$ depends on a choice of generator $\kappa_{\lambda,j}$ in this space, or dually, an element $v_{\lambda,j} \in V_{(-j,\sw(\lambda)+j)}^H \subset V_{\lambda}|_{H(\Zp)}$. For $p$-adic interpolation, we need  to choose such generators compatibly in $\lambda$. 
It is expedient to recall how we handled the Borel case in \cite[\S11.1]{BDGJW}; there we described explicit choices as follows. Define  weights
\begin{align}\label{eq:alpha_i}
	\alpha_{1} = (1,0,...,0,-1), \ \ \alpha_{2} = (1,1,0,...,0,-1,-1)&,\ \  ...,\ \  \alpha_{n-1} = (1,...,1,0,0,-1,...,-1),\notag\\
	\alpha_{0} = (1,...,1,1,...,1), \hspace{12pt} &\alpha_{n} = (1,...,1,0,...,0),
\end{align}
a $\Z$-basis for the pure algebraic weights. Note that if  $\lambda$ is a dominant algebraic weight then we can write uniquely
\[
\lambda = \lambda_\pi + \sum_{i=0}^n \mu_i \alpha_i, \qquad \mu_i \in \Z_{\geq 0},
\]
so that $\sw(\lambda) = \sw(\lambda_\pi) + 2\mu_0$. Note also that $j\in\mathrm{Crit}(\lambda_\pi)$ implies $j-\mu_0 = j - \sw(\lambda-\lambda_\pi)/2 \in \mathrm{Crit}(\lambda)$, yielding the condition in Proposition \ref {prop:galois evs commute main diagram}.

Via Notation 11.2 \emph{op.\ cit.}, for $1 \leq i \leq n-1$ let $v_{(i)} \in V_{\alpha_i}(\Qp)$ such that $H(\Zp)$ acts trivially, let $v_{(n),j} \in V_{\alpha_n}(\Qp)$ be such that $H(\Zp)$ acts as $\det_j$ (for $j = 1,2$), and fix a generator $v_{(0)} \in V_{\alpha_0}(\Qp)$. In Proposition 11.3 \emph{op.\ cit}.\ we showed
\begin{equation}\label{eq:v_lambda,j}
	v_{\lambda,j} \defeq [v_{(1)}^{\lambda_1-\lambda_2}] \cdot [v_{(2)}^{\lambda_2-\lambda_3}] \cdots [v_{(n-1)}^{\lambda_{n-1}-\lambda_n}] \cdot [v_{(n),1}^{-\lambda_{n+1}-j}] \cdot [v_{(n),2}^{\lambda_n + j}] \cdot [v_{(0)}^{\lambda_{n+1}}]  
\end{equation}
generates $V_{(-j,\sw(\lambda)+j)}^H(\Qp) \subset V_{\lambda}(\Qp)|_{H(\Zp)}$. Dualising, we obtain a map $\kappa_{\lambda,j} : V_\lambda^\vee \to V_{j,-\sw(\lambda)-j}^H$ that was used in the construction of $\cE_{\lambda,B,\chi}^{j,\eta_0}$ (see \cite[Rem.\ 4.14]{BDGJW}).

\subsubsection{$p$-adic interpolation}
We recap the main points of \cite[\S11]{BDGJW}, and simplify them; in that paper, we also incorporated cyclotomic variation, but we shall not need this generality.

For $p$-adic variation of \eqref{eq:v_lambda,j} we want to replace the algebraic weight $\lambda$ with a more general character $\kappa$ of $T(\Zp)$. In particular, we wish to make sense of $(\kappa_i-\kappa_{i+1})(v_{(i)})$. In Proposition 11.4 \emph{op.\ cit}.\ we showed that if we define 
\[
N^\beta(\Zp) \defeq N(p^\beta\Zp)\cdot u = \left\{n \in N(\Zp) : n \equiv \smallmatrd{1_n}{w_n}{0}{1_n} \newmod{p^\beta}\right\},
\]
then 
\begin{equation}\label{eq:N^1}
	v_{(i)}[N^\beta(\Zp)]  \subset 1+p^\beta\Zp,
\end{equation}
and hence $(\kappa_i-\kappa_{i+1})(v_{(i)}\big|_{N^\beta(\Zp)})$ is well-defined. This, and \eqref{eq:v_lambda,j}, motivates the definition 
\begin{align}\label{eq:v circ}
	w_{\kappa,\lambda_\pi} : N^\beta(\Zp) &\longrightarrow R^\times,\\
	g &\longmapsto v_{(0)}(g)^{\kappa_{n+1}}  \cdot \left[\prod_{i=1}^{n-1} v_{(i)}(g)^{\kappa_i - \kappa_{i+1}}\right] \cdot v_{(n),1}(g)^{-\lambda_{\pi,n+1}} \cdot v_{(n),2}(g)^{\lambda_{\pi,n}}.\notag
\end{align}
(In \cite{BDGJW}, the last two terms used $\kappa_{i}$ rather than $\lambda_{\pi,i}$, because we also wanted cyclotomic variation. Here we fix these terms, which allows us to fix $j$ and still obtain interpolation of $v_{\lambda,j_\lambda}$ as $\lambda$ varies; see \eqref{eq:j_lambda} below).

Now let $\Omega \subset \sW_0^G$, with universal character $\kappa_\Omega$ on $T(\Zp)$. For $j \in \mathrm{Crit}(\lambda_\pi)$, define a function $v_{\Omega,j} : N(\Zp) \to \cO_\Omega$ by
\[
v_{\Omega,j}(g) \defeq \left\{\begin{array}{cc} w_{\kappa_\Omega, \lambda_\pi}(g)\cdot \left(\tfrac{v_{(n),2}(g)}{v_{(n),1}(g)}\right)^j &: g \in N^\beta(\Zp),\\
	0 &: \text{otherwise}.
\end{array}\right. 
\]
Now suppose $\lambda$ is a classical weight, with $\sw(\lambda) = \sw(\lambda_\pi) + 2\mu_0$. Recall $j_\lambda = j-\mu_0 \in \mathrm{Crit}(\lambda)$. We know $\kappa_\Omega \newmod{\m_\lambda} = \lambda$ as characters of $T(\Zp)$, and one may formally verify that
\begin{equation}\label{eq:j_lambda}
	v_{\Omega,j} \newmod{\m_\lambda} = v_{\lambda,j_\lambda}|_{N^\beta(\Zp)}.
\end{equation} 
The function $v_{\Omega,j}$ extends to a unique element of $\cA_\Omega$, and dualising, we get a `$p$-adic branching law' $\kappa_{\Omega,j} : \cD_\Omega \to \overline{\Q}_p$ that, after restriction to $N^\beta(\Zp)$, formally interpolates the branching laws $\kappa_{\lambda,j}$ as $\lambda$ varies in $\Omega$.

In the construction of $\cE_{\Omega,B,\chi}^{j,\eta_0}$,  by \cite[Lem.\ 12.4]{BDGJW} the result of steps (1) and (2) (in Construction \ref{construction}, with $\sD_\Omega$ coefficients) was a distribution supported on $t_B^\beta N(\Zp)t_B^{-\beta}u \subset N^\beta(\Zp)$; so we could use $\kappa_{\Omega,j}$ to construct $\cE_{\Omega,P,\chi}^{j,\eta_0}$ (in Proposition 12.3).

\medskip 

We switch to a general parabolic $P$. Let $\cD_\Omega^{\beta,P} \subset \cD_\Omega^P$ be the subset of distributions supported on $N_P^\beta(\Zp) \defeq t_P^\beta N_P(\Zp) t_P^{-\beta}u$ (analogous to \cite[Def.\ 11.11]{BDGJW}). For a general parabolic $P$, by the same proof as \cite[Lem.\ 12.4]{BDGJW}, the output of steps (1) and (2) of Construction \ref{construction} lies in  (a quotient of) $\cD_\Omega^{\beta,P}$. Since $N_P^\beta(\Zp) \subset N^\beta(\Zp)$, we can define $v_{\Omega,j}^P : N_P(\zp) \to \overline{\Q}_p$ by
\[
v_{\Omega,j}^P(g) \defeq \left\{\begin{array}{cc} v_{\Omega,j}(g) &: g \in N_P^\beta(\Zp),\\
	0 &: \text{otherwise}.
\end{array}\right.
\]
The function $v_{\Omega,j}^P$ extends uniquely via the induction property \cite[Def.\ 3.11]{BDW20} to an element in $\cA_\Omega^P$, and hence dualises to a map
\[
\kappa_{\Omega,j}^P : \cD_\Omega^{\beta,P} \longrightarrow \cO_\Omega, \qquad \mu \mapsto \mu(v_{\Omega,j}^P).
\]
Again, formally, $\kappa_{\Omega,j}^P$ interpolates the branching laws $\kappa_{\lambda,j_\lambda}$ after restriction to $N_P^\beta(\Zp)$.

\emph{Proof of Proposition \ref{prop:galois evs commute main diagram}}: As in \cite[Rems.\ 4.14, 12.7]{BDGJW}, define $\cE_{\Omega,P,\chi}^{j,\eta_0}$ as the composition
\begin{equation}\label{eq:explicit overconvergent}
	\xymatrix@C=2mm{
		\hc{t}(S_K,\sD_\Omega^P) \ar[rrrr]^-{\oplus\mathrm{Ev}_{P,\beta,\delta}^{\cD_\Omega^P}}  &&&& \displaystyle\bigoplus_{[\delta]}(\cD_\Omega^{\beta,P})_{\Gamma_{\beta,\delta}^P} \ar[rrrrr]^-{\delta *\kappa_{\Omega,j}^P} &&&&&
		\displaystyle\bigoplus_{[\delta]}\cO_\Omega \ar[rrrrr]^-{\sum_{\brep}\chi(d)\Xi_{\brep}^{\eta_0}} &&&&&
		\cO_\Omega,
	}
\end{equation}
with $\mathrm{Ev}_{P,\beta,\delta}^{\cD_\Omega^P}$ the map of Definition 4.7 \emph{op.\ cit.}, which lands in $(\cD_{\Omega}^{\beta,P})_{\Gamma_{\beta,\delta}^P}$ exactly as in Lemma 12.4; $\kappa_{\Omega,j}^P$ descends to the coinvariants as in the proof of Proposition 12.3; and $\Xi_{d^{\eta_0}}$ is defined in Remark 4.14, all \emph{op.\ cit}., where any other undefined notation is explained. The three arrows in \eqref{eq:explicit overconvergent} correspond to (1-2), (3) and (4) in Construction \ref{construction} respectively. 

To deduce the claimed interpolation property in Proposition \ref{prop:galois evs commute main diagram}, observe that for any classical $\lambda \in \Omega,$ the diagram
\[
\xymatrix@C=2mm{
	\hc{t}(S_K,\sD_\Omega^P) \ar[rrrr]^-{\oplus\mathrm{Ev}_{P,\beta,\delta}^{\cD_\Omega}}\ar[d]^{r_\lambda}  &&&& \displaystyle\bigoplus_{[\delta]}(\cD_\Omega^{\beta,P})_{\Gamma_{\beta,\delta}^P} \ar[rrrrr]^-{\delta *\kappa_{\Omega,j}} \ar[d]^{r_\lambda}&&&&&
	\displaystyle\bigoplus_{[\delta]}\cO_\Omega \ar[rrrrr]^-{\sum_{\brep}\chi(d)\Xi_{\brep}^{\eta_0}}\ar[d]^{\newmod{\m_\lambda}} &&&&&
	\cO_\Omega\ar[d]^{\newmod{\m_\lambda}}\\
	\hc{t}(S_K,\sV_\lambda^\vee) \ar[rrrr]^-{\oplus\mathrm{Ev}_{P,\beta,\delta}^{V_\lambda^\vee}}  &&&& \displaystyle\bigoplus_{[\delta]}(V_\lambda^\vee)_{\Gamma_{\beta,\delta}^P} \ar[rrrrr]^-{\delta *\kappa_{\lambda,j_\lambda}} &&&&&
	\displaystyle\bigoplus_{[\delta]}L \ar[rrrrr]^-{\sum_{\brep}\chi(d)\Xi_{\brep}^{\eta_0}} &&&&&
	L
}
\]
commutes. For the first square, this is \cite[Lem.\ 4.8]{BDGJW}; the second is identical to Proposition 11.12 \emph{op.\ cit}; and the third is clear from the definition. Since the bottom row here is exactly $\cE_{\lambda,P,\chi}^{j_\lambda,\eta_0}$, this concludes the proof of Proposition \ref{prop:galois evs commute main diagram} (hence of Theorem \ref{thm:p-refined FJ}). \qed

\subsection{Tracing from Iwahoric to parahoric level}

The above `big evaluation' had the parabolic $P$ baked into it; it used the parahoric classical evaluation map, and $P$-parahoric distributions in the overconvergent cohomology. As in \cite{BDW20}, this is sufficient to study symplectic families through $\tilde\pi^P$ in the \emph{$P$-parabolic eigenvariety}, where we have analytic variation of some subset of the Hecke operators $U_{p,r}$. However, our study of the symplectic locus crucially used analytic variation of \emph{all} the $U_{p,r}$; in other words, it applies only to the Iwahori-level eigenvariety. We now port between the two.

There is a natural trace map $\mathrm{Tr} : \pi_p^{\Iw} \to \pi_p^{J_P}$, given by summing over translates by representatives of $J_P/\mathrm{Iw}$.

\begin{lemma}\label{lem:trace}
If $\tilde\pi = (\pi,\alpha)$ is an Iwahori refinement above the $P$-refinement $\tilde\pi^P = (\pi,\alpha^P)$, then $\mathrm{Tr}$ induces an isomorphism $\tilde\pi \isorightarrow \tilde\pi^P$.
\end{lemma}

\begin{proof}
As trace only acts at $p$, it suffices to prove $\mathrm{Tr} : \tilde\pi_p \isorightarrow \tilde\pi_p^P$. As the Satake parameter of $\pi_p$ is assumed regular, both sides are complex lines; so we need only check the map is well-defined and non-zero.

Let $\sigma = \Psi_{\UPS}(\tilde\pi)$. We have $\pi_p \cong \Ind_B^G (\UPS^\sigma)$, so it suffices to prove the result in $\Ind_B^G\UPS^\sigma$. Let $f_\sigma \in \Ind_B^G\UPS^\sigma$ be the (unique) Iwahori-invariant function supported on the big Bruhat cell $B(\Qp)\cdot w_{2n} \cdot \Iw$ with $f_\sigma(w_{2n}) = 1$. By \cite[Prop.\ 7.4]{BDGJW}, $f_\sigma$ is an $\alpha$-eigenvector, hence yields a generator of $\tilde\pi_p$. Under trace, this is mapped to a non-zero $J_P$-invariant vector supported on $B(\Qp)\cdot w_{2n} \cdot J_P$. But by the same arguments, this is an $\alpha^P$-eigenvector, hence the map on refinements is well-defined and non-zero.
\end{proof}

Let $K_B = K^p \Iw$ be an Iwahori-at-$p$ level, and $K_P = K^pJ_P$ a parahoric-at-$p$ level. We have natural trace maps from the cohomology of $S_{K_B}$ to $S_{K_P}$, which are functorial in maps between the coefficients.  Finally, we have a natural map $s_P : \cD_\Omega \twoheadrightarrow \cD_\Omega^P$ \cite[Prop.\ 4.8]{BW20}, and $r_\lambda : \cD_\Omega \to V_{\lambda}^\vee$ factors through $s_P$.  Putting this all together with Proposition \ref{prop:galois evs commute main diagram} and Lemma \ref{lem:trace} yields:

\begin{lemma}\label{lem:commutative trace}
For any classical $\lambda \in \Omega$, here is a commutative diagram
\begin{equation}\label{eq:trace diagram}
\xymatrix@C=20mm{
\hc{t}(S_{K_B},\sD_\Omega) \ar[d]^{r_{\lambda}}\ar[r]^-{\mathrm{Tr}\circ s_P} & \hc{t}(S_{K_P},\sD_\Omega^P) \ar[d]^{r_{\lambda}}\ar[r]^-{\cE_{\Omega,P,\chi}^{j,\eta_0}} & \cO_\Omega \ar[d]^{\newmod{\m_{\lambda}}}\\
\hc{t}(S_{K_B},\sV_{\lambda}^\vee) \ar[r]^-{\mathrm{Tr}} & \hc{t}(S_{K_P},\sV_{\lambda}^\vee) \ar[r]^-{\cE_{\lambda,P,\chi}^{j,\eta_0}} & L.
}
\end{equation}
\end{lemma}

\subsection{Symplectic families in the parabolic eigenvariety}

Since we are assuming (1$'$) in Conjecture \ref{conj:FJ}, by the proof of Proposition \ref{prop:local vs global} we know $(1)$ holds in Conjecture \ref{conj:local FJ} for $\pi_p$, giving $\varphi_p \in \tilde\pi_p^P$ with $\zeta_p(ut_P^\beta\cdot\varphi_p,\chi_p,s+1/2) \neq 0$ (for all $s$, by Corollary \ref{cor:non-vanishing for s}). Let $\tilde\pi$ be as given by the hypotheses of Theorem \ref{thm:p-refined FJ}. As $\pi$ has regular weight, there exists a non-vanishing Deligne-critical value $L(\pi\times\chi, j+1/2) \neq 0$ by \cite[Lem.\ 7.4]{BDW20}.

Let $\varphi_p' \in \tilde\pi_p$ be a lift of $\varphi_p$ under the trace map (via Lemma \ref{lem:trace}). Let $\varphi = \otimes_{\ell \neq p}\varphi_\ell^{\mathrm{FJ}} \otimes \varphi_p' \in \tilde\pi$. By \eqref{eq:cohomology}, attached to this is a cohomology class $\phi_{\varphi} \in \hc{t}(S_{K_B},\sV_{\lambda_\pi}^\vee)_{\tilde\pi}$. By \eqref{eq:critical value}, we have
\begin{equation}\label{eq:non-zero classical}
	\cE_{\lambda_\pi,P,\chi}^{j,\eta_0}\circ\mathrm{Tr}(\phi_\varphi)  = A_{\lambda_\pi,P,\chi}^{j} \cdot L\Big(\pi\times\chi,j+\tfrac{1}{2}\Big) \cdot \zeta_p\Big(ut_P^\beta\cdot \varphi_p, \chi_p, j+\tfrac12\Big) \neq 0,
\end{equation}
where non-vanishing is by assumption $(1')$. 

By hypothesis, the map $r_{\tilde\pi} : \hc{t}(S_{K_B},\sD_\Omega) \to \hc{t}(S_{K_B},\sV_{\lambda_\pi}^\vee)_{\tilde\pi}$ is surjective for some neighbourhood $\Omega \subset \sW_{0,\lambda_\pi}^P$ of $\lambda_\pi$. We summarise some consequences, described in detail in \cite[\S7.2,7.3]{BDW20}:
\begin{itemize}\s
	\item[(a)] By Hecke-equivariance of $r_{\lambda_\pi}$, for $h \gg 0$ the localisation of the slope $\leq h$ subspace $\hc{t}(S_{K_B},\sD_\Omega)^{\leq h}$ at $\tilde\pi$ is non-zero, giving a point $x_{\tilde\pi}$ in the top-degree eigenvariety. Let $\sC'$ be the connected component through $x_{\tilde\pi}$. 
	
	\item[(b)] Let $\Phi$ be a lift of $\phi_{\varphi}$, and $\Phi_{\sC'}$ be its projection to the direct summand of $\hc{t}(S_{K_B},\sD_\Omega)^{\leq h}$ corresponding to $\sC'$. Then $\Phi_{\sC'} \in \hc{t}(S_{K_B},\sD_\Omega)$ with $r_{\tilde\pi}(\Phi_{\sC'}) = \phi_{\varphi}$.
	
	\item[(c)] Let $\cE_\Omega \defeq \cE_{\Omega,P,\chi}^{j,\eta_0} \circ \mathrm{Tr}\circ s_P : \hc{t}(S_{K_B},\sD_\Omega) \to \cO_\Omega$, an $\cO_\Omega$-module map. By \eqref{eq:non-zero classical} and Lemma \ref{lem:commutative trace}, we have $\cE_\Omega(\Phi_{\sC'}) \neq 0 \newmod{\m_{\lambda_\pi}}$, so $\cE_\Omega(\Phi_{\sC'}) \neq 0$. As $\cO_\Omega$ is torsion-free, we deduce that $\mathrm{Ann}_{\cO_\Omega}(\Phi_{\sC'}) = 0$. As in \cite[Cor.\ 7.12]{BDW20} this forces existence of an irreducible component $\sC \subset \sC'$ of dimension $\mathrm{dim}(\Omega)$.
\end{itemize}

\begin{lemma}\label{lem:classical C}
 	Up to shrinking $\Omega$, we may take $\sC$ to be a classical cuspidal symplectic family.
\end{lemma}
\begin{proof}
Up to replacing $\Omega$ with an open neighbourhood of $\lambda_\pi$ of the same dimension, we may assume the rigid-analytic function $\cE_{\Omega}(\Phi_{\sC'}) \in \cO_\Omega$ is non-vanishing on $\Omega$. At any classical weight $\lambda \in \Omega$, combining non-vanishing of $\cE_{\Omega}(\Phi_{\sC'}) \newmod{\m_\lambda}$ with Lemma \ref{lem:commutative trace} implies $\Phi_{\sC'}$ has non-zero image in $\hc{t}(S_{K_B},\sV_\lambda^\vee)$. It must therefore have non-zero image after projection to at least one of the finite number of Hecke eigensystems that appear in $\hc{t}(S_{K_B},\sV_\lambda^\vee)$. This eigensystem thus appears in the summand of $\hc{t}(S_{K_B},\sD_\Omega)$ corresponding to $\sC'$, so gives a classical point $y_\lambda$ of $\sC'$ of weight $\lambda$.

Let $\phi_{y_\lambda}$ be the projection of $\Phi_{\sC'}$ to the $y_\lambda$-generalised eigenspace in $\hc{t}(S_{K_B},\sV_\lambda^\vee)$. By our arguments above, and the commutativity in Lemma \ref{lem:commutative trace}, we may take $y_\lambda$ so that $\cE_{\lambda,P,\chi}^{j,\eta_0}(\mathrm{Tr}(\phi_{y_\lambda})) \neq 0$. It follows that $y_\lambda$ is symplectic by \cite[Prop.\ 4.15]{BDGJW}. 
	
Now note that the classical $\lambda \in \Omega$ are very Zariski-dense, each giving rise to a classical symplectic point $y_\lambda \in \sC'$. As $\sC'$ has finitely many irreducible components, there must therefore be at least one irreducible component $\sC$ of dimension $\dim(\Omega)$ that contains a very Zariski-dense set of classical symplectic points $y_\lambda$.
	
	\medskip
	
	Finally, we must show that this is a cuspidal family. We first exhibit a related family in the \emph{parabolic} eigenvariety. Note that $\mathrm{Tr}\circ \mathrm{s}_P(\Phi_{\sC'}) \in \hc{t}(S_{K_P},\sD_\Omega^P)$. By the same argument as in (c) above, we have $\mathrm{Ann}_{\cO_\Omega}(\mathrm{Tr}\circ \mathrm{s}_P(\Phi_{\sC'})) = 0$.  But the local pieces of $P$-parabolic eigenvariety are built from the Hecke action on $\hc{t}(S_{K_P},\sD_\Omega^P)$. As in (c), this forces existence of an irreducible component $\sC^P$ of the $P$-parabolic eigenvariety of dimension $\dim(\Omega)$. Exactly as in the start of this proof, we may take $\sC^P$ to be a classical symplectic family. 
		
		By construction, this parabolic family also varies over $\Omega$; and by Lemma \ref{lem:commutative trace}, there is a bijection between classical points of $\sC$ and $\sC^P$, where every classical point $y$ of $\sC$ is a further Iwahori-refinement of a $P$-refined $\tilde\pi_y^P$ appearing in $\sC$. To show $\sC$ is cuspidal, then, it suffices to prove $\sC^P$ is a cuspidal family.
	
	By assumption, $\tilde\pi^P$ is $P$-strongly-interior and has regular weight. As in \cite[Prop.\ 5.15]{BW20}, a Zariski-dense set of classical points in $\sC^P$ are also $P$-strongly-interior, have regular weight, and are non-$P$-critical slope. As \emph{op.\ cit}., this forces them to be cuspidal, as required.
\end{proof}

This $\sC$ is the family required in Proposition \ref{prop:p-refined FJ}, completing the proof of Theorem \ref{thm:p-refined FJ}. \qed

\footnotesize
	\renewcommand{\refname}{\normalsize References}

\newcommand{\etalchar}[1]{$^{#1}$}

\Addresses
	

\begin{thebibliography}{BLGGT14}

\bibitem[AG94]{AG94}
Avner Ash and David Ginzburg.
\newblock {$p$}-adic {$L$}-functions for {${\rm GL}(2n)$}.
\newblock {\em Invent. Math.}, 116(1-3):27--73, 1994.

\bibitem[APS08]{APS08}
Avner Ash, David Pollack, and Glenn Stevens.
\newblock Rigidity of {$p$}-adic cohomology classes of congruence subgroups of
  {${\rm GL}(n,\Bbb Z)$}.
\newblock {\em Proc. Lond. Math. Soc. (3)}, 96(2):367--388, 2008.

\bibitem[AS06]{AS06}
Mahdi Asgari and Freydoon Shahidi.
\newblock Generic transfer for general spin groups.
\newblock {\em Duke Math. J.}, 132(1):137--190, 2006.

\bibitem[AS14]{AS14}
Mahdi Asgari and Freydoon Shahidi.
\newblock Image of functoriality for general spin groups.
\newblock {\em Manuscripta Math.}, 144(3-4):609--638, 2014.

\bibitem[Asg02]{Asg02}
Mahdi Asgari.
\newblock Local {$L$}-functions for split spinor groups.
\newblock {\em Canad. J. Math.}, 54(4):673--693, 2002.

\bibitem[BDG{\etalchar{+}}]{BDGJW}
Daniel {Barrera Salazar}, Mladen Dimitrov, Andrew Graham, Andrei Jorza, and
  Chris Williams.
\newblock On the {GL}(2n) eigenvariety: branching laws, {S}halika families and
  $p$-adic ${L}$-functions.
\newblock Preprint: \url{https://arxiv.org/abs/2211.08126}.

\bibitem[BDW]{BDW20}
Daniel {Barrera Salazar}, Mladen Dimitrov, and Chris Williams.
\newblock On $p$-adic {$L$}-functions for $\mathrm{{GL}}(2n)$ in finite slope
  {S}halika families.
\newblock Preprint: \url{https://arxiv.org/abs/2103.10907}.

\bibitem[Bel21]{Eigenbook}
Jo\"{e}l Bella\"{\i}che.
\newblock {\em The eigenbook---eigenvarieties, families of {G}alois
  representations, {$p$}-adic {$L$}-functions}.
\newblock Pathways in Mathematics. Birkh\"{a}user/Springer, Cham, 2021.

\bibitem[BLGGT14]{BGGT14a}
Thomas Barnet-Lamb, Toby Gee, David Geraghty, and Richard Taylor.
\newblock Potential automorphy and change of weight.
\newblock {\em Ann. of Math. (2)}, 179(2):501--609, 2014.

\bibitem[BW21]{BW20}
Daniel {Barrera Salazar} and Chris Williams.
\newblock Parabolic eigenvarieties via overconvergent cohomology.
\newblock {\em Math. Z.}, 299(1-2):961--995, 2021.

\bibitem[BZ77]{BZ77}
I.~N. Bernstein and A.~V. Zelevinsky.
\newblock Induced representations of reductive {${p}$}-adic groups. {I}.
\newblock {\em Ann. Sci. \'{E}cole Norm. Sup. (4)}, 10(4):441--472, 1977.

\bibitem[Cas80]{Cas80}
W.~Casselman.
\newblock The unramified principal series of {$\mathfrak{p}$}-adic groups. {I}.
  {T}he spherical function.
\newblock {\em Compositio Math.}, 40(3):387--406, 1980.

\bibitem[Che04]{Che04}
Ga\"{e}tan Chenevier.
\newblock Familles {$p$}-adiques de formes automorphes pour {${\rm GL}_n$}.
\newblock {\em J. Reine Angew. Math.}, 570:143--217, 2004.

\bibitem[Che05]{Che05}
Ga\"{e}tan Chenevier.
\newblock Une correspondance de {J}acquet-{L}anglands {$p$}-adique.
\newblock {\em Duke Math. J.}, 126(1):161--194, 2005.

\bibitem[Clo90]{Clo90}
Laurent Clozel.
\newblock Motifs et formes automorphes: applications du principe de
  fonctorialit\'{e}.
\newblock In {\em Automorphic forms, {S}himura varieties, and {$L$}-functions,
  {V}ol. {I} ({A}nn {A}rbor, {MI}, 1988)}, volume~10 of {\em Perspect. Math.},
  pages 77--159. Academic Press, Boston, MA, 1990.

\bibitem[CM09]{CM09}
Frank Calegari and Barry Mazur.
\newblock Nearly ordinary {G}alois deformations over arbitrary number fields.
\newblock {\em J. Inst. Math. Jussieu}, 8(1):99--177, 2009.

\bibitem[Coa89]{coates89}
John Coates.
\newblock On {$p$}-adic {$L$}-functions attached to motives over {${\bf Q}$}.
  {II}.
\newblock {\em Bol. Soc. Brasil. Mat. (N.S.)}, 20(1):101--112, 1989.

\bibitem[Con99]{Con99}
Brian Conrad.
\newblock Irreducible components of rigid spaces.
\newblock {\em Ann. Inst. Fourier (Grenoble)}, 49(2):473--541, 1999.

\bibitem[DJR20]{DJR18}
Mladen Dimitrov, Fabian Januszewski, and A.~Raghuram.
\newblock ${L}$-functions of {$\mathrm{GL}(2n)$}: {$p$}-adic properties and
  nonvanishing of twists.
\newblock {\em Compositio Math.}, 156(12):2437--2468, 2020.

\bibitem[FJ93]{FJ93}
Solomon Friedberg and Herv\'{e} Jacquet.
\newblock Linear periods.
\newblock {\em J. Reine Angew. Math.}, 443:91--139, 1993.

\bibitem[Geh18]{Geh18}
Lennart Gehrmann.
\newblock On {S}halika models and {$p$}-adic {$L$}-functions.
\newblock {\em Israel J. Math.}, 226(1):237--294, 2018.

\bibitem[GM98]{GM98}
Fernando~Q. Gouv\^{e}a and Barry Mazur.
\newblock On the density of modular representations.
\newblock In {\em Computational perspectives on number theory ({C}hicago, {IL},
  1995)}, volume~7 of {\em AMS/IP Stud. Adv. Math.}, pages 127--142. Amer.
  Math. Soc., Providence, RI, 1998.

\bibitem[GR13]{GR13}
Wee~Teck Gan and A.~Raghuram.
\newblock Arithmeticity for periods of automorphic forms.
\newblock In {\em Automorphic representations and {$L$}-functions}, volume~22
  of {\em Tata Inst. Fundam. Res. Stud. Math.}, pages 187--229. Tata Inst.
  Fund. Res., Mumbai, 2013.

\bibitem[GR14]{GR2}
Harald Grobner and A.~Raghuram.
\newblock On the arithmetic of {S}halika models and the critical values of
  {$L$}-functions for {${\rm GL}_{2n}$}.
\newblock {\em Amer. J. Math.}, 136(3):675--728, 2014.
\newblock With an appendix by Wee Teck Gan.

\bibitem[Han17]{Han17}
David Hansen.
\newblock Universal eigenvarieties, trianguline {G}alois representations and
  $p$-adic {L}anglands functoriality.
\newblock {\em J. Reine. Angew. Math.}, 730:1--64, 2017.

\bibitem[Hid98]{HidP-ord}
Haruzo Hida.
\newblock Automorphic induction and {L}eopoldt type conjectures for {${\rm
  GL}(n)$}.
\newblock volume~2, pages 667--710. 1998.
\newblock Mikio Sato: a great Japanese mathematician of the twentieth century.

\bibitem[HL11]{HL11}
Richard Hill and David Loeffler.
\newblock Emerton's {J}acquet functors for non-{B}orel parabolic subgroups.
\newblock {\em Doc. Math.}, 16:1--31, 2011.

\bibitem[HS]{HS-fern}
Valentin Hernandez and Benjamin Schraen.
\newblock The infinite fern in higher dimensions.
\newblock Preprint: \url{https://arxiv.org/abs/2210.10564}.

\bibitem[HS16]{HS16}
Joseph Hundley and Eitan Sayag.
\newblock Descent construction for {GS}pin groups.
\newblock {\em Mem. Amer. Math. Soc.}, 243(1148):v+124, 2016.

\bibitem[Hum90]{Hum90}
James~E. Humphreys.
\newblock {\em Reflection groups and {C}oxeter groups}, volume~29 of {\em
  Cambridge Studies in Advanced Mathematics}.
\newblock Cambridge University Press, Cambridge, 1990.

\bibitem[JLR99]{JLR01}
Herv\'{e} Jacquet, Erez Lapid, and Jonathan Rogawski.
\newblock Periods of automorphic forms.
\newblock {\em J. Amer. Math. Soc.}, 12(1):173--240, 1999.

\bibitem[JN19]{JoNew}
Christian Johansson and James Newton.
\newblock Irreducible components of extended eigenvarieties and interpolating
  {L}anglands functoriality.
\newblock {\em Math. Res. Lett.}, 26(1):159--201, 2019.

\bibitem[JPSS81]{JPSS}
Herv\'{e} Jacquet, Ilja Piatetski-Shapiro, and Joseph Shalika.
\newblock Conducteur des repr\'{e}sentations g\'{e}n\'{e}riques du groupe
  lin\'{e}aire.
\newblock {\em C. R. Acad. Sci. Paris S\'{e}r. I Math.}, 292(13):611--616,
  1981.

\bibitem[MSD74]{MSD74}
Barry Mazur and Peter Swinnerton-Dyer.
\newblock Arithmetic of {W}eil curves.
\newblock {\em Invent. Math.}, 25:1 -- 61, 1974.

\bibitem[OST23]{OST19}
Masao Oi, Ryotaro Sakamoto, and Hiroyoshi Tamori.
\newblock Iwahori-{H}ecke algebra and unramified local {$L$}-functions.
\newblock {\em Math. Z.}, 303(3):Paper No. 59, 42, 2023.

\bibitem[Pan94]{Pan94}
Alexei~A. Panchishkin.
\newblock Motives over totally real fields and {$p$}-adic {$L$}-functions.
\newblock {\em Ann. Inst. Fourier (Grenoble)}, 44(4):989--1023, 1994.

\bibitem[Roc23]{Roc20}
Rob Rockwood.
\newblock Plus/minus {$p$}-adic {$L$}-functions for {${\rm GL}_{2n}$}.
\newblock {\em Ann. Math. Qu\'{e}.}, 47(1):177--193, 2023.

\bibitem[Sun19]{Sun19}
Binyong Sun.
\newblock Cohomologically induced distinguished representations and
  cohomological test vectors.
\newblock {\em Duke Math. J.}, 168(1):85--126, 2019.

\bibitem[SV17]{SV17}
Yiannis Sakellaridis and Akshay Venkatesh.
\newblock Periods and harmonic analysis on spherical varieties.
\newblock {\em Ast\'{e}risque}, (396):viii+360, 2017.

\bibitem[Til96]{Til96}
Jacques Tilouine.
\newblock {\em Deformations of {G}alois representations and {H}ecke algebras}.
\newblock The Mehta Research Institute of Mathematics and Mathematical Physics,
  Allahabad; by Narosa Publishing House, New Delhi, 1996.

\bibitem[Urb11]{Urb11}
Eric Urban.
\newblock Eigenvarieties for reductive groups.
\newblock {\em Ann. of Math. (2)}, 174(3):1685--1784, 2011.

\bibitem[Wil]{Wil-2n}
Chris Williams.
\newblock {On $p$-adic $L$-functions for symplectic representations of
  $\mathrm{GL}_{N}$ over number fields}.
\newblock Preprint: \url{https://arxiv.org/abs/2305.07809}.

\bibitem[Xia18]{Xia18}
Zhengyu Xiang.
\newblock Twisted eigenvarieties and self-dual representations.
\newblock {\em Ann. Inst. Fourier (Grenoble)}, 68(6):2381--2444, 2018.

\end{thebibliography}
\end{document}